%% file: SLPAM_main.tex
\pdfoutput=1
\documentclass[reqno
]{amsart} 

\usepackage[latin1]{inputenc}
\usepackage{subfigure}

\usepackage{amssymb} 
\usepackage{amsthm}
\usepackage{upref} 
\usepackage{graphicx}
\usepackage{color}
\usepackage{url}
\usepackage{paralist} 

\usepackage[british]{babel}
 
\usepackage{calc}
\usepackage{bm} 

\usepackage{wasysym} 

\usepackage[all]{xy}

\usepackage{hyperref}

\newtheorem{Theorem}{Theorem}[section]
\newtheorem{Corollary}[Theorem]{Corollary}
\newtheorem{Proposition}[Theorem]{Proposition}
\newtheorem{Lemma}[Theorem]{Lemma}

\theoremstyle{definition}
\newtheorem{Definition}[Theorem]{Definition}

\newtheorem{Example}[Theorem]{Example}

\theoremstyle{remark}
\newtheorem{Remark}[Theorem]{Remark}

\DeclareMathOperator{\DM}{DM}

\DeclareMathOperator{\vol}{{{vol}}}
\newcommand{\abs}[1]{\left|#1\right|}
\DeclareMathOperator{\rank}{rk}
\DeclareMathOperator{\interior}{int}
\DeclareMathOperator{\todd}{Todd}
\DeclareMathOperator{\toddper}{{T}\widetilde{{od}}{d}}
\DeclareMathOperator{\toddperbox}{{T}\widetilde{{od}}{d}^{box}}

\newcommand{\psiper}{ \widetilde{\psi} }
\newcommand{\tors}{t} 

\DeclareMathOperator{\res}{Res}
\DeclareMathOperator{\var}{Var} 

\newcommand{\st}{s.\,t.\ } 
\newcommand{\ie}{\textit{i.\,e.\ }} 
\newcommand{\eg}{\textit{e.\,g.\ }} 

\newcommand{\vpf}{i} 
\newcommand{\sg}[1]{\left\langle #1 \right\rangle} 

\newcommand{\N}{\mathbb{N}}
\newcommand{\Z}{\mathbb{Z}}

\newcommand{\R}{\mathbb{R}}
\newcommand{\CC}{\mathbb{C}}

\newcommand{\Bcal}{\mathcal{B}}
\newcommand{\Ccal}{\mathcal{C}}
\newcommand{\Dcal}{\mathcal{D}}
\newcommand{\Hcal}{\mathcal{H}}

\newcommand{\Zcal}{\mathcal{Z}}
\newcommand{\Pcal}{\mathcal{P}}

\newcommand{\Jcal}{\mathcal{J}}
\newcommand{\Vcal}{\mathcal{V}}

\newcommand{\eps}{\varepsilon}

\newcommand{\Bper}{ 
\widetilde{\mathcal{B}}} 

\newcommand{\dJ}{\mathop{ \Jcal^\nabla\! } } 
\newcommand{\dJC}{\mathop{ \Jcal^\nabla_\CC\! } } 
\newcommand{\cJ}{\mathop{ \Jcal^\partial\! } } 
\newcommand{\cJC}{\mathop{ \Jcal^\partial_\CC\! } } 

\newcommand{\Pper}{ 
\widetilde{\mathcal{P}}} 

\newcommand{\Dpw}{ D_\mathrm{pw} } 

\DeclareMathOperator{\ideal}{{{ideal}}}
\DeclareMathOperator{\spa}{{{span}}}
\DeclareMathOperator{\cone}{{{cone}}}
\DeclareMathOperator{\BB}{ \mathbb B}

\newcommand{\cfrak}{\mathfrak c}

\newcommand{\tutte}{{\mathfrak T}} 
\newcommand{\aritutte}{{\mathfrak M}} 
\newcommand{\multari}{{\mathfrak m}} 

\newcommand{\diff}[1]{\frac{\partial}{\partial #1}}

\DeclareMathOperator{\hilb}{{{Hilb}}}
\DeclareMathOperator{\clos}{{{cl}}} 
\DeclareMathOperator{\sym}{Sym} 
\newcommand{\pair}[2]{\langle #1,#2 \rangle}
\newcommand{\discpair}[2]{ {\langle #1,#2 \rangle_\nabla}} 
\newcommand{\discpairP}[2]{ {\langle #1,#2 \rangle_{\Pper}}} 

\newcommand{\Xintro}{%
Let $X\subseteq U\cong \R^d$ be a finite list of vectors that spans $U$\!.
}

\newcommand{\XintroList}{%
Let $X = (x_1,\ldots, x_N) \subseteq U\cong \R^d$ be a list of vectors that spans $U$\!.
}

\newcommand{\XintroN}{%
Let $X\subseteq U\cong \R^d$ be a  list of $N$ vectors that spans $U$\!.
}

\newcommand{\XintroLattice}{%
Let $X\subseteq   \Lambda\subseteq U\cong \R^d$ be a finite list of vectors that spans $U$\!.
}

\newcommand{\XintroLatticeUnimod}{%
Let $X\subseteq   \Lambda\subseteq U\cong \R^d$ be a finite list of vectors that is unimodular and spans $U$\!.
}

\newcommand{\XintroAbelianGroup}{%
Let $G$ be a finitely generated abelian group and let
$X$ be a finite list of elements of $G$ that generates a subgroup of finite index.
}

\newcommand{\XintroAbelianGroupN}{%
Let $G$ be a finitely generated abelian group and let
$X$ be a  list of $N$ elements of $G$ that generates a subgroup of finite index.
}

\email{lenz@maths.ox.ac.uk}
\author{Matthias Lenz
}
\thanks{The author was supported by a Junior Research Fellowship
 of Merton College (University of Oxford).
}

\title
{%
Splines, lattice points, and arithmetic matroids
}

\date{\today}
\address{%
Mathematical Institute, University of Oxford, 
Andrew Wiles Building\\
Woodstock Road\\
Oxford 
OX2 6GG\\
United Kingdom
}
\keywords{lattice polytope, vector partition function, Todd operator, Brion--Vergne formula,  
 arithmetic matroid, zonotopal algebra}

\subjclass[2010]{Primary: 
05B35, 
19L10, 
52B20; 
Secondary:
13B25,  
14M25, 
16S32,  
41A15, 
47F05, 
52B40, 
52C35
}

\begin{document}

\begin{abstract}
Let $X$ be a $(d\times N)$-matrix.
 We consider the variable polytope  
 $\Pi_X(u) = \{ w \ge 0 : X w = u \}$. It is known that the function
 $T_X$ that assigns to a parameter $u \in \R^d$
 the volume of the polytope $\Pi_X(u)$ is piecewise polynomial.
 The  
 Brion--Vergne formula implies that the number of
 lattice points in $\Pi_X(u)$ can be obtained by applying a certain 
 differential operator to the function $T_X$.
 In this article we slightly improve 
 the Brion--Vergne formula and we  study   
 two spaces of differential operators that arise in this context: 
 the space of relevant differential operators (\ie operators that do not annihilate $T_X$)
  and the space of nice differential operators (\ie operators that leave $T_X$ continuous).
 These two spaces are finite-dimensional homogeneous vector spaces and their
 Hilbert series are evaluations of the Tutte polynomial of the arithmetic matroid defined by the matrix $X$.
They are closely related to the $\Pcal$-spaces studied by 
 Ardila--Postnikov and Holtz--Ron in the context of zonotopal algebra and power ideals.
\end{abstract}

\maketitle

\section{Introduction}

The problem of determining the number of integer points in a convex polytope  appears in many areas of mathematics
including 
commutative algebra, combinatorics, representation theory, statistics, and combinatorial optimisation (see \cite{deloera-2005} for a survey).
The number of integer points in a polytope can be seen as a discrete version of its volume.
In this article we will study the relationship between these two quantities using the language of 
vector partition functions and multivariate splines. We will also study related combinatorial and algebraic structures.

Let $X\subseteq \Z^d$ be a finite list of vectors that all lie on the same side of some hyperplane.
For $u\in \R^d$, we consider the variable polytope
$\Pi_X(u) = \{ w \in \R^N_{\ge 0} : X w = u \}$. The multivariate spline (or truncated power) $T_X : \R^d\to \R$ measures the volume
of these polytopes, whereas the vector partition function $\vpf_X :\Z^d \to \Z$ counts the number of  integral points they contain.
 These two functions have been studied by many authors.
 The combinatorial and algebraic aspects are stressed in the book  
 by De Concini and Procesi~\cite{concini-procesi-book}.
A standard reference from the 
approximation theory point of view is the book by
 de~Boor, H\"ollig, and Riemenschneider~\cite{BoxSplineBook}. Another good reference is Vergne's survey article on integral
points in polytopes~\cite{vergne-2003}.

 Khovaniskii and Pukhlikov proved a remarkable formula that relates the volume and the number of integer points in the polytope $\Pi_X(u)$ in the 
 case where the list $X$ is unimodular, \ie every basis for $\R^d$ that can be selected from $X$ has determinant $\pm 1$ \cite{pukhlikov-khovanski-1992}.
  The connection is made via Todd operators, 
  \ie differential operators  of type $\frac{\partial_x}{1 - e^{\partial_x}}$.
 The formula is closely related to the Hirzebruch--Riemann--Roch Theorem for smooth projective toric varieties %
(see \cite[Chapter 13]{cox-little-schenck-2011}).   
Brion and Vergne have extended the Khovaniskii--Pukhlikov formula to arbitrary rational polytopes \cite{brion-vergne-1997}.

Starting with the work of de~Boor--H\"ollig \cite{boor-hoellig-1982} and Dahmen--Micchelli \cite{dahmen-micchelli-1985,dahmen-micchelli-1985b}
in the 1980s, various authors have studied $\Dcal$-spaces, \ie  vector spaces of multivariate polynomials spanned by the local pieces of these splines and 
various other related spaces. This includes spaces of differential operators that act on the splines, the so-called $\Pcal$-spaces. 
Recently,  Holtz and Ron have developed a theory of \emph{zonotopal algebra}  
that describes the relationship between some of these spaces and various combinatorial structures including the matroid and the zonotope defined by the list~$X$~\cite{holtz-ron-2011} .
Ardila--Postnikov have studied $\Pcal$-spaces in the context of power ideals \cite{ardila-postnikov-2009}.
Related work has also appeared in the literature on hyperplane arrangements, see \eg \cite{berget-2010,orlik-terao-1994}.
Recent work of 
De~Concini--Procesi--Vergne \cite{deconcini-procesi-vergne-2010b, deconcini-procesi-vergne-2011, deconcini-procesi-vergne-infinitesimal-2013}
and Cavazzani--Moci \cite{cavazzani-moci-2013} shows that some of these spaces can be 
``geometrically realised'' as   equivariant cohomology or  $K$-theory of certain differentiable manifolds.

In a previous article, the author has identified the space of differential operators
with constant coefficients that  leave the spline $T_X$ continuous in the case where the list $X$ is unimodular and used this to slightly improve the Khovanskii--Pukhlikov formula~\cite{lenz-todd-online-2014}. 

 The goal of this paper is twofold. 
 Firstly, we will generalise the results in \cite{lenz-todd-online-2014} to the case where the list $X$ is no longer required to be unimodular.
 We will obtain a slight generalisation of the Brion--Vergne formula and we will identify two types of periodic $\Pcal$-spaces, \ie spaces of differential operators
 with periodic coefficients that appear naturally in this context.

 Secondly, we will study combinatorial properties of these spaces in the spirit of  zonotopal algebra. It will turn 
 out that these spaces are strongly related to  \emph{arithmetic matroids} 
 that were recently discovered by D'Adderio--Moci~\cite{moci-adderio-2013}.
 
 An extended abstract of this paper has appeared in the proceedings of the conference
 \mbox{FPSAC~2014} \cite{lenz-fpsac-2014}.

\subsection*{Organisation of the article.}
In the following paragraphs,
 some known results will be labelled by 
an $r$ and a natural number. The  generalisations of these statements that will be proven in this paper 
will be labelled by an $R$ and the same natural number.
The remainder of this article is organised as follows:
\begin{asparaitem}
\item
in Section~\ref{Section:Notation} we will introduce our notation and 
review  some facts about splines and vector partition functions. This includes the definition of the 
 Dahmen--Micchelli spaces $\Dcal(X)$ and $\DM(X)$ that are spanned by the local pieces of  splines and vector partition functions, respectively.
 We will also recall the definitions of the  spaces $\Pcal(X)$ and $\Pcal_-(X)$ that act on the splines as partial differential operators with 
 constant coefficients and we will recall that their Hilbert series are evaluations of the Tutte polynomial of the %
 matroid defined by $X$ (r1). We will also recall the definition of a pairing under which $\Dcal(X)$ and $\Pcal(X)$ are dual vector spaces (r2).
\item
In Section~\ref{Section:TUMresults} we will review some results from 
\cite{lenz-interpolation-online-2013,lenz-todd-online-2014}, where the author 
has studied the relationship between the Khovanskii--Pukhlikov formula and the spaces $\Pcal_-(X)$ and $\Pcal(X)$ 
in the case where the list $X$ is unimodular. In this case, one can replace the (complicated)
 Todd operator that appears in the Khovanskii--Pukhlikov formula by a (simpler) element of $\Pcal(X)$ (r3). 
 The space $\Pcal_-(X)$ can be characterised as the space of differential operators the leave the spline continuous (r4).
The section ends with an outlook on how we will generalise these results in this paper.
\item 
In Section~\ref{Section:MoreBackground}  we will recall the definitions of generalised toric arrangements, arithmetic matroids, and their Tutte polynomials.
\item 
In Section~\ref{Section:GeneralCase}  we will prove a refined Brion--Vergne formula (R3).
\item In Section~\ref{Section:ArithmeticMatroids}
we will introduce the internal periodic $\Pcal$-space $\Pper_-(X)$ and the central periodic $\Pcal$-space $\Pper(X)$ 
and prove some results about these spaces. We will construct various bases for these spaces and state that 
their Hilbert series is an evaluation of the arithmetic Tutte polynomial defined by the list $X$ (R1).
\item
In Section~\ref{Section:PeriodicDuality} we will define a pairing between the spaces $\Pper(X)$ and $\DM(X)$ under which they are dual vector spaces (R2).
\item In Section~\ref{Section:WallCrossingProofs} we will discuss a wall-crossing formula for splines due to Boysal--Vergne \cite{boysal-vergne-2009}
 and use it
to prove that the space $\Pper_-(X)$ can be characterised as the space of differential operators with periodic coefficients that leave the spline continuous (R4).
\item In Section~\ref{Section:DeletionContraction} we will define deletion and contraction for the periodic $\Pcal$-spaces and we will use this technique
 to  prove  that the  Hilbert series of the internal space is an evaluation of the arithmetic Tutte polynomial (part of R1).
\item
 Section~\ref{Section:Examples} contains some more complicated examples. Shorter examples are interspersed throughout the text.
\end{asparaitem}

\subsection*{Acknowledgements}
The author would like to thank Lars Kastner and Zhiqiang Xu 
for helpful conversations.

\input{SLPAM_background}

\input{SLPAM_unimodular}

\input{SLPAM_arithmetic}

\input{SLPAM_duality}

\input{SLPAM_proofs.tex}

\input{SLPAM_examples}

\input{SLPAM_sageappendix}

\renewcommand{\MR}[1]{} 

\bibliographystyle{amsplain}
\bibliography{../../MasonsConjecture/Mason_Literatur}
%
%



\end{document}

%% file: SLPAM_background.tex
\section{Preliminaries}
\label{Section:Notation}
In this section we will introduce our notation and review some facts about splines, vector partition functions, and related algebraic structures.
 The notation is similar to the one used in \cite{concini-procesi-book}.
  We fix a $d$-dimensional real vector space $U$ and a lattice
  $\Lambda\subseteq U$.
  Let $X=(x_1,\ldots, x_N) \subseteq \Lambda$ be a finite list of
  vectors that spans $U$. 
 The list $X$ is called \emph{unimodular} with respect to $\Lambda$ if and only if 
  every basis for $U$ that can be selected from $X$ is also
 a lattice basis for $\Lambda$.
Note that $X$ can be identified with a linear map $X : \R^N \to U$. 
Let $u\in U$. We define the variable polytopes
\begin{align}
 \Pi_X(u) &:=  \{ w \in \R^N_{\ge 0} :   X w = u  \} 
 \quad \text{ and } \quad  \Pi^1_X(u) :=  \Pi_X(u) \cap [0,1]^N.
\end{align}
Note that every convex polytope can be written in the form  $\Pi_X(u)$ for suitable $X$ and $u$.
 The dimension of these two polytopes is at most $N-d$. Now we define functions 
 $\vpf_X : \Lambda\to\Z_{\ge 0}$ and $B_X,T_X : U \to \R_{\ge 0}$, namely the
\begin{align}
\text{\emph{vector partition function} }
\vpf_X(u) &:= \abs{\Pi_X(u) \cap \Z^N}, \displaybreak[2] \\
\text{ the \emph{box spline} } B_X(u) &:= \det(XX^T)^{-1/2}\vol\nolimits_{N-d}{\Pi^1_X(u)}, \displaybreak[2]  \\
\text{
 and the \emph{multivariate spline} }
 T_X(u) &:= \det(XX^T)^{-1/2}\vol\nolimits_{N-d}{\Pi_X(u)}.
 \end{align}
 Note that we have to assume that $0$ is not contained in the convex hull of $X$ in order for $T_X$ and $\vpf_X$ to be well-defined.
 Otherwise,  $\Pi_X(u)$ may be unbounded.
 It makes sense to define $\vpf_X$ only on $\Lambda$ as $\Pi_X(u)\cap \Z^N = \emptyset$ for $ u \not \in \Lambda $.

   The \emph{zonotope} $Z(X)$ and the \emph{cone} $\cone(X)$ are  defined as
 \begin{equation}
 Z(X) := \left\{ \sum_{i=1}^N \lambda_i x_i : 0\le \lambda_i \le 1  \right\}  
 \quad\text{ and } \quad \cone(X) := \left\{ \sum_{i=1}^N \lambda_i x_i :   \lambda_i  \ge 0  \right\}. 
 \end{equation}
We denote the set of interior lattice points of $Z(X)$ by $\Zcal_-(X) := \interior(Z(X)) \cap \Lambda$. 
Here are the first three examples.
\begin{Example}
\label{Example:twoOnes}
 Let $X=(1,1)$. Then $T_X(u)=u$ for $u\ge 0$, $\vpf_X(u)=u+1$ for $u\in \Z_{\ge 0}$ and $B_X$ is the piecewise linear  function with
  maximum $B_X(1)=1$  whose support is the zonotope $Z(X)=[0,2]$ and that is smooth on $\R\setminus \{0,1,2\}$. 
\end{Example}
\begin{Example}
\label{Example:onetwo}
 Let $X = (1,2)$. Then $T_X(u)=\frac u2$ for $u\ge 0$, $\vpf_X(u) = \frac u2 + \frac 34 + (-1)^u \frac 14$ for $u\in \Z_{\ge 0}$ and $B_X$
 is the piecewise linear function with
  $B_X(1)= B_X(2)=\frac 12$  whose support is the zonotope $Z(X)=[0,3]$  and that is smooth on $\R\setminus \{0,1,2,3\}$. 
 \end{Example}
 \begin{Example}[Zwart--Powell]
 \label{Example:ZPelementA}
  We consider the matrix 
  $X=\begin{pmatrix}
                             1 & 0 & 1 & -1 \\ 0 & 1 & 1 & 1
                            \end{pmatrix}$.
The corresponding box spline is known in the literature as the
 Zwart--Powell element.  Its support is the zonotope $Z(X)$.
The functions $T_X$ and $\vpf_X$ agree with  certain non-zero (quasi-)polynomials
on three different polyhedral cones.
The three cones and the corresponding (quasi)-polynomials are depicted in 
Figure~\ref{Figure:BoxSplineValuesC}. 
\begin{figure}[tbp]
  \begin{center}
   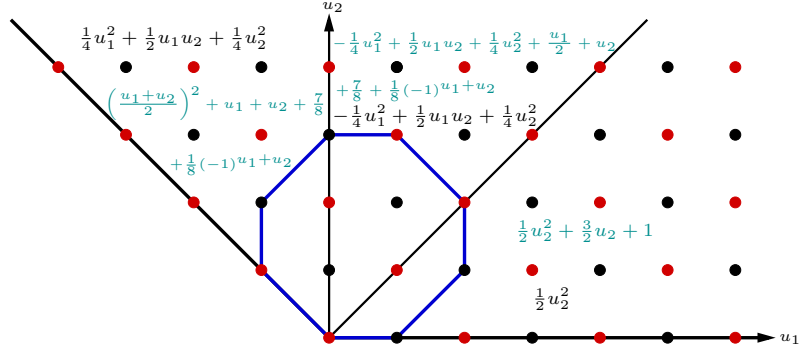
  \end{center}
  \caption{The multivariate spline (black) and the vector partition function (cyan) 
  corresponding to the Zwart--Powell element (cf.~Example~\ref{Example:ZPelementA}).
They have three different
non-zero  local pieces each. The zonotope $Z(X)$ is shown as well.
}
 \label{Figure:BoxSplineValuesC}
\end{figure} 

\end{Example}

\subsection{Commutative algebra}

 The symmetric algebra over $U$ is denoted by $\sym(U)$.
  We fix a basis $s_1,\ldots, s_d$ for the lattice $\Lambda$. This makes it possible to identify $\Lambda$ with $\Z^d$,
  $U$ with $\R^d$, $\sym(U)$ with the polynomial ring
   $\R[s_1,\ldots, s_d]$, and $X$ with a $(d\times N)$-matrix. Then
   $X$ is unimodular if and only if every non-singular   $(d\times d)$-submatrix of this matrix has determinant $1$ or $-1$.
 The base-free setup is more convenient when working with quotients of
 vector spaces.
 
 We denote the dual vector space by $V = U^*$ and we fix a basis $t_1,\ldots, t_d$ that is dual to the basis for 
  $U$. An element of $\sym(U)$ can be seen as a differential operator on $\sym(V)$,  \ie
   $\sym(U) \cong \R[s_1,\ldots, s_d] \cong \R[\diff{t_1},\ldots, \diff{t_d}]$.
   For $f\in \sym(U)$ and $p \in \sym(V)$ we  write $f(D) p$ to denote the polynomial 
   in $\sym(V)$ that is obtained when $f$ acts on $p$ as a differential operator.
  It is known that the two spline functions are piecewise polynomial and that their local pieces are contained in $\sym(V)$. 
  We will mostly use elements of $\sym(U)$ as differential operators 
   on these local pieces.
   Sometimes we will consider the complexified spaces $U_\CC:= U\otimes \CC$, $V_\CC:=V\otimes \CC$,
    $\sym(U_\CC)\cong \CC[s_1,\ldots, s_d]$, and $\sym(V_\CC):=\CC[t_1,\ldots, t_d]$.

 \smallskip
 Note that the group ring of $\Lambda$ over a ring $R$ is isomorphic to the ring of Laurent polynomials in $d$ variables over $R$. In particular
 $\Z[\Lambda]\cong \Z[a_1^{\pm 1},\ldots, a_d^{\pm 1}]$ and $\CC[\Lambda]\cong \CC[a_1^{\pm 1},\ldots, a_d^{\pm 1}]$.
 We will write $\Ccal_R[\Lambda]$ to denote the set of all functions $f: \Lambda\to R$. In particular, we will use the sets
 $\Ccal_\Z[\Lambda]$  and $\Ccal_\CC[\Lambda]$.
The lattice $\Lambda$ 
acts on $\Ccal_\Z[\Lambda]$ and $\Ccal_\CC[\Lambda]$ via translations. For $\lambda\in \Lambda$ we define the translation operator $\tau_\lambda$ by
 $\tau_\lambda f := f(\cdot - \lambda)$. 
   This extends to an action of $\Z[\Lambda]$ on $\Ccal_\Z[\Lambda]$ and of $\CC[\Lambda]$ on $\Ccal_\CC[\Lambda]$.
 We define the \emph{difference operator} 
 $ \nabla_\lambda := 1 - \tau_\lambda$ and for $Y\subseteq X$, $\nabla_Y:=\prod_{\lambda\in Y} \nabla_\lambda$.
 
 Let $x\in \Lambda$ and $f\in \sym(U_\CC)$. Then $f|_\Lambda \in \Ccal_\CC[\Lambda]$.
 Note that $\nabla_x (f|_\Lambda)$ is a discrete analogue of $\diff{x} f$. The relationship between difference and differential operators
 will play an important role in this paper.

 \subsection{Piecewise (quasi-)polynomial functions}
\label{Subsection:quasipolynomials}
 
 In this subsection we will review some facts about piecewise polynomial and piecewise quasipolynomial functions. 
 The definitions here follow  \cite{concini-procesi-book} and \cite{deconcini-procesi-vergne-2013}.

  A hyperplane in $U$  that is spanned by a sublist  $Y \subseteq X$ is called an \emph{admissible hyperplane}. A shift
  of such a hyperplane by a vector  $\lambda \in \Lambda$ is called an \emph{affine admissible hyperplane}.
  An \emph{alcove} $\cfrak\subseteq U$ is a connected component of the complement of the union of all affine admissible hyperplanes.
 A vector $w\in U$ is called \emph{affine singular} if it is contained in any affine admissible hyperplane.
A vector $w\in U$ is called \emph{affine regular} if it is not affine singular. %
  We call $w$ \emph{short affine regular} if it is so short that  it is contained in an alcove whose closure contains the origin. 
 A point $p\in U$ is called \emph{strongly regular} if $p$ is not contained in any $\cone(Y)$ where $Y\subseteq X$ and $Y$ spans a subspace of 
  dimension at most $d-1$.
  A connected component of the set of strongly regular points is called a \emph{big cell}.%

 In Example~\ref{Example:twoOnes} the alcoves are the open intervals $(j,j+1)$ for $j\in \Z$. In 
 Example~\ref{Example:ZPelementA}  there are four big cells, three of them are convex cones that are contained in the support of $T_X$.

 For a set $A\subseteq U$, we denote the topological \emph{closure} of $A$ in the standard topology by $\clos(A)$.

  A function defined on the affine regular points (resp.\ strongly regular points) is called \emph{piecewise polynomial} with respect to the alcoves
  (resp.\ with respect to the big cells)
  if for each alcove (resp.\ big cell) $\cfrak$, the restriction $f|_\cfrak$ 
  coincides with a polynomial.%
   
  Note that a function which is piecewise polynomial with respect to the big cells is automatically 
  piecewise polynomial with respect to the alcoves since the closure of each big cell is the union of
  countably many closures of alcoves.

 A function on a lattice $\Lambda$ is called a \emph{quasipolynomial} (or periodic polynomial) 
 if there exists a sublattice $\Lambda^0$ \st $f$ restricted to each coset of $\Lambda^0$ is
 (the restriction of) a polynomial.
 A quasipolynomial on the vector space $U$ can be written as a linear combination   of exponential polynomials, \ie
 functions of type $e^{\phi(u)}g(u)$, where $g\in \sym(V)$ and $\phi\in V=U^*$ is rational, \ie $\phi(u)\in \mathbb{Q}$ for all $u\in \Lambda$. 
 
 A function $f: \Lambda \to \CC$ is called \emph{piecewise quasipolynomial} %
with respect to the alcoves (resp.\ with respect to the big cells) if  for each alcove (resp.\ big cell)
the restriction $f|_{\clos(\cfrak) \cap \Lambda}$ 
  coincides with a quasipolynomial.

\subsection{Piecewise polynomial functions and continuity.}
\label{Subsection:ContinuityLimits}

A function that is piecewise polynomial with respect to the alcoves is only defined 
on the affine regular points. We 
 will however be most interested in evaluations and derivatives of these functions 
at points in the lattice $\Lambda$, which are affine singular. 
In this subsection we will use limits to define these evaluations.

 Let $h$ be a piecewise polynomial function and let 
 $\lambda$ be an affine singular point. If 
 $\lim_{\eps \to 0} h(\lambda+ \eps w)=\lim_{\eps\to 0} h(\lambda+ \eps w') =:c_\lambda$ for all  affine regular vectors
 $w$, $w'$, then we call $h$ continuous in $\lambda$ and define $h(\lambda):=c_\lambda$.
 In general, 
 we can use a limit
 procedure  as follows.
 We fix an affine regular vector $w$ and define $\lim_w h(\lambda) := \lim_{\eps \searrow 0} h(\lambda+ \eps w)$.

Differentiation can be defined in a similar way. 
 We fix an affine regular vector  $w\in U$. %
  Let $u \in \Lambda$. 
  Let $\cfrak \subseteq U$ be an alcove \st %
  $u$ and $u+ \eps w$ are contained in its closure for some small $\eps>0$ and 
  let $h_\cfrak$ be the polynomial that agrees with $h$ on the closure of $\cfrak$. 
 For a differential operator $p(D) \in \sym(U)$ we define
 \begin{align}
  \lim_w p(\Dpw) h (u) :=  p(D) h_\cfrak(u) %
 \end{align}
 (pw stands for piecewise). 
  More information on this construction can be found in \cite{deconcini-procesi-vergne-2013} where it was  introduced.

  Note that the choice of the vector $w$ is important. For example, for the list $X=(1)$,
  $\lim_w B_{X}(0)$ is either $1$ or $0$ depending on whether $w$ is positive or negative.

\subsection{Zonotopal spaces}
\label{Subsection:ZonotopalSpaces}

In this subsection we will define the spaces  $\Dcal(X)$ and $\DM(X)$ which will turn out to be the spaces
spanned by the local pieces of $T_X$ and $\vpf_X$. We will also define the space $\Pcal(X)$ which is dual to $\Dcal(X)$.

 Recall that the list of vectors $X$ is contained in a vector space $U \cong \R^d$ and that we denote the dual space by $V$.
We start by defining a pairing between the symmetric algebras $\sym(U) \cong \R[s_1,\ldots, s_d]$ and 
$\sym(V) \cong \R[t_1,\ldots, t_d]$:
\begin{equation}
\label{eq:pairing}
\begin{split}
 \pair{\cdot}{\cdot} : \R[s_1,\ldots, s_d] \times \R[t_1,\ldots, t_d] &\to \R  \\
 \pair{ p }{ f } := \left( p\left(\diff{t_1},\ldots, \frac{\partial}{\partial t_d} \right)  f \right) (0), 
\end{split}
\end{equation}
\ie we let $p$ act on $f$ as a differential operator and take the degree zero part of the result.
Note that this pairing  extends to a pairing
 $\pair{\cdot}{\cdot} : \R[[s_1,\ldots, s_d]] \times \R[t_1,\ldots, t_d] \to \R$.

A sublist $C\subseteq X$ is called a \emph{cocircuit} if  $\rank(X\setminus C) < \rank(X) $ and $C$ is inclusion-minimal with this property.

 A vector $u\in U$ corresponds to  a linear form $p_u \in \sym(U)$.
For a sublist $Y\subseteq X$, we define $p_Y := \prod_{y\in Y} p_y$.  For example, if $Y=((1,0),(1,2))$, then 
 $p_Y=s_1(s_1 + 2s_2)$. Furthermore, $p_\emptyset := 1$.

 \begin{Definition}
 \Xintro  We define
 \begin{align}
    \text{the \emph{cocircuit ideal} } \; \cJ(X) &:= \ideal\{  p_Y : Y \text{ cocircuit}  \} \;\text{ and}%
    \\
    \notag
    \text{the \emph{continuous $\Dcal$-space} } \; \Dcal(X) &:= \{ f : \sym(V) : p(D) f = 0 \text{ for all } p \in \cJ(X)  \}. %
 \end{align}
 Equivalently, $\Dcal(X)$ is the orthogonal complement of $\cJ(X)$ under the pairing $\pair{}{}$.
\end{Definition}

 We define the \emph{rank} of a sublist $Y\subseteq X$ as
 the 
 dimension 
 of the vector space spanned by $Y$. We denote it by
  $\rank(Y)$.
 Now we define the  
\begin{align}
\text{\emph{central $\Pcal$-space} } \Pcal(X) &:= \spa\{ p_Y :  \rank(X\setminus Y)= \rank(X) \}
\label{equation:CentralP}
\\
\text{and the \emph{internal $\Pcal$-space} }   \Pcal_-(X) &:= \bigcap_{x\in X} \Pcal(X\setminus x).
\label{equation:InternalP}
\end{align}
The space $\Pcal(X)$ first appeared in approximation theory \cite{akopyan-saakyan-1988,boor-dyn-ron-1991,dyn-ron-1990}. 
The space $\Pcal_-(X)$ was introduced in
\cite{holtz-ron-2011}.

\begin{Proposition}[\cite{dyn-ron-1990, holtz-ron-2011}]
\label{Proposition:JPdecomposition}
\Xintro 
Then  $\sym(U) = \Pcal(X) \oplus \cJ(X)$.
\end{Proposition}

 \begin{Theorem}[\cite{dyn-ron-1990,jia-1990}] %
\label{Proposition:PDduality}
\Xintro
Then the spaces $\Pcal(X)$ and $\Dcal(X)$ are dual under the pairing $\pair{\cdot}{\cdot}$, \ie the map
\begin{align}
 \begin{split}
 \Dcal(X) &\to \Pcal(X)^* \\
  f &\mapsto \pair{\cdot}{f}
 \end{split}
\end{align}
is an isomorphism.
\end{Theorem}

 Recall that $\Z[\Lambda]$ acts via translation on $\Ccal_\Z[\Lambda]= \{ f:\Lambda \to \Z\}$.
 For $p\in \Z[\Lambda]$ and $f\in\Ccal_\Z[\Lambda]$, we will sometimes write $p(\nabla)f$ to denote the function that is obtained when $p$ acts on $f$.
 \begin{Definition}
 \XintroLattice Then we define
 \begin{align}
      \text{the \emph{discrete cocircuit ideal} } \dJ(X) &:= \ideal\{  \nabla_Y : Y \text{ cocircuit}  \} \subseteq \Z[\Lambda] \\
    \text{and the \emph{discrete $\Dcal$-space} }   \DM(X) &:= \{ f \in \Ccal_\Z[\Lambda] :  p (\nabla) f = 0 \text{ for all } p \in \dJ(X)   \}. \notag
 \end{align}
\end{Definition}
\begin{Remark}
 The spaces $\Dcal(X)$ and $\DM(X)$ are sometimes called \emph{continuous} and \emph{discrete Dahmen-Micchelli spaces}.
\end{Remark}

 \begin{Definition}
   We will write $\Dcal_\CC(X)$, $\DM_\CC(X)$, $\Pcal_\CC(X)$, $\dJC(X)$ etc.\ to denote the complexified versions of these vector spaces and ideals.
 \end{Definition}

 \begin{Remark}
  If $X$ is unimodular, then
  $\Dcal_\CC(X)|_\Lambda = \DM_\CC(X)$. 
  This is a special case of %
  Proposition~\ref{Proposition:DMdecomposition} below.
 \end{Remark}
Recall that $Z(X)$ denotes the zonotope defined by $X$.
 For $w\in U$, we define $\Zcal(X,w) := (Z(X) - w ) \cap \Lambda$.
 \begin{Theorem}[Theorem 13.21 in \cite{concini-procesi-book}]
 \label{Theorem:DMdefinedbyDelta}
 \XintroLattice
  $\DM(X)$ is a free abelian group %
  consisting of quasipolynomials. Its dimension is equal to $\vol(Z(X))$.
  For any affine regular vector $w$, 
  evaluation of the functions in $\DM(X)$ on the set 
  $  \Zcal(X,w)  $
  establishes a linear
  isomorphism of $\DM(X)$ with the abelian group of all $\Z$-valued functions on $\Zcal(X,w)$.
 \end{Theorem}

In \cite{holtz-ron-2011} it was shown that 
if $X$ is unimodular then   $\dim \Pcal_-(X) = \abs{\Zcal_-(X)}$ and $\dim\Pcal(X)=\dim\Dcal(X)= \vol Z(X)$.
More specifically, the following is known.
\begin{Theorem}[\cite{ardila-postnikov-2009, holtz-ron-2011}]
\label{Proposition:HilbertSeriesTuttePolynomialP}
\XintroN
Then
\begin{align}
 \hilb(\Pcal_-(X),q) &=  q^{ N - d}{\mathfrak T}_X(0,q^{-1})  \text{ and } 
\\ 
  \hilb( \Pcal(X), q) &= \hilb( \Dcal(X), q) =  q^{ N - d }{\mathfrak T}_X(1, q^{-1}),
\end{align}
where $\tutte_X( \alpha, \beta) := \sum_{A\subseteq X} ( \alpha - 1 )^{ r - \rank(A) }( \beta - 1 )^{ \abs{A}- \rank(A) }$ denotes the Tutte polynomial of the matroid defined by $X$ and 
$\hilb(\bullet,q)$ denotes the Hilbert series of the graded vector space $\bullet$. 
\end{Theorem}

Let $x\in X$. We call the list $X\setminus x$ the \emph{deletion} of $x$. 
 The image of $X\setminus x$ under the canonical projection $\pi_x : U \to U/\spa(x) =:U/x$ 
 is called the \emph{contraction} of $x$. It is denoted by $X/x$. %

The projection $\pi_x$ induces a map $ \sym(U)\to \sym(U/x)$ that we will also denote by $\pi_x$.
 If we identify $\sym(U)$ with the polynomial ring $\R[s_1,\ldots,s_d]$ 
  and $x=s_d$, then
  $\pi_x$ is the map from 
 $\R[s_1,\ldots, s_{d}]$ to $\R[s_1,\ldots, s_{d-1}]$ that sends $s_d$ to zero and $s_1,\ldots, s_{d-1}$ to themselves.

 Theorem~\ref{Proposition:HilbertSeriesTuttePolynomialP} can be deduced  from the following proposition.
  \begin{Proposition}[\cite{ardila-postnikov-2009,ardila-postnikov-errata-2012}]
  \label{Proposition:DeletionContractionP}
  \Xintro
  Let $x\in X$ be an element that is non-zero. %
  Then the following sequences of graded vector spaces are exact:
   \begin{align}
   \label{eq:exactSequencePcentral}
    0 \to \Pcal ( X\setminus x )[1]  \stackrel{\cdot p_x}{\longrightarrow}& \Pcal (X) \stackrel{\pi_x}{\longrightarrow} \Pcal (X/x) \to 0
\\
    \label{eq:exactSequencePInt}
 \text{and } \quad  0 \to \Pcal_-( X\setminus x )[1]  \stackrel{\cdot p_x}{\longrightarrow} &\Pcal_-(X) \stackrel{\pi_x}{\longrightarrow} \Pcal_-(X/x) \to 0.
   \end{align}
  Here, $[1]$ means that the degree of the graded vector space should be shifted up by one. 
 \end{Proposition}

\begin{Proposition}[\cite{dyn-ron-1990}%
 ]
\label{Proposition:Pbasis}
\XintroList
  A basis for $\Pcal(X)$ is given by  
  $\Bcal(X) := \{ Q_B   : B \in \BB(X) \}$,
 where $Q_B:=p_{ X\setminus (B \cup E(B))}$ and $E(B)$ denotes the set of externally active elements in $X$ with respect to the basis $B$,
 \ie 
 $E(B):=\{  x_j \in X\setminus B : x_j\not\in \spa\{ x_i : x_i \in B\text{ and } i < j   \} \}$.
\end{Proposition}

\subsection{The structure of splines and vector partition functions}

 \begin{Theorem}[Theorems~11.35 and 11.37 in \cite{concini-procesi-book}]
\label{Proposition:Dlocalpieces}
\Xintro
 On each big cell, $T_X$ agrees with polynomial that is 
 contained in $\Dcal(X)$.
 These polynomials are pairwise different.
 Furthermore, the space $\Dcal(X)$ is spanned by the local pieces of $T_X$ and their partial derivatives.
 \end{Theorem}
 It is not difficult to see that
 \begin{align}
 \label{eq:boxsplinebytx}
   B_X(u) = \sum_{A\subseteq X}  (-1)^{\abs{A}} T_X( u - \sum_{x\in A} x ). 
 \end{align}
 One can use this fact to  deduce the following result.
 \begin{Corollary}
\label{Proposition:boxlocalpieces}
 The box spline $B_X$ agrees with a polynomial in $\Dcal(X)$ on each alcove.
\end{Corollary}

\begin{Theorem}[\cite{szenes-vergne-2003} and Theorem 13.52 in \cite{concini-procesi-book}]
\label{Theorem:LocalPiecesOverlap}
\XintroLattice
Let $\Omega$ be a big cell. Then the vector partition function $\vpf_X$ agrees with
a quasipolynomial $\vpf_X^\Omega  \in \DM(X)$ on 
 $(\Omega - Z(X)) \cap \Lambda$.
\end{Theorem}

\begin{Remark}
  Dahmen and Micchelli observed that  
 \begin{align}
 \label{eq:DahmenMicchelli}
   T_X(u) = B_X *_d  \vpf_X (u):= \sum_{\lambda\in \Lambda } B_X(u - \lambda)\vpf_X(\lambda)  
 \end{align}
 (cf.~\cite[Proposition 17.17]{concini-procesi-book}).
 The symbol $*_d$ stands for %
 (semi-)discrete convolution.
 \end{Remark}

%% file: ZP_multspline_paper.pdf_t
\begin{picture}(0,0)%
\includegraphics{ZP_multspline_paper.pdf}%
\end{picture}%
\setlength{\unitlength}{1865sp}%
\begingroup\makeatletter\ifx\SetFigFont\undefined%
\gdef\SetFigFont#1#2#3#4#5{%
  \reset@font\fontsize{#1}{#2pt}%
  \fontfamily{#3}\fontseries{#4}\fontshape{#5}%
  \selectfont}%
\fi\endgroup%
\begin{picture}(10274,4605)(-2023,-7816)
\put(4951,-7261){\makebox(0,0)[lb]{\smash{{\SetFigFont{7}{8.4}{\familydefault}{\mddefault}{\updefault}{\color[rgb]{0,0,0}$ \frac 12 u_2^2 $}%
}}}}
\put(4726,-6361){\makebox(0,0)[lb]{\smash{{\SetFigFont{7}{8.4}{\rmdefault}{\mddefault}{\updefault}{\color[rgb]{0,.56,.56}$ \frac 12 u_2^2 + \frac 32 u_2 + 1$}%
}}}}
\put(2341,-4471){\makebox(0,0)[lb]{\smash{{\SetFigFont{5}{6.0}{\rmdefault}{\mddefault}{\updefault}{\color[rgb]{0,.56,.56}$  +\frac 78 + \frac 18  (-1)^{u_1 + u_2}$}%
}}}}
\put(2296,-4831){\makebox(0,0)[lb]{\smash{{\SetFigFont{7}{8.4}{\familydefault}{\mddefault}{\updefault}{\color[rgb]{0,0,0}$-\frac 14 u_1^2 + \frac 12 u_1u_2 + \frac 14 u_2^2$}%
}}}}
\put(2161,-3346){\makebox(0,0)[lb]{\smash{{\SetFigFont{6}{7.2}{\familydefault}{\mddefault}{\updefault}{\color[rgb]{0,0,0}$u_2$}%
}}}}
\put(8236,-7756){\makebox(0,0)[lb]{\smash{{\SetFigFont{6}{7.2}{\familydefault}{\mddefault}{\updefault}{\color[rgb]{0,0,0}$u_1$}%
}}}}
\put(-719,-4651){\makebox(0,0)[lb]{\smash{{\SetFigFont{5}{6.0}{\rmdefault}{\mddefault}{\updefault}{\color[rgb]{0,.56,.56}$ \left( \frac{u_1+u_2}{2}\right)^2 + u_1 + u_2 +\frac 78 $}%
}}}}
\put(136,-5461){\makebox(0,0)[lb]{\smash{{\SetFigFont{5}{6.0}{\rmdefault}{\mddefault}{\updefault}{\color[rgb]{0,.56,.56}$  + \frac 18 (-1)^{u_1 + u_2}$}%
}}}}
\put(-1079,-3796){\makebox(0,0)[lb]{\smash{{\SetFigFont{7}{8.4}{\familydefault}{\mddefault}{\updefault}{\color[rgb]{0,0,0}$ \frac 14 u_1^2 + \frac 12 u_1 u_2 + \frac 14 u_2^2 $}%
}}}}
\put(2296,-3841){\makebox(0,0)[lb]{\smash{{\SetFigFont{5}{6.0}{\rmdefault}{\mddefault}{\updefault}{\color[rgb]{0,.56,.56}$ -\frac 14 u_1^2 + \frac 12 u_1u_2 + \frac 14 u_2^2 + \frac{u_1}2 + u_2 $}%
}}}}
\end{picture}%

%% file: SLPAM_unimodular.tex
  \section{Results in the unimodular case}
  \label{Section:TUMresults}
  In this section we will review previously known results in the case where the list
  $X$ is unimodular.

Recall that the splines $B_X$ and $T_X$ are piecewise polynomial (Theorem~\ref{Proposition:Dlocalpieces} and Corollary~\ref{Proposition:boxlocalpieces}).
  The splines are obviously smooth in the interior of the regions of polynomiality. 
  This is in general not the
  case where two regions of polynomiality overlap.  
  The following theorem characterises the differential operators with constant coefficients  that leave the splines continuous.
 \begin{Theorem}[\cite{lenz-todd-online-2014}]
  \label{ML-Corollary:internalPcontinuous}
\XintroLatticeUnimod
  Then
   \begin{align}
    \Pcal_-(X) &= \{ p \in \Pcal(X) %
      : p(D) B_X \text{ is a continuous function} \}.
   \end{align}
  \end{Theorem}
  Note that because of \eqref{eq:boxsplinebytx}, a differential operator $p(D)$ with constant coefficients
  leaves $B_X$ continuous if and only if it
  leaves $T_X$ continuous.
Theorem~\ref{ML-Corollary:internalPcontinuous} ensures that the derivatives of $B_X$ that appear in the following theorem exist.
\begin{Theorem}[\cite{lenz-interpolation-online-2013}, conjectured in \cite{holtz-ron-2011}]
\label{ML-Theorem:weakHoltzRon}
\XintroLatticeUnimod
Let $f$ be a real-valued function on $\Zcal_-(X)$, 
the set of interior lattice points of the zonotope defined by $X$.

Then the space $\Pcal_-(X)\subseteq\R[s_1,\ldots, s_d]$ contains a unique polynomial $p$ \st 
  $p(D)B_X|_{ \Zcal_-(X)}=f$. %
\end{Theorem}

Let $z \in U$. As usual, the exponential is defined as $e^z:=\sum_{k\ge 0} \frac{z^k}{k!} \in \R[[s_1,\ldots, s_d]]$.
We define the ($z$-shifted) \emph{Todd operator}
 \begin{align}
  \todd(X,z) :=  e^{-p_z} \prod_{x\in X} \frac{p_x}{1 - e^{-p_x}} \in \R[[s_1,\ldots, s_d]].
 \end{align}
The Todd operator was introduced by Hirzebruch in the 1950s 
\cite{hirzebruch-1956}  
 and plays a fundamental role in the Hirzebruch--Riemann--Roch theorem for complex algebraic varieties.
 It can be expressed in terms of
the \emph{Bernoulli numbers}  $B_0= 1$, $B_1 = - \frac 12 $, $B_2 = \frac 16$,$\ldots $ 
Recall that they 
are defined by the equation
 $\frac{s}{e^s-1} = \sum_{k\ge 0} \frac{B_k}{k!} s^k$. %
One should note that 
 $e^{z}\frac{z}{e^z-1} = \frac{z}{1-e^{-z}} = \sum_{k\ge 0} \frac{B_k}{k!} (-z)^k$.
For $z\in \Zcal_-(X)$, we can fix a list $S\subseteq X$ \st $z=\sum_{x\in S} x$, since $X$ is unimodular. Let $T:=X\setminus S$. 
Then we can write the Todd operator as
$\todd(X,z) =\prod_{x\in S} \frac{ p_x}{ e^{p_x} -1}  
\prod_{x\in T}  \frac{ p_x}{ 1 - e^{ -p_x }  }$.

Recall that there is a decomposition $\sym(U) = \Pcal(X) \oplus \cJ(X)$ (cf.~Proposition~\ref{Proposition:JPdecomposition}).
Let 
 $\psi_X : \Pcal(X) \oplus \cJ(X) \to \Pcal(X) $
 denote the projection. 
 Note that this is a graded linear map and that $\psi_X$ maps  to zero any homogeneous polynomial  whose degree is at least $N-d+1$.
  This implies that there is a canonical extension 
$ \psi_X : \R[[s_1,\ldots, s_d ]] \to \Pcal(X)$  given by $\psi_X(\sum_i( g_i)) :=
 \sum_i \psi_X(  g_i )$, where $g_i$ denotes a homogeneous polynomial of degree~$i$.
Let
\begin{align}
 f_z = f_z^X :=\psi_X(\todd(X,z)).
\end{align}
\begin{Example}
For $X=(1,1)$ we obtain 
 $\todd((1,1),1) = (1 + B_1 s + \ldots )(1 - B_1 s + \ldots ) = 1 + 0s + \ldots\;$
 Hence $f_1^{(1,1)}=1 \in \Pcal_-(1,1)= \R$. Note that $\Pcal(1,1)= \spa\{1,s\}$ and $\cJ(1,1) = \ideal\{s^2\}$.
\end{Example}
\begin{Theorem}[{\cite{lenz-todd-online-2014}}]%
\label{Theorem:MainTheorem}
\XintroLatticeUnimod
 Let $z$
 be a lattice point in the interior of 
 the zonotope $Z(X)$.
 Then $f_z\in \Pcal_-(X)$, $\todd(X,z)(D)B_X$ extends continuously on $U$, and
 \begin{align}
 \label{eq:MainThm}
    f_z (D)  B_X |_\Lambda = 
    \todd(X,z) (D) B_X |_\Lambda = \delta_z  %
.
 \end{align}
Here, $\delta_z: \Lambda \to \{ 0,1 \}$ denotes the function that takes the value 1 at z and is zero elsewhere. 
\end{Theorem}
Using \eqref{eq:DahmenMicchelli}, the following variant of the
 Khovanskii--Pukhlikov formula \cite{pukhlikov-khovanski-1992} follows immediately.
 \begin{Corollary}[{\cite{lenz-todd-online-2014}}]
   \label{Corollary:KhovanskiiPukhlikovModified}
     Let $X\subseteq \Lambda \subseteq U \cong \R^d$ be a 
      list of vectors that is unimodular and spans $U$.
      Let $u\in \Lambda$ and 
      $ z \in \Zcal_-(X) $. Then
   \begin{align}
    \abs{\Pi_X(u - z)\cap \Lambda } = 
     \vpf_X(u - z) = \todd(X,z)(D) T_X(u)  = f_z(D)  T_X(u). 
   \end{align}
\end{Corollary}
 Here is an extension of Theorem~\ref{Theorem:MainTheorem}
 to the case were $z$ is allowed to lie in the boundary of the zonotope.
 In this case, $f_z\in \Pcal(X)\setminus \Pcal_-(X)$, so $f_z (D) B_X|_\Lambda$
 may not be well-defined and we have to use the limit construction explained in Subsection~\ref{Subsection:ContinuityLimits}.
\begin{Theorem}[{\cite{lenz-todd-online-2014}}]
\label{Theorem:MainTheoremBoundary}
\XintroLatticeUnimod
Let $w$ be a short affine regular vector and let $z \in \Zcal(X,w)$. Then
 \begin{align}
\lim_{w}  f_z(\Dpw) B_X |_\Lambda = \lim_{w}  \todd(X,z)(\Dpw) B_X |_\Lambda = \delta_{z}. 
\end{align}
\end{Theorem}
\begin{Corollary}[{\cite{lenz-todd-online-2014}}]
\label{Corollary:KhovanskiiPukhlikov}
\XintroLatticeUnimod
Let $w \in \cone(X)$ be a short affine regular and let $z \in \Zcal(X,w)$. 
 Let $u\in \Lambda$ and let $\Omega\subseteq \cone(X)$ be a big cell \st $u$ is contained in its closure. 
 Let $\vpf_X^\Omega$ be the quasipolynomial that agrees with $\vpf_X$ on $\Omega$. 
 Then 
  \begin{equation*}
  \label{eq:KV}
    \abs{ \Pi_X( u - z ) \cap \Lambda } = \vpf_X^\Omega ( u - z ) = \lim_w \todd(X,z)(\Dpw) T_X(u) = \lim_w f_z(\Dpw) T_X(u).
   \end{equation*}
\end{Corollary}
\begin{Remark}
 De Concini, Procesi, and Vergne 
 proved the case $z=0$ of Theorem~\ref{Theorem:MainTheoremBoundary}  in \cite{deconcini-procesi-vergne-2013}.
 They refer to it is a as a \emph{deconvolution formula}.
The original Khovanskii--Pukhlikov formula is the case $z=0$ in \eqref{eq:KV}. %
 An explanation of the Khovanskii--Pukhlikov formula that is easy to read is contained 
  in the book by Beck and Robins 
  \cite[Chapter 10]{beck-robins-ComputingTheContinuousDiscretely}.
\end{Remark}
  \begin{Corollary}[{\cite{lenz-todd-online-2014}}]
  \label{Corollary:DahmenMicchelli}
  \XintroLatticeUnimod
  Then 
    $\sum_{z \in \Zcal_-(X)} B_X(z) f_z = 1$. 
   This implies formula \eqref{eq:DahmenMicchelli} for $u\in \Lambda$.
 \end{Corollary}
  \enlargethispage{4pt}
Recall that there is a homogeneous basis for the space $\Pcal(X)$ (Proposition~\ref{Proposition:Pbasis}).
 For the internal space $\Pcal_-(X)$, there is no similar construction.
 In general this space is not spanned by polynomials of type $p_Y$ for some 
 $Y\subseteq X$~\cite{ardila-postnikov-errata-2012}.
In the unimodular case, the polynomials $f_z$ form inhomogeneous bases for both spaces.
  \begin{Corollary} [{\cite{lenz-todd-online-2014}}]
  \label{Corollary:InternalPbasis}
   Let $X\subseteq \Lambda \subseteq U \cong \R^d$ be a list of vectors that is unimodular
     and spans $U$.
  Then  $\{ f_z : z \in \Zcal_-(X) \}$ is a basis for $\Pcal_-(X)$.
  \end{Corollary}
We also obtain a new basis for the central space $\Pcal(X)$. Let
 $w\in U$ be a \emph{short affine regular vector}, \ie a vector whose Euclidian length 
  is close to zero that is not contained in any hyperplane generated by sublists of $X$.
 Let 
   $\Zcal(X,w) := (Z(X) - w ) \cap \Lambda$.
   It is known that $\dim \Pcal(X) = \abs{\Zcal(X,w)} = \vol(Z(X))$ \cite{holtz-ron-2011}.
  \begin{Corollary}[{\cite{lenz-todd-online-2014}}]
  \label{Corollary:newPbasis}
 Let $X\subseteq \Lambda \subseteq U\cong \R^d$ be a list of vectors that is unimodular and spans $U$. 
 Then
  $\{ f_z : z \in \Zcal(X,w)  \}$ is a basis for $\Pcal(X)$.
  \end{Corollary}
  \begin{Remark}
  \label{Remark:Pspacerelevant}
It is known that for $f\in \cJ(X)$, $f(D)B_X = f(D)T_X = 0 $. On the other hand, if $f\in \Pcal(X)$, 
then $f(D) B_X \neq 0$ and $f(D) T_X \neq 0$. 
Hence $\Pcal(X)$ can be seen as the space of 
relevant differential operators on $B_X$ and $T_X$ with constant coefficients.
\end{Remark}

  \subsection*{How we will generalise these results}

  In the remainder of this article, we will generalise most of the results that were mentioned in this section to the general case, \ie 
  the case where the list $X$
  is contained in a lattice or a finitely generated abelian group and $X$ is not necessarily unimodular.

  As stated in the introduction, a generalisation of the Khovanskii--Pukhlikov formula
  (essentially Corollary~\ref{Corollary:KhovanskiiPukhlikov}) is known:
 the  Brion--Vergne formula (Theorem~\ref{Theorem:BV}).
  We will use it to generalise 
  Corollary~\ref{Corollary:KhovanskiiPukhlikovModified}
  to Theorem~\ref{Theorem:ModifiedBrionVergne}.
  The main difference with the original Brion--Vergne formula is that we 
  use differential operators that leave the spline continuous so that there is no need to use limits.
  The Brion--Vergne formula uses a generalised Todd operator
  (Definition~\ref{Definition:PeriodicTodd}).
  Again, for each interior lattice point $z$ of the zonotope, we  will define  a 
  differential operator $\tilde f_z(D)$ 
  (formula \eqref{eq:psiperdefinition})
  and these differential operators will all sum to $1$, \ie
  Corollary~\ref{Corollary:DahmenMicchelli} will be generalised to Corollary~\ref{Corollary:DahmenMicchelliGeneral}.

  An operator that turns a local piece of $T_X$ into a local piece of $\vpf_X$ must map elements of 
  $\Dcal(X)$ to elements of $\DM(X)$.  
  In the unimodular case it is sufficient to take an element of $\Pcal(X)$ that defines 
  a map  $\Dcal(X)\to \Dcal(X)$ and then restrict to $\Lambda$ since in this case, restriction to $\Lambda$
  defines an isomorphism $\Dcal(X)\to \DM(X)$.
  In general, the Todd operator must turn polynomials into quasipolynomials. This motivates the
  definition of the central periodic $\Pcal$-space $\Pper(X)$ (Definition~\ref{Definition:PeriodicCentralP} which generalises \eqref{equation:CentralP}).

  There is also an internal periodic $\Pcal$-space (Definition~\ref{Definition:internalPeriodicPsimple} which generalises \eqref{equation:InternalP}).
  It can be characterised as the set of differential operators contained in  the  central periodic $\Pcal$-space that leave $T_X$ continuous 
  (Theorem~\ref{Proposition:internalContinuousCharacterisation} generalises Theorem~\ref{ML-Corollary:internalPcontinuous}).

  We will define a pairing between $\Pper_\CC(X)$ and $\DM_\CC(X)$ in  \eqref{eq:PperDMduality} that agrees with the pairing
  between $\Pcal(X)$ and $\Dcal(X)$
  defined in
  \eqref{eq:pairing} in the unimodular case.
  The spaces $\Pper_\CC(X)$ and $\DM_\CC(X)$ are in fact dual under this pairing  (Theorem~\ref{Proposition:PperDMduality}) in the same way as $\Pcal(X)$ and $\Dcal(X)$
   are dual (Theorem~\ref{Proposition:PDduality}).

 The central  periodic  space has two bases: a homogeneous basis
  (Proposition~\ref{Proposition:HomogeneousBasisCentralPeriodic}, generalising Proposition~\ref{Proposition:Pbasis}) and an
  inhomogeneous basis (Proposition~\ref{Proposition:InhomogeneousBasisCentralPeriodic}, generalising  Corollary~\ref{Corollary:newPbasis}).
  The internal space has an inhomogeneous basis (Proposition~\ref{Proposition:InhomogeneousInternalBasis}, generalising
  Corollary~\ref{Corollary:InternalPbasis}).

  Theorem~\ref{Proposition:HilbertSeriesTuttePolynomialP} that connects the Hilbert series of the $\Pcal$-spaces with the Tutte polynomial
  of the underlying matroid can also  be generalised:  the Hilbert series 
  of the periodic $\Pcal$-spaces are evaluations of the Tutte polynomial of the arithmetic matroid defined by the list $X$
  (Theorems~\ref{Proposition:CentralPeriodicPHilbertTutte} and~\ref{Proposition:PeriodicInternalTutteEval}).

 There are also short exact sequences for both types of periodic $\Pcal$-spaces:  Proposition~\ref{Proposition:DeletionContractionP}
 will be generalised to 
 Propositions~\ref{Proposition:DeletionContractionCentralPeriodic}
 and~\ref{Proposition:DeletionContractionInternalPeriodic}.

  We do not have generalisations of 
  Theorems~\ref{ML-Theorem:weakHoltzRon} and~\ref{Theorem:MainTheorem}. The reason for this is explained in
  Remark~\ref{Remark:NoBoxGeneralisation}.

%% file: SLPAM_arithmetic.tex
 \section{Generalised toric arrangements and arithmetic matroids}
 \label{Section:MoreBackground}
 
 In this section we will review some facts about 
 finitely generated abelian groups, generalised toric arrangements, and arithmetic matroids.
 The vertices of the toric arrangement  will appear in the definition of the central periodic $\Pcal$-space and the 
 arithmetic matroid captures the combinatorics of this space.

\subsection{Finitely generated abelian groups}
 Let $G$ be a group.  For a subset $A\subseteq G$, $\sg{A}$ denotes the subgroup of $G$ generated by $A$.

 If $X\subseteq \Lambda$ is unimodular, then for any $x\in X$, the quotient $\Lambda/\sg{x}$
 is still a lattice and $\sg{X/x} = \Lambda/\sg{x}$.
 For arbitrary $X\subseteq \Lambda$, this is in general not the case.
 Some deletion-contraction proofs later will require us to consider quotients.
  Therefore, it is natural for us to work 
 with $X\subseteq G$, where $G$ denotes a finitely generated abelian group.

 Let $G$ be a finitely generated abelian group.
  Let $G_t:= \{ h\in G : \text{there exists } k \in \Z_{>0} \text{ \st } k \cdot h=0\}$ denote the \emph{torsion subgroup} of $G$.
    By the fundamental theorem of finitely generated abelian groups, $G/G_t$ is isomorphic to $\Z^d$ for some $d$.
     $d$ is called the \emph{rank} of the group $G$.
     It is natural to associate with $G$ the lattice     
     $\Lambda := G / G_t = G\otimes_\Z \Z$ and the Euclidian vector space $U := \Lambda \otimes_\Z \R = G \otimes_\Z \R$.   
     So choosing a finitely generated abelian group is more general than the setting in Section~\ref{Section:Notation}, where we haven chosen a 
     vector space $U\cong \R^d$ and a lattice $\Lambda \subseteq U$.
     In Section~\ref{Section:Notation} we required that $X$ generates $U$. 
     In the case $X\subseteq G$, the suitable generalisation is that $X$ generates a subgroup of finite index.
     Recall that  the \emph{index} of a subgroup $H\subseteq G$ is defined as $\abs{G/H}$.

  Warning: working with finitely generated abelian groups instead of lattices makes some of the statements 
  appear rather complicated.
  A reader who is not interested in the proofs may always assume that $X$ is contained in a lattice.
  In fact, most of the proofs also work in this setting.
  Deletion-contraction is used only  in the proof of   Theorem~\ref{Proposition:PeriodicInternalTutteEval}
  and Proposition~\ref{Proposition:InhomogeneousInternalBasis} and  of course in the statement of the short exact sequences  
  (Propositions~\ref{Proposition:DeletionContractionCentralPeriodic}
 and~\ref{Proposition:DeletionContractionInternalPeriodic}).  
  The results involving vector partition functions all assume $X\subseteq \Lambda$ as all the previous work on this topic has
  been done in this setting. 

\subsection{Generalised toric arrangements}

 We  will now define generalised toric arrangements, which are arrangements of (generalised) subtori
 on a (generalised) torus.

 As usual, $S^1 := \{ z \in \CC : \abs{z} = 1 \}$. Recall that $G\cong \Lambda \oplus G_t$ denotes a finitely generated abelian group.
 Consider the abelian group $T( G ) = \hom( G, S^1 )$.  We can identify $G$ with
  $\hom( T(G), S^1)$. This is a special case of Pontryagin duality between compact and discrete abelian groups.
  
The group $T(G)$ is canonically isomorphic to the group of homomorphisms $G\to (\R/\Z)$.
  Let $\phi :  G \to (\R/\Z)$ be such a homomorphism.
 This defines an element $e_\phi\in T(G)$ via $e_\phi (g) :=
 e^{2\pi i \phi(g)}$.

 Note that
 $\hom(\Lambda, S^1) \times \hom(G_t, S^1)= \hom(G, S^1)$. An isomorphism is given by the map
 that sends $(e_{\phi_1},e_{\phi_2})$ to $ e_\phi(a,b) := e_{\phi_1}(a) e_{\phi_2}(b) $. 
 Since $\hom(\Lambda, S^1)$ is a compact torus and $\hom(G_t, S^1)\cong G_t$ is a finite abelian group, 
 $T(G)$ is topologically the disjoint union of $\abs{G_t}$ copies of the $d$-dimensional compact  torus.
 
 Choosing a basis for $\Lambda$ is equivalent to choosing an isomorphism
 $\hom(\Lambda,S_1) \cong (S^1)^d $. Given a basis $s_1,\ldots, s_d$,
 one can map $e_\phi \in T(G)$ to $( e_\phi(\lambda_1), \ldots, e_\phi(\lambda_d) ) \in ( S^1 )^d $.

 Every $x\in X$ defines a (possibly disconnected)  hypersurface in $T(G)$:
 \begin{align}
  H_x := \{ e_\phi \in T(G) : e_\phi (x) = 1 \}.
 \end{align}

 \begin{Definition}[toric arrangements]
 \XintroAbelianGroup
  The set $\{ H_x : x \in X \} $ is called the \emph{generalised toric arrangement}
  defined by $X$.
  
  The set $ T(G) \cap  \bigcup  \left(\bigcap_{ x\in B} H_x \right)$ where the union runs over all bases $B\subseteq X$
   is a finite set.  
   It is called the set of \emph{vertices of the toric arrangement} 
   and denoted by  $\Vcal(X)$. By basis, we mean a set of cardinality $d$ that generates a subgroup of finite index. 
   The intersection with $T(G)$ ensures that $\Vcal(X) = T(G)$ if $\rank(G)=0$.
\end{Definition}

Note that if $X\subseteq \Z$, then $\Vcal(X)$ is a set of roots of unity.
 See Figure~\ref{Figure:toricarr} on page~\pageref{Figure:toricarr} for a two-dimensional example and 
 Example~\ref{Example:toriarrwithtorsion} for a toric arrangement on the torus $T(\Z\oplus \Z/3\Z)$.

  If $G$ is isomorphic to a lattice $\Lambda$, everything is a bit simpler. In particular, the torus $T(G)$ will be connected.
  We denote the dual lattice of $\Lambda$ by $V\supseteq \Gamma :=
   \{ v \in V : v(\lambda) \in \Z \text{ for all } \lambda\in\Lambda\}$.
   Note that if we identify 
   $U$ and $V$ with $\R^d$ and a basis for $\Lambda$ is given by the columns of a  $(d\times d)$-matrix, 
   then the rows of the inverse of this matrix form a basis for $\Gamma$.

 Recall that a vector $x \in U$  defines a hyperplane $H_x = 
 \{ v \in V : v(x) = 0  \}$.
  The set $H_x^p = \{ v \in V : v(x) \in \Z  \}$ is a periodic
  arrangement of countably 
   many shifts of the hyperplane $H_x$.
   Note that $\gamma(x)\in \Z$ for all $\gamma\in \Gamma$ if $x\in \Lambda$.
  This implies that for all $x\in X$, $\Gamma$ acts on $H_x^p$ by translation. 
  The quotient
   $H_x^t :=H_x^p/\Gamma = \{   v \in V/\Gamma : v(x) = 0 \}$ is a (possibly disconnected) 
   hypersurface in the torus $V/\Gamma  \cong (\R/\Z)^d$.
  The toric arrangement defined by $X$ is then   the set $\{ H_x^t : x \in X \}$.

 In Section~\ref{Section:PeriodicDuality}, we will use the \emph{algebraic torus} 
 $T_\CC(G) := \hom(G, \CC^*) \cong T(G) \times \hom(G, \R_{> 0})$.
 Note that if one defines a toric arrangement as a family of subsets of $T_\CC(G)$, the set of vertices $\Vcal(X)$ will 
 still be contained in $T(G)$. For this reason, it does not make a big difference for us whether we work with $T(G)$ or $T_\CC(G)$.
 The compact torus is better suited for drawing pictures and the algebraic torus has nicer
 algebraic properties that we will use in
 Section~\ref{Section:PeriodicDuality}.

\smallskip
The following remark and proposition show that toric arrangements appear naturally in the theory of vector partition functions.
\begin{Remark}
 The Laplace transform of the vector partition function $\vpf_X$ can be interpreted 
 as a rational function on the torus $T(G)$ that maps $e_\phi$ to $\frac{1}{\prod_{x\in X} (1 - e_\phi(x))}$.
 The set of poles of this function 
is precisely the toric arrangement defined by the list $X$.
 
 For the multivariate spline $T_X$ there is an analogous statement: 
 the  Laplace transform  is the rational function on the vector space $V$ that maps $v$ to $\frac{1}{\prod_{x\in X} v(x)}$.
 The set of its poles is the central hyperplane arrangement  defined by the list $X$ 
 (see \eg \cite{concini-procesi-book}).
\end{Remark}

 Let $e_\phi\in \Vcal(X)$.  We define a sublist 
 $X \supseteq X_\phi := ( x \in X : e_\phi(x) = 1 ) = X \cap H_x$. 
 This is the maximal sublist of $X$ such that $ \bigcap_{x\in  X_\phi} H_x = \{e_\phi\}$. Note that by construction, $X_\phi$ always generates a subgroup
 of finite index.

 \begin{Proposition}[Section 16.1 in \cite{concini-procesi-book}]
 \label{Proposition:DMdecomposition}
  \XintroLattice 
 Then  $\DM_\CC(X) = \bigoplus_{e_\phi \in \Vcal(X)} e_{\phi} \Dcal_\CC(X_\phi)|_\Lambda$.
  
  So in particular, if $X$ is unimodular, then $\DM_\CC(X)=\Dcal_\CC(X)|_\Lambda$.
 \end{Proposition}

 \subsection{Arithmetic matroids}
   
   We assume that the reader is familiar with the  definition of a matroid (see \eg \cite{concini-procesi-book,MatroidTheory-Oxley}).
  An \emph{arithmetic matroid} is a pair $ (M, \multari ) $, where
  $M$ is a matroid on the ground set $A$ and  
 $\multari : 2^A \to \Z_{\ge 0}$ is a function that satisfies certain axioms \cite{branden-moci-2014,moci-adderio-2013}. 
 The function $\multari$ is called the \emph{multiplicity function}.
 
The prototype of an arithmetic matroid is the one that is
  canonically associated with a finite list $X$ of elements of a finitely generated abelian group $G$.
 Given a sublist $S\subseteq X$, the \emph{rank} $\rank(S)$ of $S$ is defined to be the rank of the group $\sg{S}$.
 Let $G_S\subseteq G$ be the maximal subgroup of $G$ \st the index $\abs{G_S/\sg{S}}$ is finite.
 Then we define $\multari(S) :=  \abs{G_S/\sg{S}}$.

If the list $X$ is contained in a lattice, then one can equivalently define 
 $\rank(S):=\dim \spa(S)$ and 
 $\multari(S):= \abs{   (\spa(S)\cap \Lambda) / \sg{S}  }$ for $S\subseteq X$.
 Note that in this case if $S\subseteq X$ is linearly independent, then $\multari(S)$ is equal to the number of lattice points in the half-open
  parallelepiped $\{ \sum_{s\in S} \lambda_s  s  :  0 \le \lambda_s  <  1 \}$.

   The \emph{arithmetic Tutte polynomial} \cite{moci-adderio-2013, moci-tutte-2012}
   is defined as
 \begin{align}
  \aritutte_X (\alpha, \beta) = \sum_{ S \subseteq X } \multari(S) (\alpha-1)^{ d - \rank(S)  }  
  (\beta - 1)^{ \abs{S} - \rank(S)  }.
 \end{align}
 Note that if $X\subseteq \Lambda$ is unimodular, then the multiplicity function is constant and equal to $1$. Hence the arithmetic Tutte polynomial 
 and the Tutte polynomial are equal in this case.

 We call an element $x\in X$ a \emph{coloop} if $\rank(S \cup x)= \rank(S) + 1$ for all $S\subseteq X \setminus x$.
 Recall that in matroid theory an element of rank $0$ is called a \emph{loop}. If the matroid is represented by a list of vectors, loops are always
 represented by the vector $0$.
 It is important to note that in the case of arithmetic matroids there can be elements of rank $0$ that are non-zero, namely elements of the torsion subgroup.
  
 An important property is the following deletion-contraction identity (Lemma 5.4 in \cite{moci-adderio-2013}).
 If the arithmetic matroid $(M,\multari)$ is represented by the list $X$ and $x \in X$, then 
 the lists $X\setminus x$ and $X/x$ (as defined in Subsections~\ref{Subsection:ZonotopalSpaces} and~\ref{Subsection:DeletionContraction}) represent the arithmetic matroids obtained by
 deleting and contracting $x$, respectively. 
 Let $x\in X$ be a vector that is 
 neither torsion nor a coloop. Then
 \begin{equation}
 \label{equation:AriTutteDelCon}
  \aritutte_X(\alpha,\beta) = \aritutte_{X\setminus x}(\alpha,\beta) + \aritutte_{X/x}(\alpha,\beta). 
 \end{equation}

 Simple matroids capture the combinatorial structure of central hyperplane arrangements (see \eg \cite{stanley-2007}). 
 In a similar way, arithmetic matroids describe the combinatorial structure of toric arrangements.
 For example,
 the characteristic polynomial of the  toric arrangement defined by a list $X$ is equal to 
 $(-1)^d \aritutte_X(1 - q, 0)$ (\cite[Theorem 5.6]{moci-tutte-2012}).
 Toric arrangements also appear naturally in the theory of vector partition functions.
 The following result is a discrete analogue of a special case  of Theorem~\ref{Proposition:HilbertSeriesTuttePolynomialP}.
 \begin{Proposition}
   \label{Proposition:dimensionDM}
   \XintroLattice 
 Then
  $\dim(\DM_\CC(X)) = \aritutte_X(1,1)$. 
   \end{Proposition}
 \begin{proof}
  This follows directly from  Corollary~3.4 in  \cite{moci-adderio-ehrhart-2012}
 (Proposition~\ref{Proposition:ZonotopeArithmeticTutte} below) and Theorem~\ref{Theorem:DMdefinedbyDelta}.
 \end{proof}
 Theorem~6.3 in \cite{moci-tutte-2012} states a stronger result, \ie a relationship between 
 the Hilbert series of $\DM_\CC(X)$ and $\aritutte_X(1,q )$.
 However, the result in \cite{moci-tutte-2012} is slightly incorrect, \ie it only holds if one uses a special grading on $\DM_\CC(X)$.

 \section{The improved Brion--Vergne formula}
 \label{Section:GeneralCase}
 
 In this section and the next two, we will discuss the new results that are contained in this paper. 
  We will first introduce the space $\Pper(X)$, a space of differential operators
  with periodic coefficients, before proving analogues of some of the results in Section~\ref{Section:TUMresults}, in particular an improved version of the
  Brion--Vergne formula.

 Recall that for a vertex of the toric arrangement $e_\phi\in \Vcal(X)$, we have defined the sublist 
 $X \supseteq X_\phi := ( x \in X : e_\phi(x) = 1 ) = X \cap H_x$.

 \begin{Definition}
 \label{Definition:PeriodicCentralP}
\XintroLattice
 We define the periodic coefficient analogue of the central $\Pcal$-space, the  
\begin{align*}
  \text{\emph{central periodic $\Pcal$-space} } \Pper(X) &:= 
  \bigoplus_{e_\phi\in \Vcal(X)} e_\phi p_{X\setminus  X_\phi}  \Pcal(X_\phi) %
   \subseteq \bigoplus_{e_\phi\in \Vcal(X)} e_\phi \sym(U).
\end{align*}
  \end{Definition}

  \begin{Remark}
 Let $p\in \Pcal(X)$. Then $p(D)$ obviously defines 
 a map $\Dcal(X) \to \Dcal(X)$.
 Now let $p\in \Pper(X)$. It is slightly less obvious that
 $p(D)$ (followed by restriction to $\Lambda$) defines
 a map $\Dcal(X)\to \DM_\R(X)$. This is a consequence of the decomposition 
 $\DM_\CC(X) = \bigoplus e_\phi \Dcal_\CC(X_\phi)|_\Lambda$ in Proposition~\ref{Proposition:DMdecomposition}.
 The relationship between $\Pper(X)$ and $\DM(X)$ will be explained in more detail in Section~\ref{Section:PeriodicDuality}.

    Note that even though the spaces
    $\Pper_\CC(X) = \bigoplus_{e_\phi\in \Vcal(X)} e_\phi p_{X\setminus X_\phi} \Pcal_\CC(X_\phi)$ and $\DM_\CC(X)
     = \bigoplus_{e_\phi\in \Vcal(X)} e_\phi \Dcal_\CC(X_\phi)|_\Lambda$ look quite similar, there is an important difference between them:
     both, $e_\phi$ and $f\in \Dcal(X_\phi)|_\Lambda$ are functions defined on $\Lambda$, so $e_\phi f \in \DM(X)$ is a function on $\Lambda$ as well.
     For $e_\phi p \in e_\phi p_{X\setminus X_\phi} \Pcal(X_\phi)$ however,  the situation is different.
     The polynomial $p$ is contained in $\sym(U)$, so it is a differential operator on $\sym(V)$ and $e_\phi$ is still a function on $\Lambda$. So the term $e_\phi p$
     can be thought of as function that assigns a differential operator with complex coefficients acting on $\sym(V)$ to each point in~$\Lambda$.
  \end{Remark}
  
 \begin{Definition}[Periodic Todd operator]
 \label{Definition:PeriodicTodd}
  Let $X\subseteq U \cong \R^d$ be a finite list of vectors and
  let $z\in U$. Then we define the ($z$-shifted) \emph{periodic Todd operator}
  \begin{align}
    \toddper(X, z) :=  \sum_{e_\phi \in \Vcal(X)} 
    e_{\phi} \cdot e_{\phi}( - z) e^{-p_z} \prod_{x \in X} \frac{ p_x }{ 1 - e_\phi(-x) e^{-p_x}  }.
  \end{align}
\end{Definition}

$\toddper(X,z)$ can be thought of as a map $\Lambda \to \R[[s_1,\ldots, s_d]]$. The term  $e_{\phi}$ is a map $\Lambda \to S^1$, whereas $e_{\phi}(-z) \in S^1$.
  Note that if the list $X$ is unimodular, 
  then $\Vcal(X) = \{ 1 \} $. This implies $\toddper(X,z) = \todd(X,z)$ and $\Pper(X) = \Pcal(X)$.

 The following theorem first appeared in \cite[p.~802]{brion-vergne-1997}. 
 In \cite[Theorem 3.3]{deconcini-procesi-vergne-2013}, the notation is more similar to ours.
\begin{Theorem}[Brion--Vergne formula]
\label{Theorem:BV}
\XintroLattice
 Let $u\in \cone(X) \cap \Lambda$ and let $w \in U$   be a short affine regular vector 
 \st $ u + w \in \cone(X) $.
   Then
  \begin{align}
    \lim_w \toddper(X,0)(\Dpw) T_{X} (u) = \vpf_X(u).
  \end{align}
\end{Theorem}

Recall that
we have defined  a projection map
 $\psi_X : \R[[s_1,\ldots, s_d]] \to \Pcal(X)$ earlier.
  Now we require a projection 
  $ \psiper_X : \bigoplus_{e_\phi\in \Vcal(X)}  e_\phi 
  \R[[s_1,\ldots, s_d]] \to \bigoplus_{e_\phi\in \Vcal(X)}  e_\phi 
  \R[s_1,\ldots, s_d]$ that maps $\toddper(X,z)$ to $\Pper(X)$.
 
 Let $f \in \bigoplus_{e_\phi\in \Vcal(X)}  e_\phi \R[[s_1,\ldots, s_d]] $. Then 
 $f$ can be written uniquely as $f= \sum_{e_\phi \in \Vcal(X)} e_\phi f_\phi$
 for some $f_\phi \in \R[[s_1,\ldots, s_d]]$. We define 
 \begin{equation}
 \label{eq:psiperdefinition}
  \psiper_X( f ) :=  \sum_\phi e_\phi   \psi_X ( f_\phi ) 
 \quad \text{ and } \quad \tilde f_z := \psiper_X ( \toddper(X,z)).
  \end{equation}

\begin{Remark}
 Note that $1 - e_\phi(-x) e^{-x}$ is invertible as a formal power series
 if and only if $ e_\phi(-x)\neq 1$ (a formula for the inverse is given on p.~516 of \cite{deconcini-procesi-vergne-2013}).
 This implies that $p_{X\setminus X_\phi}$ divides $T^\phi$, the $e_\phi$ component
 of $\toddper(X)$. Hence $\psi_X(T^\phi)\in p_{X\setminus X_\phi} \Pcal(X_\phi)$.
 This implies that $\tilde f_z \in \Pper(X)$ for any $z\in U$.
 \end{Remark}

\begin{Remark}
We can also define 
   $\psiper_X(f(\lambda, \cdot)) :=  
   \psi_X(\sum_\phi e_\phi(\lambda)  f_\phi(\cdot)))$
   for fixed $\lambda\in \Lambda$ if we complexify all the vector spaces.
\end{Remark}
We will be  able to prove the  following result using Theorem~\ref{Theorem:BV}.
\begin{Theorem}[Improved Brion-Vergne formula]
\label{Theorem:ModifiedBrionVergne}
\XintroLattice

 \begin{asparaenum}[(i)]
  \item
   Let $w \in U$ be a short affine regular vector, $u\in \cone(X) \cap \Lambda$ and 
   let $z \in \Lambda$ \st $u-z+w\in \cone(X)$. 
   Let $\Omega$ denote the   big cell  whose closure contains $u$ and $u+\eps w$ for some small $\eps >0$.
   Let $i_X^\Omega$ denote  the quasipolynomial that agrees with $\vpf_X$ on $(\Omega - Z(X)) \cap \Lambda$.
   Then
  \begin{equation}
    \lim_w \tilde f_z(\Dpw) T_{X} (u)  = i_X^\Omega( u - z ).
  \end{equation}
 Furthermore,  if $z\in \Zcal(X,w)$, then $i_X^\Omega(u-z) = i_X(u-z)$. 
  \item 
 If $z\in \Zcal_-(X)$, then $\tilde f_z(D) T_X$ is continuous
      in $\Lambda$ and
   the following formula holds:
  \begin{equation}
   \tilde f_z(D) T_{X}(u) = \vpf_X( u - z ).
  \end{equation}
 \end{asparaenum}
\end{Theorem}
 Note that the theorem only states that $\tilde f_z(D) T_{X}$ is  continuous on $\Lambda$ and not on all of $U$.
 There are two reasons for this: $\tilde f_z$ is \textit{a priori} defined only on $\Lambda$ and there are many different ways of 
 extending $\tilde f_z$ to $U$.  Furthermore, if one extends $\tilde f_z(D)$ to $U$, then 
 $\tilde f_z(D) T_{X}$ will usually be discontinuous at the non-lattice points where two regions of polynomiality overlap
  (cf.~Figure~\ref{Figure:BoxSplineValuesC}).

\begin{Example}
 Let $X = (1,2)$ (cf.~Example~\ref{Example:onetwo}). 
 Then $\Pper(X) = \spa \{1,s,  (-1)^\lambda s \}$,
 $f_1 = 1 + \frac s2  - (-1)^\lambda \frac s2$, and $f_2 = 1 - \frac s2 + (-1)^\lambda \frac s2$.
 Theorem~\ref{Theorem:ModifiedBrionVergne} correctly predicts that 
 $f_1(D)T_X(u)|_\Z = \vpf_X(u-1)$. $\vpf_X(u)$ is equal to  $\frac u2+1$ for even $u$ and $\frac{u+1}{2}$ for odd $u$. 
\end{Example}

\begin{Corollary}
  \label{Corollary:DahmenMicchelliGeneral}
\XintroLattice
  Then 
    $\sum_{z \in \Zcal_-(X)} B_X(z) \tilde f_z = 1$. 
 \end{Corollary}

  \begin{proof}[Proof of Theorem~\ref{Theorem:ModifiedBrionVergne}]
  \begin{asparaenum}[(i)]
   \item
   The second statement follows from Theorem~\ref{Theorem:LocalPiecesOverlap}.
   We will prove the first statement in two steps: (a)
   Let $T_X^\Omega$ denote the polynomial that agrees with $T_X$ on $\Omega$. Then
    \begin{align}
     \lim_w \toddper(X, z)(\Dpw) T_X(u) &=
    \sum_{e_\phi \in \Vcal(X)} 
    e_{\phi}(u) \cdot e_{\phi}( - z) e^{-p_z} \prod_{x \in X} \frac{ p_x }{ 1 - e_\phi(-x) e^{-p_x}  } T_X^\Omega(u) 
    \notag\\ &=
    \sum_{e_\phi \in \Vcal(X)} 
    e_{\phi}(u-z)  \prod_{x \in X} \frac{ p_x }{ 1 - e_\phi(-x) e^{-p_x}  } T_X^\Omega(u-z) 
    \notag\\ &= i_X^\Omega(u-z).
    \end{align}
   The last step uses Theorem~\ref{Theorem:BV}. 
   (b) %
       Let $e_\phi\in \Vcal(X)$ and let $f_\phi$ be the formal power series that is the $e_\phi$ part of $\toddper(X,z)$.
       For $i\in \N$, the degree $i$ part of $ j_\phi := f_\phi - \psi_X(f_\phi) $ is contained in $ \cJ(X) $.
       By Theorem~\ref{Proposition:Dlocalpieces}, this implies that $j_\phi$ annihilates all  the local pieces of $T_X$.
       Hence $\lim_w \toddper(X,z)(\Dpw) T_X (u) =  \lim_w f_z(\Dpw) T_X(u)$.  
\item
If $u$ lies in the interior of a big cell, then $T_X$ agrees with a polynomial in a small neighbourhood of $u$ and nothing needs to be shown.
Now suppose that $u$ lies in the intersection of the closures of two big cells $\Omega_1$ and $\Omega_2$.
 Let $w_1$ and $w_2$ be two affine regular vectors \st $u+\eps w_i \in \Omega_i$ for sufficiently small $\eps>0$.
 Let $i_X^{\Omega_1}$ and $i_X^{\Omega_2}$ denote the corresponding quasipolynomials as in (i).
 Using  (i)   we obtain
 \begin{equation*}
   \lim_{w_1} \tilde f_z(\Dpw) T_X%
   (u) = i_X^{\Omega_1}(u-z) = \vpf_X( u - z ) = i_X^{\Omega_2}(u-z) =  \lim_{w_2} \tilde f_z(\Dpw) T_X%
   (u). 
  \end{equation*}
 The second and third equalities follow from Theorem~\ref{Theorem:LocalPiecesOverlap} and the fact that 
  $ u - z \in ( \Omega_1 - Z(X) ) \cap ( \Omega_2 - Z(X) ) $. 
 Hence $\tilde f_z(D)T_X$ is continuous in $u$. This implies that we can drop the limit and 
 $\tilde f_z(D) T_{X}(u) = \vpf_X( u -z)$.
\qedhere
  \end{asparaenum}
 \end{proof}

\begin{proof}[Proof of Corollary \ref{Corollary:DahmenMicchelliGeneral}]
 Let $u \in  \Lambda$. 
 Note that by formula \eqref{eq:DahmenMicchelli} and Theorem~\ref{Theorem:ModifiedBrionVergne}, for all $u\in \Lambda$
 \begin{align}
   \left(\sum_{z \in \Zcal_-(X)} B_X(z) \tilde f_z \right)(D) T_X(u) &=
   \sum_{ z \in \Zcal_-(X)} B_X(z)  \vpf_X ( u - z )
   = T_X(u).
 \end{align}

 So the actions of $F:=\left(\sum_{z \in \Zcal_-(X)} B_X(z) \tilde f_z \right)$ and $1\in \Pper(X)$ on 
 $T_X$ are the same. Hence $(F-1)(D) T_X = 0 $. We will now show that this implies $F=1$.
 
 One can choose a sublattice $\Lambda^0\subseteq \Lambda$ \st $F-1$ agrees with a polynomial  
 that is contained in $\Pcal_\CC(X)$ on each coset of $\Lambda^0$. 
 Let $p\in \Pcal_\CC(X)$ be one of these polynomials. 
 By assumption, $p(D)$ annihilates all local pieces of $T_X$.
 Hence, by Theorem~\ref{Proposition:Dlocalpieces}, $p$ annihilates all of $\Dcal(X)$.
 It follows from the Duality Theorem (Theorem~\ref{Proposition:PDduality}) that $p=0$.
 Since $(F-1)$ restricted to an arbitrary coset of $\Lambda^0$ is $0$,  $F=1$.
\end{proof}

\begin{Remark}
 The space $\Pper(X)$ is inclusion-maximal with the following property:
 for every $0 \neq p\in \Pper_\R(X)$,
 the differential operator $p(D)$ defines a map $\Dcal_\R(X)\to \DM_\R(X)$ that does not annihilate
  $\Dcal_\R(X)$. In particular, $p(D)$ does not annihilate $T_X$.
 Hence $\Pper(X)$ can be seen as the space of 
 relevant differential operators on $T_X$ with periodic coefficients 
  (cf.~Remark~\ref{Remark:Pspacerelevant}).
\end{Remark}

\begin{Remark}
 \label{Remark:NoBoxGeneralisation}
  Theorem~\ref{ML-Theorem:weakHoltzRon} has no obvious generalisation to the general case.
  Consider the list $X=(1,a)$ for $a\in \Z$ and $a\ge 2$. Then $B_X|_{[1,a]} = \frac 1a$ 
  and the function is
  linear with slope $\pm\frac 1a$ on $[0,1]$ and $[a,a+1]$ and constant on $[1,a]$.
  \begin{equation}
   \Pper(1,a) = \spa\{ 1, s, e^{2 \pi i \lambda/ a }s, \ldots, e^{2 \pi i  (a-1) \lambda/ a } s \}.    
  \end{equation}
  The space $\Pper(1,a)$ is $a+1$ dimensional, 
  but all but one basis element ($1$) send $B_X$ to a function that is 
  zero everywhere on $\{1,\ldots, a-1\}$ except in one point 
  (which one depends on whether we use a limit 
  from the left or the right). Hence there is no 
  subspace of $\Pper(X)$ that contains unique interpolants.
  
  There are however more complicated operators that are inverse to the box spline.
  The following statement is contained in \cite[Theorem~2.29]{deconcini-procesi-vergne-2013}:
  let $w \in \cone(X)$ be a short affine regular vector. Then
  $\lim_w \toddperbox(X)(\Dpw) B_X  = \delta_0$,
  \begin{align}
  \text{where } \toddperbox(X)  &:=
   \sum_{e_\phi \in \Vcal(X)} 
    e_{\phi} \prod_{x \in X} \frac{ p_x }{ 1 - e_\phi(-x) e^{-p_x}}
    \prod_{x\in X\setminus X_\phi}
       \frac{ 1 - e_\phi(-x)\tau_x}{ 1 - \tau_x }. \notag
  \end{align}
  As usual, $\tau_x$ denotes the translation operator defined by $\tau_x(f):= f(\cdot -x)$. See also Example~\ref{Example:OneTwoFour}.
  \end{Remark}

\section{Results on periodic $\Pcal$-spaces and arithmetic matroids}
 \label{Section:ArithmeticMatroids}
 
 In this section we will define and study internal periodic $\Pcal$-spaces and prove further results on central periodic $\Pcal$-spaces.
 We will construct bases for these spaces and show that their Hilbert series are evaluations of the arithmetic Tutte polynomial.

 \subsection{Central periodic $\Pcal$-spaces}
 Let us first recall the connection between the zonotope $Z(X)$ and the arithmetic matroid defined by $X$.
 \begin{Proposition}[Corollary 3.4 in  \cite{moci-adderio-ehrhart-2012}]
 \label{Proposition:ZonotopeArithmeticTutte}
 \XintroLattice 
 Suppose that the fundamental region of $\Lambda$ has volume $1$.
 Then
 \begin{compactenum}
  \item the volume $\vol(Z(X))$ of the zonotope is equal to $\aritutte_X (1, 1)$ and
  \item
   the number $\abs{\Zcal_-(X)}$ of integer points in the interior of the zonotope is equal to $\aritutte_X (0, 1)$.  
 \end{compactenum}
 \end{Proposition}
 We will later see that the dimension of the central periodic $\Pcal$-space is equal to $\vol(Z(X))$
 and that the dimension of the internal periodic $\Pcal$-space is equal to  $\abs{\Zcal_-(X)}$.

 It will be useful to have a definition of the space $\Pper(X)$ in the case where the list $X$ is contained in a finitely generated abelian group.
Let $G$ be a finitely generated abelian group and let $X\subseteq G$.
For $y\in X$, we define  $p_y:= y \otimes 1 \in G\otimes \R = U \subseteq  \sym(U)$. %
Then define $p_Y := \prod_{y \in Y} p_y$
 and $\Pcal(X):=  \spa\{ p_Y : Y\subseteq X,\, X\setminus Y \text{ generates a subgroup of finite index} \}$ as in \eqref{equation:CentralP}.

    Let $X_t := X \cap G_t$ be the sublist of $X$ that contains all the torsion elements. 
Note that if $x\in X_t$ then $x\otimes 1 = 0 \in U$. Hence adding or removing torsion elements from $X$ leaves $\Pcal(X)$ unchanged.
  The same is true for $\Vcal(X)$.

  Note that in Definition~\ref{Definition:PeriodicCentralP} there are factors of type $p_{X\setminus X_\phi}$.
  We do not want these to vanish if $X\setminus X_\phi$ contains torsion elements and we want these factors to have degree $\abs{X\setminus X_\phi}$.
Therefore, we add a new variable $s_0$ that keeps track of the torsion elements.

\begin{Definition}
 \label{Definition:PeriodicCentralPGeneral}
 \XintroAbelianGroup
We define the \emph{central periodic $\Pcal$-space}
\begin{align}
\label{eq:centralperiodicPdefinition}
  \Pper(X) &:= 
  \bigoplus_{e_\phi\in \Vcal(X)} e_\phi p_{X\setminus (X_\phi \cup X_t)} s_0^{\tors(\phi) } \Pcal(X_\phi) %
   \subseteq \bigoplus_{e_\phi\in \Vcal(X)} e_\phi \R[s_0] \otimes \sym(U),
\end{align}
where $\tors(\phi) = \tors_X(\phi) := \abs{X_t\setminus X_\phi}$.
  \end{Definition}

The central periodic $\Pcal$-space has both a homogeneous 'matroid-theoretic' basis and an inhomogeneous basis.
The following two results generalise
Proposition~\ref{Proposition:Pbasis} and
Corollary~\ref{Corollary:newPbasis}.

 Recall that $X_t$ denotes the sublist of $X$ that contains all torsion elements.
\begin{Proposition}[Homogeneous basis]
 \label{Proposition:HomogeneousBasisCentralPeriodic}
 \XintroAbelianGroup
  Then the set
   $\Bper(X) := \{  e_\phi  s_0^{\tors(\phi) } p_{ X \setminus (B \cup ( E(B) \cap X_\phi) \cup X_t)} : e_\phi\in 
   \Vcal(X), \, B \in \BB(X_\phi)   \}$
  is a homogeneous basis for $\Pper(X)$.
   Here, $E(B)$ denotes the set of externally active elements in $X$ with respect to the basis $B$.
 \end{Proposition}
\begin{proof}[Proof of Proposition~\ref{Proposition:HomogeneousBasisCentralPeriodic}]
 Use Proposition~\ref{Proposition:Pbasis} for each of the direct summands in \eqref{eq:centralperiodicPdefinition} and note that
 $p_{X\setminus (X_\phi \cup X_t)}  p_{X_\phi \setminus (B\cup E(B))} =  p_{ X \setminus (B \cup ( E(B) \cap X_\phi) \cup X_t)}$.
\end{proof}

Note that there is a natural decomposition 
$\Pper(X) = \bigoplus_{ i \ge 0} \bigoplus_{e_\phi \in \Vcal(X)} e_{\phi} P_{i,\phi}$, where
 each of the spaces $P_{i,\phi}\subseteq \R[s_0] \otimes \sym(U)$ contains only homogeneous polynomials of degree $i$.
 This allows us to define the Hilbert series $\hilb( \Pper(X), q) = \sum_{i\ge 0} 
 \left(\bigoplus_{e_\phi \in \Vcal(X)} \dim P_{i,\phi} \right) q^i$.
The following theorem and Theorem~\ref{Proposition:PeriodicInternalTutteEval} below generalise Theorem~\ref{Proposition:HilbertSeriesTuttePolynomialP}.
 \begin{Theorem}
 \label{Proposition:CentralPeriodicPHilbertTutte}
  \XintroAbelianGroupN
 Then
  \begin{equation}
    \hilb(\Pper(X) , q) = q^{ N  - d } \aritutte_{X }(1, q^{-1}).
    \end{equation}
    In particular, if $X$ is contained in a lattice $\Lambda$ whose fundamental region has volume $1$, then the dimension of $\Pper(X)$ is equal to the volume of the zonotope $Z(X)$.
 \end{Theorem}
\begin{proof}[Proof of Theorem~\ref{Proposition:CentralPeriodicPHilbertTutte}]

It is known that  $\aritutte_X(1, \beta) = \sum_{\phi \in \Vcal(X)} \tutte_{X_\phi}(1, \beta)$.
 This is Lemma~6.1 in \cite{moci-tutte-2012}. 
 Note that in this equation, $\tutte_{X_\phi}$ denotes the Tutte polynomial of the matroid defined by $X_\phi$, \ie the torsion part of all elements of $X$ is ignored 
 and elements of $G_t$ count as loops.
Hence 
 using Theorem~\ref{Proposition:HilbertSeriesTuttePolynomialP} we obtain
\begin{align*}
  \hilb( \Pper(X), q) &= \sum_{ \phi \in \Vcal(X)  } q^{ \abs{X\setminus X_\phi}}
  \hilb(\Pcal(X_\phi),q) =  \sum_{ \phi \in \Vcal(X)  }  q^{\abs{X\setminus X_\phi} +  \abs{X_\phi} - d} \tutte_{X_\phi}(1, q^{-1})
  \\
  &=  q^{N-d} \aritutte_{X}(1, q^{-1}). \qedhere
\end{align*}
\end{proof}
 \begin{Proposition}[Inhomogeneous basis]
 \label{Proposition:InhomogeneousBasisCentralPeriodic}
  \XintroLattice
 Let $w$ be a short affine regular vector. Then
  $\{ \tilde f_z : z\in \Zcal(X,w) \}$ is a basis for $\Pper(X)$.
 \end{Proposition}
\begin{proof}[Proof of Proposition~\ref{Proposition:InhomogeneousBasisCentralPeriodic}]
By definition, each $\tilde f_z$ is contained in $\Pper(X)$.
It is known that $\abs{\Zcal(X,w)} = \vol(Z(X)$ (\eg Proposition 13.3 in \cite{concini-procesi-book}).
Hence it follows from Proposition~\ref{Proposition:ZonotopeArithmeticTutte} and Theorem~\ref{Proposition:CentralPeriodicPHilbertTutte}
that $\dim \Pper(X)= \abs{\Zcal(X,w)}$.

Note that the real vector space of all functions $\{ f : \Zcal(X,w) \to \R \} $ is equal to
\begin{align}
   \spa\{ \lim_w \tilde f_z(\Dpw) T_X|_{\Zcal(X,w)} : z \in \Zcal(X,w) \}.
\end{align}
This follows from the fact that for $z\in \Zcal(X,w)$, the support of $\vpf_X(\cdot - z)|_{\Zcal(X,w)} = \lim_w \tilde f_z(\Dpw) T_X|_{\Zcal(X,w)}$
is contained $\cone(X)+z $ and this function assumes the value one at $z$ (cf.~Theorem~\ref{Theorem:ModifiedBrionVergne}).
We can deduce that the set $ \{ \tilde f_z: z \in \Zcal_-(X, w) \} $ is linearly independent.
\end{proof}

\subsection{Internal periodic $\Pcal$-spaces}

The elements of $\Pper(X)$ can be thought of as functions that assign to each $g\in G$ a polynomial
 in $\sym_\CC(U)$. For $p\in \Pper(X)$, we will write $p(g,\cdot)$
to denote this ``local'' part of $p$.

\begin{Definition}[internal periodic $\Pcal$-space]
\label{Definition:internalPeriodicPsimple}
\XintroLattice
Then we define the  \emph{internal periodic $\Pcal$-space}
\begin{equation}
\label{Def:internalperiodicPlattice}
 \Pper_-(X) := \{  p  \in  \Pper(X)  :   D_{\eta_H}^{  m(H)  -  1  }  p(\lambda, \cdot)  =  0 \text{ for all }  H \in \Hcal(X) \text{ and  all }  \lambda \in H \cap \Lambda \}
\end{equation}
where $\Hcal(X)$ denotes the set of all hyperplanes $H$ that are spanned by a sublist of $X$.
  For $H\in \Hcal(X)$,
 $\eta_H \in V$ denotes a normal vector and $m(H):= \abs{X\setminus H}$.
\end{Definition}

 \begin{Example}
 \label{Example:InternalSpaceColoops}
  Let $X= ((2,0),(0,2))\subseteq \Z^2$. The set of vertices of the toric arrangement $\Vcal(X)$ consists of 
  the four maps that send $ (a,b) \in \Z^2$ to 
  $1$, $(-1)^a$, $(-1)^b$, and $(-1)^{a+b}$, respectively. $\Pcal(X)= \bigoplus_{e_\phi\in \Vcal(X)} e_\phi \R$.
  The ``differential'' equations for $\Pper_-(X)$ are $p(0,\cdot) = p((1,0),\cdot) = p((0,1),\cdot)=0$.   
  Hence $ \Pper_-(X) = \spa\{  1 -  (-1)^a - (-1)^b    +  (-1)^{a+b}   \} $.
 \end{Example}

In some proofs, we will require a more general definition, where the list $X$ is contained in a finitely generated abelian group $G$. 
Before making this definition, we have to generalise the notion of a hyperplane.

Recall that we have associated with $G$ a vector space $U=G\otimes \R$
  and a lattice $\Lambda= G\otimes \R \cong G/G_t$, where $G_t\subseteq G$ denotes the torsion subgroup.
  $X\otimes 1$ denotes the image of $X$ under the projection $G\twoheadrightarrow \Lambda$.
  Then we define 
 $\clos_G(Y):= (\spa(X\otimes 1) \cap \Lambda)\times G_t$.
 Note that the isomorphism $\Lambda \oplus G_t \cong G$ is not canonical and that the image of $\lambda\in \Lambda$ in $G$ can vary by a torsion element under different isomorphisms.
 However, $\clos_G(Y)$ can be seen as an element of $G$ in a canonical way. 
We define the set of \emph{generalised hyperplanes} as 
 \begin{equation}
 \label{eq:generalhyperplanes}
 \Hcal(X) := \{ \clos_G(Y) : Y\subseteq X,\, \rank(Y) = d - 1 \}.
 \end{equation}
 
%

 Let $H\in \Hcal(X)$.
 As before, we define $m(H) := \abs{X\setminus H}$ and 
  $\eta_H \in V$ denotes a normal vector for the hyperplane $\spa(H\otimes 1)\subseteq U$.
  Note that  $D_\eta$ acts on $\R[s_0] \otimes\sym(U)$ in the natural way, we just ignore the $s_0$ when differentiating.
 
\begin{Definition}[internal periodic $\Pcal$-space, general definition]
\label{Definition:internalPeriodicP}
\XintroAbelianGroup
Then we define the \emph{internal periodic $\Pcal$-space} as the space 
 \begin{align*}
   \Pper_-(X) := \{  p \in \Pper(X) :  D_{\eta_H}^{m(H)-1}  p(g, \cdot) = 0 \text{ for all } H \in \Hcal(X) \text{ and for all } g \in H  \},
   \label{eq:internalDiffEqs}
 \end{align*}
 where $\Hcal(X)$ denotes the set of generalised hyperplanes as defined in \eqref{eq:generalhyperplanes}.
\end{Definition}
%
%
%

 \begin{Example}
 \label{Example:internal}
 Let $ X = (( 2, \bar 0 ))  \subseteq \Z\oplus \Z/2\Z$.
  Then $\Hcal(X)= \{ \{(0,\bar 0),(0,\bar 1)\}$ and
  $\Pper_-(X) = \spa \{ 1 - (-1)^a, (-1)^{ \bar b } - (-1)^{ a + \bar b } \}$.
 \end{Example}

  \begin{Example}
  \label{Example:internalzwo}
  Let $( (2, \bar 0), (0,\bar 1) ) = X \subseteq \Z\oplus \Z/2\Z$.
  Then $\Hcal(X)= \{ \{(0,\bar 0),(0,\bar 1)\}$ and
  $\Pper(X) = \spa\{ 1,  (-1)^a, (-1)^{\bar b} s_0, (-1)^{a+\bar b} s_0 \},$
  $\Pper_-(X) = \spa \{ 1 - (-1)^a, (-1)^{ \bar b }s_0  - (-1)^{ a + \bar b } s_0 \}$.
  $\aritutte_X(\alpha,\beta) =  2 (\alpha-1) + 4 + (\alpha-1)(\beta-1) + 2(\beta-1) = \alpha\beta + \alpha + \beta + 1$.
 \end{Example}

The following result is the periodic analogue of 
Theorem~\ref{ML-Corollary:internalPcontinuous}.
\begin{Theorem}
\label{Proposition:internalContinuousCharacterisation}
\XintroLattice
Then
   \begin{align}
    \Pper_-(X) &= \{ p\in \Pper(X) %
      : p(D) T_X \text{ is continuous in } \Lambda \}.
      \end{align}
\end{Theorem}

  \begin{Theorem}
  \label{Proposition:PeriodicInternalTutteEval}
  \XintroAbelianGroupN
  Then
  \begin{equation}
    \hilb(\Pper_-(X) , q) = q^{ N - d} \aritutte_{X}(0, q^{-1}).
    \end{equation}
    In particular, if $X$ is contained in a lattice, the dimension of $\Pper_-(X)$ is equal to the number of interior lattice points of the zonotope $Z(X)$.
 \end{Theorem}

Here is a generalisation of Corollary~\ref{Corollary:InternalPbasis}.
\begin{Proposition}[Inhomogeneous basis]
\label{Proposition:InhomogeneousInternalBasis}
 \XintroLattice
 Then $\{ \tilde f_z : z \in \Zcal_-(X) \}$ is a basis for 
$\Pper_-(X) $.
 \end{Proposition} 
 
\begin{Remark}
 In contrast to the central periodic space, the internal periodic space in general does not have a decomposition $\Pper_-(X) =
  \bigoplus_{\phi} e_\phi P_\phi$ for some $P_\phi\subseteq \R[s_0] \otimes \sym(U)$ 
  (\eg Example~\ref{Example:InternalSpaceColoops}). This and the fact that we do not have   statement analogous to Proposition~\ref{Proposition:HomogeneousBasisCentralPeriodic}
  make it a lot more difficult to handle this space.

Therefore, the proofs of the results in this subsection are considerably longer than the ones in the previous subsection.
For the proof of Theorem~\ref{Proposition:internalContinuousCharacterisation} we will use a residue 
formula for the jump of the multivariate spline across a wall that is due to Boysal--Vergne (see Section~\ref{Section:WallCrossingProofs}).
Theorem~\ref{Proposition:PeriodicInternalTutteEval} requires the most work.
We will prove it inductively using the exact sequence in Proposition~\ref{Proposition:DeletionContractionInternalPeriodic} below.
In its proof, we will use the ``$\subseteq$''-part of Proposition~\ref{Proposition:InhomogeneousInternalBasis} that is fairly simple (Lemma~\ref{Lemma:fzInternalIndependent}).
The rest of Proposition~\ref{Proposition:InhomogeneousInternalBasis} will then follow via a dimension argument (see Section~\ref{Section:DeletionContraction}).
\end{Remark}

 \begin{Remark}
   The definition of the internal periodic $\Pcal$-space was inspired by the representation
   of the $\Pcal$-spaces as
    an inverse systems of  power ideals given in \cite{ardila-postnikov-2009, holtz-ron-2011}.

    If $X\subseteq \Lambda$ is unimodular, then $\Pper_-(X)=\Pcal_-(X)$ and
  the description of this space in \eqref{Def:internalperiodicPlattice} is the same as the description of  
  the internal $\Pcal$-space as  an inverse system (or kernel) of
  a power ideal  in these two papers.
 \end{Remark}

\begin{Remark}
For each $g\in G$ there is a ``local'' version of the periodic $\Pcal$-spaces at $g$,
\ie $\Pper(X)_g:= \{ p(g,\cdot)  : p \in \Pper(X)\}\subseteq\sym(U_\CC)$ and
$\Pper_-(X)_g:= \{ p(g,\cdot)  : p \in \Pper_-(X)\}$.
It is easy to see that $\Pper(X)_g \subseteq \Pcal_\CC(X)$ for any $g \in G$.
The space $\Pper_-(X)_g$ is a \emph{semi-internal}
space in the sense of \cite{holtz-ron-xu-2012,lenz-hzpi-2012}. However, 
in general $\Pper_-(X)_g$ is not equal to one of the specific types of semi-internal
spaces that were studied in these two papers.
\end{Remark}

%% file: SLPAM_duality.tex
\section{Duality between $\DM(X)$ and $\Pper(X)$} 
 \label{Section:PeriodicDuality}
 
 \subsection{Overview}
 The goal of this section is to prove that $\Pper_\CC(X)$ and $\DM_\CC(X)$ are dual in analogy with 
  Theorem~\ref{Proposition:PDduality}. 
  We will first   define a pairing that induces this duality. If $X$ is unimodular, this pairing agrees with the one defined in \eqref{eq:pairing}.
  Then we will show that  $\Pper_\CC(X)$  is canonically isomorphic to $\CC[\Lambda]/\dJC(X)$.
  We will   see that one can also obtain the pairing using this isomorphism and a canonical pairing  $\CC[\Lambda] / \dJC(X) \times \DM_\CC(X) \to \CC$.

\begin{Definition}
\XintroLattice 
 Let $p= \sum_{e_\phi\in \Vcal(X)} e_\phi p_{X\setminus X_\phi} p_\phi \in \Pper_\CC(X)$ and $f= \sum_{e_\phi\in \Vcal(X)} e_\phi f_\phi \in \DM_\CC(X)$.
 Then we define
 \begin{equation}
  \discpairP{p}{f} :=  \sum_{e_\phi \in \Vcal(X)} \pair{p_\phi}{ f_\phi}.  
 \end{equation}
\end{Definition}
Note that \textit{a priori}, the function $f$ above is not a polynomial, but a function $\Lambda\to \CC$.
However, by Proposition~\ref{Proposition:DMdecomposition},
 the functions $f_\phi$ are all restrictions to $\Lambda$ of polynomials in $\Dcal_\CC(X_\phi)$.
 Therefore, we can identify them in a unique way with polynomials
in $\Dcal_\CC(X_\phi)$.

 \begin{Theorem}
 \label{Proposition:PperDMduality}
 \XintroLattice
 Then the  spaces $\Pper_\CC(X)$ and $\DM_\CC(X)$ are dual under the pairing $\discpairP{\cdot}{\cdot}$, \ie the map
 \begin{equation}
 \label{eq:PperDMduality}
 \begin{split}
  \DM_\CC(X) &\to \Pper_\CC(X)^* \\
  f & \mapsto \discpairP{\cdot}{f}
 \end{split}
  \end{equation}
is an isomorphism.
\end{Theorem}
\begin{proof}
 It follows from the definition that $\discpairP{e_\phi p}{e_\varphi f}= \pair{p}{0} +\pair{0}{f} =  0$ for 
 $e_\phi \neq e_\varphi$, $p\in p_{X\setminus X_\phi} \Pcal_\CC(X_\phi)$, and $f\in \Dcal_\CC(X_\phi)$.
 The statement can then easily be deduced from Theorem~\ref{Proposition:PDduality}, taking into account Proposition~\ref{Proposition:DMdecomposition}.
\end{proof}

  There is a natural pairing
 $\discpair{}{} : \Z[\Lambda]/\dJ(X) \times \DM(X) \to \Z$ defined by $\discpair{\lambda}{f}:= f(\lambda)$ for $\lambda\in \Lambda$.
 This pairing can be extended   to a pairing   $\discpair{}{} : \CC[\Lambda]/\dJC(X) \times \DM_\CC(X) \to \CC$.

\begin{Theorem}
\label{Proposition:DMquotientduality}
  \XintroLattice
 Let $w \in U$ be an affine regular vector.
  Then the set $\{  \bar\lambda : \lambda \in \Zcal(X,w) \} \subseteq \CC[\Lambda]/\dJC(X) $ is a basis for the vector space
  $\CC[\Lambda]/\dJC(X) $.

 Furthermore, the pairing $\discpair{\cdot}{\cdot}$
 induces a  duality between the two spaces,
  \ie the map $\DM_\CC(X) \to  (\CC[\Lambda]/\dJC(X))^*$, $f\mapsto \discpair{\cdot}{f}$ is an isomorphism.
\end{Theorem}
\begin{proof}
 This follows from Proposition~13.16 and Theorem~13.19 in \cite{concini-procesi-book}.
\end{proof}

\begin{Theorem}
\label{Theorem:pairingisomorphism}
\XintroLattice
 There exists a canonical isomorphism $L : \Pper_\CC(X) \to \CC[\Lambda]/\dJC(X)$ \st
 for $p\in \Pper_\CC(X)$ and $f\in \DM_\CC(X)$,
 $\discpairP{p}{f}=\discpair{L(p)}{f}$.
\end{Theorem}
 \begin{Corollary}
 \label{Corollary:PbasisExponentials}
  \XintroLatticeUnimod
  Let $w$ be an affine regular vector. 
  Then $\{ \psi_X(e^z): z\in \Zcal(X,w)    \}$ 
  is a basis for $\Pcal(X)$.
  
  Furthermore,  $L(\psi_X(e^z))=z$
   and this 
 induces a bijection 
 between this  basis and the basis in 
 Theorem~\ref{Proposition:DMquotientduality}.
 \end{Corollary}
 \begin{Remark}[K-Theory]
  Recent work of De~Concini--Procesi--Vergne and Cavazzani--Moci relates the zonotopal spaces studied in this paper
  with geometry.
 
 $\Dcal(X)$ and $\sym(U)/\cJ(X) \cong \Pcal(X)$ can be realised as equivariant cohomology of certain differentiable manifolds
 and $\DM(X)$ and can be realised as equivariant $K$-theory \cite{deconcini-procesi-vergne-2010b, deconcini-procesi-vergne-2011, deconcini-procesi-vergne-infinitesimal-2013}.

 The space  $\Z[\Lambda]/ \dJ(X)$ can also be realised as equivariant $K$-theory of a certain manifold  \cite[Theorem 5.4]{cavazzani-moci-2013}.
 The  complexification of this space is by  Theorem~\ref{Theorem:pairingisomorphism} isomorphic $\Pper_\CC(X)$.
 \end{Remark}

\subsection{The details}
 The construction of the map $L$ in Theorem~\ref{Theorem:pairingisomorphism} requires a few concepts from commutative algebra that we 
 will now recall.
 
 In this section we will work with the algebraic torus $T_\CC(\Lambda)= \hom(\Lambda, \CC^*  )$, which will allow us to use algebraic techniques such as primary decomposition.
 Recall that the algebraic torus 
 is an algebraic variety that is isomorphic to $ \{ (\alpha_1,\beta_1,\ldots, \alpha_d,\beta_d ) \in \CC^{2d} : \alpha_i \beta_i = 1 \} $.
 Its  coordinate ring is $\CC[\Lambda]$.
 Let $T_\CC(\Lambda)\ni e_\phi : \Lambda \to \CC^*$ and $f=\sum_{\lambda\in \Lambda} \nu_\lambda \lambda \in \CC[\Lambda]$.
 Then $f(e_\phi) := \sum_{\lambda\in \Lambda} \nu_\lambda e_\phi(\lambda)$.
 The choice of a basis $s_1,\ldots, s_d$ for $\Lambda$ induces isomorphisms $T_\CC(\Lambda)\cong (\CC^*)^d$ via $e_\phi \mapsto (e_\phi(s_1),\ldots,
  e_\phi(s_d))$ and  $\CC[\Lambda] \cong \CC[ a_1^{\pm 1}, \ldots, a_d^{\pm 1}]$ via $\Lambda \ni \sum_{i=1}^d \nu_i s_i \mapsto \prod_{i=1}^d  a_i^{\nu_i}
  \in \CC[ a_1^{\pm 1}, \ldots, a_d^{\pm 1}]$.
  Under this identification $f(e_\phi)$ is equal to the evaluation of the Laurent polynomial 
  $f\in \CC[ a_1^{\pm 1}, \ldots, a_d^{\pm 1}]$ at the point $(e_\phi(s_1),\ldots, e_\phi(s_d)) \in (\CC^*)^d$.
 
 As usual, the \emph{subvariety} defined by an ideal $I\subseteq \CC[\Lambda]$ is the set 
 $\var(I) := \{ e_\phi  \in T_\CC(\Lambda) : f(e_\phi)=0 \text{ for all } f\in I \} $. 
 Recall that an ideal $I\subseteq \CC[\Lambda]$ is \emph{zero-dimensional} if one of the 
 following two equivalent conditions is satisfied:
 $\CC[\Lambda]/I$ is finite-dimensional or the variety $\var(I)\subseteq T_\CC(\Lambda)$ is a finite set. 

\begin{Lemma}
\label{Lemma:VarietyVertices}
\XintroLattice
The ideal $\dJC(X)\subseteq \CC[\Lambda]$ defines a zero-dimensional subvariety of $T_\CC(\Lambda)$ that coincides with 
 the set  of vertices of the toric arrangement $\Vcal(X)$.
\end{Lemma}
\begin{proof}
Let $f,g\in \CC[\Lambda]$ and $e_\phi\in T_\CC(\Lambda)$. It is important to note that
$(fg)(e_\phi) = f(e_\phi)g(e_\phi)$.
 Then it is immediately clear that $\Vcal(X)\subseteq \var(\dJ(X))$: a vertex of the toric arrangement is annihilated by 
 some basis and every cocircuit intersects this basis. 

 Now let $e_\phi \in \var(\dJ(X)) \subseteq T_\CC(\Lambda)$.
 Hence $\nabla_C(e_\phi)= 0$ for all cocircuits $C\subseteq X$.
 Since $\CC$ is an integral domain, this implies that $e_\phi$ annihilates
 at least one factor of each cocircuit. Let $Y\subseteq X$ be the list of elements that are annihilated by $e_\phi$. Suppose that $Y$ is contained in some hyperplane $H$.
 Then $e_\phi$ does not annihilate an element of the cocircuit $X\setminus H$. This is a contradiction.
 Hence $e_\phi$ annihilates a basis $B$. This basis defines a vertex of the toric arrangement.
\end{proof}

  \begin{Theorem}[Chinese remainder theorem, \eg {\cite[Exercise 2.6]{eisenbud-1995}}]
  Let $R$ be a commutative ring, and let $Q_1,\ldots, Q_m \subseteq R$ be ideals  \st $Q_i + Q_j = R$ for all $i\neq j$. 
  Then $R/\bigcap_i Q_i \cong \prod_{i=1}^m R/Q_i$. The isomorphism is given by the product of the $m$ canonical projection maps.
 \end{Theorem}
 
 The following related result follows from \cite[Exercise 4.\S2.11]{cox-little-oshea-2005} (see also
 {\cite[Theorem 2.13 and Chapter 3]{eisenbud-1995}}).
 \begin{Theorem}[Primary decomposition]
 \label{Theorem:primdec}
   Let $J \subseteq \CC[\Lambda]$ be a zero-dimensional ideal with $\var(J) = \{ p_1,\ldots, p_m\}$.

   Let $Q_i = \{f\in \CC[\Lambda] : \text{ there exists } u \in \CC[\Lambda],\, u(p_i) \neq 0 \text{  \st  } u f \in I   \}$.
   Then $ J =Q_1\cap \ldots \cap Q_m$ is the primary decomposition, so in particular, $\var(Q_i)=\{p_i\}$ and
   $ \CC[\Lambda] / I \cong \CC[\Lambda] / Q_1 \times \ldots \times \CC[\Lambda] / Q_m$.
   \end{Theorem}
 
 \smallskip
 Let $J\subseteq \CC[\Lambda] $ be an ideal \st $\var(J)$ contains the point $e_\phi$.
 We say that $\theta \in V_\CC=U_\CC^*$ represents $e_\phi$ if
 $e^{2\pi \theta(\lambda)} = e_\phi(\lambda)$ for all $\lambda\in \Lambda$.
 Note that the  map $\theta$ is not uniquely determined by this condition.
 
Let $\lambda \in \Lambda$. Note that $e_\phi(-\lambda) \lambda - 1 \in \CC[\Lambda]$ vanishes at the point $e_\phi$. 
By Hilbert's Nullstellensatz, this  implies that $e_\phi(-\lambda) \lambda \in \sqrt{J}$, or put differently,
 $t_\lambda := e_{\phi}(-\lambda) \lambda - 1 \in \CC[\Lambda]/J$ is nilpotent. 
This implies that the following term is a finite sum: $\log(1+ t_\lambda) := t_\lambda - t_\lambda^2/2 + t_\lambda^3/3 - \ldots $

 Let $i^\theta : \Lambda \to \CC[ \Lambda ] / J$ be the map given by
  $i^\theta(\lambda) := \log(1+ t_\lambda) ) + \theta(\lambda)$.  
This map is additive since $1 + t_{\lambda_1 + \lambda_2} = (1+t_{\lambda_1})(1+t_{\lambda_1})$.
 It can be extended to a map $i^\theta : \sym(U) \to \CC[\Lambda]/J$.
 The following result follows from Proposition~5.23 in \cite{concini-procesi-book}.
 \begin{Proposition}
 \label{Proposition:LogIso}
  Let $J\subseteq \CC[\Lambda] $ be an ideal \st $\var(J)$ contains a unique point $e_\phi$. Let $\theta\in V_\CC$ be a map that represents $e_\phi$.
  Let $i^\theta : \sym(U) \to \CC[\Lambda]/J$ be the map defined above 
  and let $I := \ker( i^\theta ) $. 
  Then $ i^\theta$ induces an isomorphism $i_{\log}^\phi : \sym(U)/I \to \CC[\Lambda]/J$ and $\var(I) = \{\theta\}$.
    
 We will call the map $ i_{\log}^\theta$  the \emph{logarithmic isomorphism}.
  \end{Proposition}

 By Lemma~\ref{Lemma:VarietyVertices},
 the ideal $\dJC(X)$ defines a zero-dimensional subvariety of $T_\CC(\Lambda)$, the set of vertices of the toric arrangement $\Vcal(X)$.
 Hence by Theorem~\ref{Theorem:primdec} there is a  decomposition $\CC[\Lambda] / \dJC(X) \cong \bigoplus_{ e_\phi \in \Vcal(X)} \CC[\Lambda] / \dJC(X)_\phi$.
 Note that while we have explicit descriptions of 
 $\Pper_\CC(X)$,  $\dJC(X)$,  $\cJC(X)$, and to a certain extent also of $\DM_\CC(X)$,
 we do not know an explicit description of the ideals  $\dJC(X)_\phi$ appearing in this decomposition.
 We will however see that quotients of these ideals are isomorphic to quotients of the following ideals.

  \begin{Definition}[Inhomogeneous cocircuit ideal]
   \XintroLattice Let $\theta\in V_\CC= U_\CC^*$. 
   We define the \emph{inhomogeneous continuous cocircuit ideal} %
   \begin{equation}
    \cJC(X,\theta) := \ideal \left\{  \prod_{x\in C} (p_x - \theta(x)) : C\subseteq X \text{ cocircuit} \right\}.
   \end{equation}
  \end{Definition}
  Note that $ \var(\cJC(X)) = \{0\} \subseteq V_\CC$ and  $\var(\cJC(X,\theta)) = \{ \theta\}\subseteq V_\CC$.
 Inhomogeneous cocircuit ideals first appeared implicitly in a paper by  Ben-Artzi and Ron on exponential box splines \cite{benartzi-ron-1988}.

 \begin{Lemma}
 \label{eq:logisomap}
 Let $e_\phi\in \Vcal(X)$. Let $\theta\in V_\CC$ be a representative of $e_\phi$, \ie $e_\phi(\lambda)= e^{2\pi i \theta(\lambda)}$ for all $\lambda\in \Lambda$.
Then the logarithmic isomorphism defines an isomorphism
 $i_{\log}^\theta : \sym(U_\CC) / \cJC( X_\phi, \theta ) \to \CC[ \Lambda ] / \dJC(X)_\phi$.
 \end{Lemma}
\begin{proof}
Let us consider the map $i^\theta: \sym(U) \to \CC[ \Lambda ] / \dJC(X)_\phi$. 
Let $C\subseteq X_\phi$ be a cocircuit. Recall that for $x\in X_\phi$,  $t_x:= e_\phi(-x) x - 1 = x -1 \in \CC[\Lambda]/\dJC(X)$. Then
\begin{align}
 i^\theta\left(\prod_{x\in C} ( p_x - \theta(x) )\right) &=
 \prod_{x\in C}  ( t_x - \frac 12 t_x^2 + \frac 13 t_x^2 - \ldots) = f\prod_{x\in C} (x-1)
\end{align}
for some $f\in \CC[\Lambda]$. 
Let $Y\subseteq  X\setminus X_\phi$ be a set \st
$C\cup Y\subseteq X$ is a cocircuit.
Note that $\prod_{x\in Y}(x-1)$ does not vanish at $e_\phi$ and
$f\prod_{x\in C} (x-1) \prod_{x\in Y}(x-1)\in \dJC(X)_\phi$ by Theorem~\ref{Theorem:primdec}.
This implies that $\cJC(X_\phi,\theta)\subseteq \ker i^\theta$. 
Hence we have a canonical surjection $\sym(U_\CC)/ \cJC(X_\phi,\theta) \twoheadrightarrow \sym(U_\CC) / \ker i^\theta$
and $\dim \sym(U_\CC)/ \cJC(X_\phi,\theta) \ge \CC[ \Lambda ] / \dJC(X)_\phi$.

It is known that $\dim(\sym(U_\CC) /  \cJC(X_\phi, \theta)) = \dim (\sym(U_\CC) / \cJC(X_\phi))$ (\eg \cite[Proposition 11.16]{concini-procesi-book})
and that this number is equal to $\tutte_{X_\phi}(1,1)$ (Proposition~\ref{Proposition:JPdecomposition} and Theorem~\ref{Proposition:HilbertSeriesTuttePolynomialP}).
We obtain
\begin{align}
 \aritutte_X(1,1) &=
 \dim \CC[\Lambda] / \dJC(X) 
 = \sum_{e_\phi\in \Vcal(X)} \dim \CC[ \Lambda ] / \dJC(X)_\phi
 \\
&\le   \sum_{e_\phi\in \Vcal(X)} \dim \sym(U_\CC) / \cJC( X_\phi, \theta )
= \sum_{e_\phi\in \Vcal(X)} \tutte_{X_\phi}(1,1)  = \aritutte_X(1,1).
 \notag
 \end{align}
The first equality follows from Proposition~\ref{Proposition:dimensionDM} %
and Theorem~\ref{Proposition:DMquotientduality} and the 
second equality follows from the Chinese Remainder Theorem.
The last equality is \cite[Lemma 6.1]{moci-tutte-2012}.

Hence the canonical surjection must be an isomorphism and $\sym(U_\CC)/ \cJC(X_\phi,\theta)$ is equal to $\sym(U_\CC) / \ker i^\theta$.
Now the statement follows from Proposition~\ref{Proposition:LogIso}.
\end{proof}
 By the Chinese remainder theorem, the map
 \begin{align}
   \alpha : \CC[\Lambda] / \dJC(X) \to \bigoplus_{e_\phi \in \Vcal(X)} \CC[\Lambda] / \dJC(X)_\phi  
 \end{align}
that sends $ f\in \CC[\Lambda] / \dJC(X)$ to $(\pi^\phi(f))_{e_\phi \in \Vcal(X)}$ is an isomorphism  ($\pi^\phi$ denotes the canonical projection).
Hence for each $e_\phi \in \Vcal(X)$, there 
  exists a uniquely determined map 
 $\kappa^\phi : \CC[\Lambda] / \dJC(X)_\phi \to \CC[\Lambda] / \dJC(X)$  
 \st $  \pi^\phi \circ \kappa^\phi = \text{id}$.
 Note that the inverse of the  map $\alpha$ is the isomorphism
 $\sum_{e_\phi \in \Vcal(X)} \kappa^\phi  : \bigoplus_{e_\phi \in \Vcal(X)} \CC[\Lambda]/ \dJC(X_\phi) \to \CC[\Lambda]/ \dJC(X)$.
 
 By Proposition~\ref{Proposition:JPdecomposition}, the map
 $j_\phi : \Pcal_\CC(X_\phi) \to \sym(U_\CC) / \cJC(X_\phi)$ that sends a polynomial $p$ to its class $\bar p$ is an isomorphism.
 We define $\tau_\theta :  \sym(U_\CC) / \cJC( X_\phi ) \to \sym(U_\CC) / \cJC( X_\phi, \theta )  $ 
 by  $\tau_\theta(p_x) :=  p_x- \theta(p_x)$.

 To summarise, we just defined four maps, the first three  are isomorphisms:
 \begin{equation}
 \begin{split}
    \Pcal_\CC(X_\phi) \stackrel{j_\phi}{\longrightarrow}   \sym(U_\CC) / \cJC( X_\phi ) \stackrel{\tau_\theta}{\longrightarrow}  \sym(U_\CC) / \cJC( X_\phi, \theta ) 
    &\stackrel{i^\theta_{\log}}{\longrightarrow}  
    \CC[\Lambda] / \dJC(X)_\phi   \\
    &\stackrel{\kappa^\phi}{\hookrightarrow}   
     \CC[\Lambda] / \dJC(X).
     \end{split}
 \end{equation}
 Note that the map $i_{\log}^\theta \circ \tau_\theta$ depends only on $X$ and $e_\phi$. It is independent of the choice of the representative $\theta$.  
Recall that every $p\in \Pper(X)$ can be written uniquely as $p = \sum_{e_\phi\in \Vcal(X)} e_\phi p_{X\setminus X_\phi} p_\phi$ with $p_\phi \in \Pcal(X_\phi)$.
We are now ready to define the  map $L: \Pper(X) \to \CC[\Lambda]/\dJC(X)$:
\begin{equation}
\label{eq:isomorphismL}
 L(p) := \sum_{e_\phi\in \Vcal(X)} \kappa^\phi ( i_{\log}^\theta ( \tau_\theta ( j_\phi (p_\phi)))).
\end{equation}

 \begin{Remark}
 \label{Remark:PairingEval}
 Here is an %
 an algorithm to calculate $L(p)$:
 \begin{compactenum}[(a)]
  \item Calculate the primary decomposition of $\dJC(X) = \bigcap_{e_\phi \in \Vcal(X)} \dJC(X)_\phi$.
  \label{enum:primdec}
  \item Decompose $p=\sum_{e_\phi \in \Vcal(X)} e_\phi p_{X\setminus X_\phi} p_\phi$.
  Then for each $e_\phi$, consider the class of $p_\phi \in \sym(U_\CC)/\cJC(X_\phi)$ 
  and apply   $\tau_\theta (i_{\log}^\phi) $ to it to obtain 
   $q_\phi \in \CC[\Lambda]/\dJC(X)_\phi$.
  \item Lift each $ q_\phi $ to an element $r_\phi \in \CC[\Lambda]/\dJC(X)$  \label{enum:lifting} using the map $\kappa^\phi$.
  Then $L(p) = \sum_{e_\phi \in \Vcal(X)} r_\phi$.
 \end{compactenum}
 Steps (\ref{enum:primdec}) and (\ref{enum:lifting}) are quite difficult to do by hand even for small examples, but they can 
 easily be done by a computer algebra.
 See Appendix~\ref{appendix:sagesingular} and Examples~\ref{Example:ZPpairing} and~\ref{Example:OneTwoFourIso} for more details.
\end{Remark}

\begin{Lemma}
\label{Lemma:KernelOflocalIdeal}
\XintroLattice Let $e_\phi\in \Vcal(X)$. Then
$\{  f\in \Ccal_\CC[\Lambda] : p(\nabla)f = 0 \text{ for all } p\in \dJC(X)_\phi  \} = e_\phi \Dcal_\CC(X_\phi)|_\Lambda$.
\end{Lemma}
\begin{proof}
Let $\theta\in V_\CC$ be a vector that represents $e_\phi$.
Then by \cite[Theorem~11.17]{concini-procesi-book}, $\Dcal(X,\theta):=\{ f \text{ distribution on $U$} : p(D) f = 0 \text{ for all } p \in \cJC(X, \theta) \}
= e_\theta \Dcal(X_\theta)$.

On the other hand, by Lemma~\ref{eq:logisomap} and \cite[Proposition~5.26]{concini-procesi-book} the space $\{  f\in \Ccal_\CC[\Lambda] : p(\nabla)f = 0 \text{ for all } p\in \dJC(X)_\phi  \}$
is equal to $\Dcal(X,\theta)|_\Lambda$.
\end{proof}

\begin{proof}[Proof of Theorem~\ref{Theorem:pairingisomorphism}]
Let $ \Pper(X) \ni p = \sum_{e_\phi\in \Vcal(X)} e_\phi p_{X\setminus X_\phi} p_\phi$ with $p_\phi \in \Pcal(X_\phi)$.
We have defined the isomorphism 
$L : \Pper_\CC(X)\to \CC[\Lambda]/\dJC(X)$ in 
\eqref{eq:isomorphismL} by
 $L(p) := \sum_{e_\phi\in \Vcal(X)} \kappa^\phi ( i_{\log}^\theta ( \tau_\theta ( j_\phi (p_\phi))))$.
 So all that remains to be shown is that $\discpairP{p}{f}=\discpair{L(p)}{f}$.
 As usual, we decompose $f\in \DM_\CC(X)$ as 
 $f= \sum_{e_\phi \in \Vcal(X)} e_\phi f_\phi$ with $f_\phi \in \Dcal_\CC(X_\phi)|_\Lambda$.

 First note that by definition, $\discpair{1-x}{f} = f(0) - f(x) = (\nabla_{-x} f)(0)$ and more generally,  for $Y\subseteq X$,
 \begin{equation}
\label{eq:discpairCocircuit}
 \discpair{\prod_{x\in Y} (1-x)}{f} = (\nabla_{-Y} f)(0).
 \end{equation}
 Let us fix a vertex $e_\phi\in \Vcal(X)$ and
 let  $ h_\phi := i_{\log}^\theta ( \tau_\theta ( j_\phi (  1 ))) \in \CC[\Lambda]/\dJC(X)$.
%
 Let $e_\phi \neq e_\varphi \in \Vcal(X)$.
 By the Chinese Remainder Theorem, $\pi^{\varphi} (h_\phi) = 0$. Hence 
 $h_\phi \in \dJC(X)_\varphi$. 
 By Lemma~\ref{Lemma:KernelOflocalIdeal} and \eqref{eq:discpairCocircuit}, this implies that 
 $\discpair{h_\phi}{e_\varphi f_\varphi} =0$.
  Now we have established that $\discpair{L(p)}{f} :=  \sum_{e_\phi \in \Vcal(X)} \discpair{L(e_\phi p_{X\setminus X_\phi} p_\phi)}{e_\phi f_\phi}$.
  
 On the other hand by the Chinese Remainder Theorem, $\pi^{\phi} (h_\phi) = 1$. Hence 
  $h_\phi = 1 + \gamma_\phi$ for some $\gamma_\phi\in \dJC(X)_\phi$.
 By Lemma~\ref{Lemma:KernelOflocalIdeal} and \eqref{eq:discpairCocircuit}, this implies $\discpair{h_\phi}{e_\phi f_\phi} = \discpair{1}{e_\phi f_\phi} = f_\phi(0)$.
 
 Let $x \in \Lambda$. Note that $\tau_{-x} = e^{D_x}$ as operators on $\sym(U)$, where $\tau_{-x}$ acts by translation and $e^{D_x}$ acts as a differential operator.
 This is equivalent to   $\log(\tau_{-x}) = D_x$  (cf.~\cite[equation (5.9)]{concini-procesi-book}).
  Furthermore,
 $( e_\phi( - x ) \tau_{-x} ) (e_\phi f_\phi)(u) = e_\phi(-x)e_\phi ( u + x)  f_\phi( u + x ) 
 =    ( e_\phi  (\tau_{-x} f_\phi)) (u)$.
This implies 
\begin{align}
 \log( e_\phi(-x)\tau_{-x}) )  (e_\phi f_\phi) &= e_\phi  \log( \tau_{-x})   f_\phi  = e_\phi D_x f_\phi. \quad  \text{Hence }
 \\
 \discpair{ L( e_\phi p_{X\setminus X_\phi} p_x ) }{e_\phi f_\phi} &= \discpair{ \kappa^\phi (i_{\log}^\theta ( \tau_\theta ( p_x )))}{ e_\phi f_\phi } = 
 (e_\phi  D_x f_\phi)(0)= (D_x f_\phi)(0)  \notag
\end{align}
  and more generally, for $Y\subseteq X_\phi$,
  \begin{equation*}
  \discpair{ L( e_\phi p_{X\setminus X_\phi} p_Y ) }{e_\phi f_\phi} = ( p_Y(D) f_\phi )(0) = \pair{p_Y}{f_\phi}.
  \qedhere
  \end{equation*}
\end{proof}

\begin{proof}[Proof of Corollary~\ref{Corollary:PbasisExponentials}]
Let $z\in \Zcal(X,w)$. 
Since $X$ is unimodular,  the toric arrangement has only one vertex. This implies
$ \discpairP{p}{f} =  \pair{p}{f}=(p(D)f)(0)$.
By Taylor's Theorem and Theorem~\ref{Theorem:pairingisomorphism}, 
$\discpair{L(\psi_X(e^z))}{f}= \pair{ \psi_X(e^z) }{ f } = e^{z}(D)f = f( z ) = \discpair{z}{f}$.
  Using  Theorem~\ref{Proposition:PperDMduality} we obtain that
 $L(\psi_X(e^z))= z$.

 Since the image of a basis under an isomorphism is also a basis, the set $\{ \psi_X(e^z): z\in \Zcal(X,w) \}\subseteq \Pcal(X)$ is 
 a basis by Theorem~\ref{Proposition:DMquotientduality}.
\end{proof}

%% file: SLPAM_proofs.tex
\section{Wall crossing and the proof of Theorem~\ref{Proposition:internalContinuousCharacterisation}}
\label{Section:WallCrossingProofs}

In this section we will prove Theorem~\ref{Proposition:internalContinuousCharacterisation}. In the proof we will use
the following  wall-crossing formula of Boysal--Vergne.
\begin{Theorem}[{\cite[Theorem 1.1]{boysal-vergne-2009}}]
\label{Theorem:BoysalVergneWallCrossing}
 \XintroLattice
 
 Let $\Omega_1$ and $\Omega_2$ be two big cells whose closures have a $(d-1)$-dimensional intersection.
 Let $H$ be the hyperplane that contains this intersection.
   The intersection is contained in the closure of a big cell $\Omega_{12}$ of $X\cap H \subseteq H$.
   Let $T_{X\cap H}^{\Omega_{12}}$ denote the polynomial that agrees with $T_{X\cap H}$ on $\Omega_{12}$.
  Let $V_{12}$ be a polynomial that extends the polynomial $T_{X\cap H}^{\Omega_{12}}$ to $U$
  (\eg $V_{12}|_H=T_{X\cap H}^{\Omega_{12}}$ and  $V_{12}$ constant on lines perpendicular to $H$).
 Let $\eta$ be a normal vector for $H$. 
 Suppose that $\eta(\Omega_1)>0$. 
 Then
 \begin{align}
 \label{eq:BoysalVergneWallCrossing}
   (T_X^{\Omega_1} - T_X^{\Omega_2})  &=
   \res_{z = 0} \left( \left( 
   V_{12}(D) \frac{ e^{ t_1 s_1 + \ldots + t_d s_d + \eta  z  } }{\prod_{x\in X\setminus H} ( t_1(x)s_1 + \ldots + t_d(x)s_d + \eta(x) z ) }
   \right)_{ s = 0 } \right).
 \end{align}
\end{Theorem}
As usual, $s_1,\ldots, s_d$ is a basis for the vector space $U$, $t_1,\ldots, t_d$ is a basis for the dual space, and
$V_{12}(D):=V_{12}\left( \diff{s_1},\ldots,  \diff{s_d}\right)$.
Hence $t_i(x)$ is a real number that depends only on $X\setminus H$.
The  term inside of $\res_{z = 0} \left(\cdot \right)$ on the right hand side of \eqref{eq:BoysalVergneWallCrossing} 
can be considered to be an element of the ring $\R[[t_1,\ldots, t_d, z,z^{-1}]]$. As usual, the residue map $\res_{z=0} : \R[[t_1,\ldots, t_d,z, z^{-1}]] \to \R[[t_1,\ldots, t_d]]$ 
is the map that sends  $f= \sum_{i\in \Z} f_i z^{i}$ to $f_{-1}$ ($f_i \in  \R[[t_1,\ldots, t_d]]$).
The subscript $s=0$ is an abbreviation for $ s_1 = \ldots = s_d = 0 $.
$(T_X^{\Omega_1} - T_X^{\Omega_2})$ is a polynomial in $\R[t_1,\ldots, t_d]$

\begin{Example}
 Consider $X=((1,0),(0,1),(1,1))$. 
 Let $\Omega_1 = \cone\{(1,0),(1,1)\}$ and $\Omega_2= \R^2\setminus \R^2_{\ge 0}$. Then $\Omega_{12}$
 is the ray spanned by $(1,0)$ and $V_{12}=1$.
 \begin{equation*}
  (T_X^{\Omega_1} - T_X^{\Omega_2})  = \res_{z=0} \left( \frac{ e^{ t_1  s_1  + t_2   s_2 + t_2 z     } }{   (t_2 s_2 + z)(t_1 s_1  + t_2 s_2  + z)}  
  \right)_{s=0}
 =
  \res_{z=0} \left( \frac{ e^{ t_2 z   } }{   z^2 }  \right ) = t_2.
 \end{equation*}
\end{Example}

\begin{Lemma}
\label{Lemma:WallCrossing}
We use the same terminology as in Theorem~\ref{Theorem:BoysalVergneWallCrossing}
and assume in addition that $\eta=t_1$.
Let  $V_{12}$ be the polynomial \st 
  $V_{12}|_H=T_{X\cap H}^{\Omega_{12}}$ and  $V_{12}$ is constant on lines perpendicular to $H$.
This implies that $V_{12} \in \R[t_2,\ldots, t_d]$.
Then 
\begin{align}
 (T_X^{\Omega_1} - T_X^{\Omega_2}) &= %
    c_{X} t_1^{m(H)-1} V_{12} + t_1^{m(H)} h  
\end{align}
for some homogeneous polynomial $h\in \R[t_1,\ldots, t_d]$ of degree $ \abs H - d$ and $c_{X}:=\frac{1}{ (m(H)-1)! \prod_{x\in X\setminus H} t_1(x)  } \in \R$.
If $\abs{H}=d-1$, then  $h=0$.

More generally, for a homogeneous polynomial
$p\in \R[t_2,\ldots, t_d] $,
 \begin{align}
 \label{eq:WallCrossingLemma}
  \res_{z = 0} \left( 
   p(D)
    \frac{ e^{ t_1 s_1 + \ldots t_d s_d + t_1  z  } }{\prod_{x\in X\setminus H} ( t_1(x)s_1 + \ldots  + t_d(x)s_d + t_1(x) z ) }    
   \right)_{ s = 0 }
   &=   c_{X} t_1^{m(H)-1} p  + t_1^{m(H)} g  
\end{align}
for   $c_X$ as above and $g \in \R[t_1,t_2,\ldots, t_d]$ that is homogeneous of degree $\deg p  - 1$.
 If $p$ is constant, then $g=0$. 
\end{Lemma}
\begin{proof}
 We use induction over the degree of $p$ to prove the second statement.
 Suppose first that $p=1$.
 Then the  term on the left-hand side of \eqref{eq:WallCrossingLemma}
 is equal to
 \begin{align*}
     \res_{ z = 0 } \left( 
     \frac{e^{ t_1 z  }}{  \left(  \prod_{ x\in X\setminus H}  t_1(x)  \right) z^{m(H)} }
   \right) %
   = \frac{ t_1^{m(H) - 1}   }{  (m(H)-1)! \prod_{x\in X\setminus H} t_1(x) } = c_X t_1^{m(H) - 1}.
 \end{align*}
 \begin{equation}
 \text{Let } \quad G_{X} := \frac{ e^{ t_1 s_1 + \ldots  +  t_d s_d +  t_1 z  } }{\prod_{x\in X\setminus H} ( t_1(x)s_1 + \ldots  + t_d(x)s_d + t_1(x) z ) }    
 \quad \text{ and let } 
 \end{equation}
  $ p =  q \cdot t_j\in \R[t_2,\ldots, t_d]$ be a monomial. 
 Recall that $p(D)$ denotes the differential operator obtained from $p$ by replacing $t_i$ by $\diff{s_i}$.
 Using the quotient rule we obtain
  \begin{equation*}
   p(D) G_X
  =
       t_j q(D)      G_X
        - q(D) \sum_{x\in X\setminus H} t_j(x) \frac{ G_X } %
        {( t_1(x)s_1 + \ldots + t_d(x) s_d + \eta(x) z ) }  .
 \end{equation*}
 By induction 
 the residue of $q(D)       G_X |_{s=0}$
  is $c_X t_1^{m(H)-1}q + t_1^{m(H)} g_1$ with $c_X \in \R$ as defined above
  and a homogeneous polynomial $g_1$ of degree $\deg(q)-1$.
 Note that   
 $ G_X / {( t_1(x)s_1 + \ldots + t_d(x) s_d + \eta(x) z ) } = G_{X_x'}$, where $X_x'$ is obtained from $X$ by adding an extra copy of $x$ 
 and that
 the term $t_j(x)$ is  just a real number.
 By induction  the residue  of $q(D)G_{X_x'}|_{s=0}$ is equal to 
 $c_{X_x'} t_1^{m(H)} q + t_1^{m(H)+1} g_x$ for some homogeneous polynomial $g_x \in \R[t_1,\ldots, t_d]$ of degree $\deg(q)-1$.
Hence
 \begin{align}
   \res_{z=0} (G_X)_{s=0} &= t_j (c_X t_1^{m(H)-1}q + t_1^{m(H)} g_1) + \sum_{x\in X\setminus H} t_j(x) (c_{X_x'}   t_1^{m(H)} q + t_1^{m(H)+1} g_x)
 \notag \\
   &=   c_X t_1^{m(H)-1}p  + t_1^{m(H)}  \underbrace{(t_j g_1  + \sum_{x\in X\setminus H} t_j(x) (c_{X_x'}  q + t_1 g_x))}_{\text{homogeneous of degree } \deg p-1}
 \end{align}
Using the fact that  homogeneous polynomials are sums of monomials of the same degree, 
 the  second statement follows.

 The first statement follows easily from the second using Theorem~\ref{Theorem:BoysalVergneWallCrossing}  
  taking into account that $T_{X\cap H}^{\Omega_{12}}$ is a homogeneous polynomial of degree $\abs H - d + 1$.
\end{proof}

\begin{Lemma}
\label{Lemma:LatticeConeVanishes}
\XintroLattice
 Let $\mu\in \Lambda$ and let $\Lambda'\subseteq \Lambda$ be a sublattice.
  Let $C\subseteq U$ be a full-dimensional cone.
  Let $f\in \sym(V)$. 
  Suppose that $f(C\cap (\Lambda' + \mu))=0$. Then $f=0$.
\end{Lemma}
\begin{proof}
 Let $0 \neq \lambda \in C\cap (\Lambda' + \mu)$. 
 For $k\in \R$, let $p(k):=f(k\lambda)$. 
 There exists a positive integer $l$ \st $l\mu\in \Lambda'$.
 Hence $(rl+1)\lambda \in \Lambda'+ \mu$ for all $r\in \Z$. This implies that for any $r\in \N$,
 $f( (rl+1)\lambda)=0$. Thus
  $p$ is a univariate polynomial in $k$ with infinitely many zeroes.
 This implies $p(k)=0$ and thus $f(u) = 0$ for any $u\in C$ that can be written as  
  $ k \lambda $ with $ \lambda \in C \cap (\Lambda' + \mu) $ and $k\in \R$.
    Not every  $u\in C$ can be written in this way, but every $u\in C$ is the limit of a sequence of points with this property.   
   Since polynomials are continuous, $f(C)=0$ and as  $C$ is full-dimensional, this implies $f=0$.
\end{proof}

\begin{proof}[Proof of Theorem~\ref{Proposition:internalContinuousCharacterisation}]
Let $\Omega_1$ and $\Omega_2$ be two big cells whose closures have a $(d-1)$-dimensional intersection.
 Let $H$ be the hyperplane that contains this intersection.
 Let $T_X^{\Omega_1}$ and $ T_X^{\Omega_2}$ denote the polynomials that agree with $T_X$ on $\Omega_1$ and $\Omega_2$, respectively.
 Without loss of generality, we may assume that $H$ is the hyperplane perpendicular  to $t_1$.
 Let $\lambda\in H \cap \Lambda$. 
 Let $p \in \Pper(X)$ and let $p_\lambda = p(\lambda,\cdot) \in \Pcal_\CC(X)$ denote the local part at $\lambda$.
 Let $m:=m(H)$. 
 By definition, we can write $p_\lambda$ uniquely as
 \begin{equation}
    p_\lambda  = p_{m-1} s_1^{m-1} + \ldots + p_1 s_1 + p_0  \text{  for some } p_i\in \CC[s_2,\ldots, s_d].
 \end{equation}
  Note that $p \in \Pper_-(X)$ if $p_{m-1} = 0 $ for all hyperplanes $H\in \Hcal(X)$ and $\lambda\in H\cap \Lambda$.
  By Lemma~\ref{Lemma:WallCrossing},
\begin{equation}
  p_\lambda(D)  (T_X^{\Omega_1} - T_X^{\Omega_2}) =   c_X(m(H)-1)! p_{m-1}(D)  V_{12}   + t_1 g
 \end{equation}
 for some $g \in \sym(V)$ and $V_{12}$ as defined in Lemma~\ref{Lemma:WallCrossing}.

 Suppose that $p\in \Pper_-(X)$.
 Then by definition, $p_{m-1} = 0$.
 This implies that $p_\lambda(D)(T_X^{\Omega_1} - T_X^{\Omega_2})(\lambda) = 0$, as $\lambda\in H$ implies $t_1(\lambda)=0$.
 Hence $T_X$ is continuous in $\lambda$ across the wall $H$.
 \smallskip
 
 Now we want to show that if $p_\lambda(D)T_X$ is continuous,  then $p_{m-1} = 0 $.
 Let $\Lambda'\subseteq \Lambda$ be a sublattice \st the restriction of $p$ to a coset of $\Lambda'$ is a polynomial.
 It is sufficient to show that if $p_{m-1} \neq 0$, then there is a $\mu\in (\Lambda' + \lambda) \cap H$ 
 (\ie $p_\lambda=p_\mu$) \st  $p_{m-1}(D)  T_{X\cap H}(\mu) \neq 0$.

 \emph{Claim: $p_{m-1}$ is contained in $\Pcal(X\cap H)$}. %
 Let $p_Y$ be a generator of $\Pcal(X)$. Let $ Y_1 = Y \cap H $ and $ Y_2 = Y \setminus H $. If the polynomial $p_Y$ contributes to the $s_1^{m-1}$ term, then
 $ \abs{Y_2} = m-1 $. 
 This implies that $X\setminus (H\cup Y)$ contains a unique element $y_0$.
 Since $X\setminus Y$ has full rank, $ (X\cap H)  \setminus Y_1$ must span $H$. Hence $p_{Y_1} \in \Pcal(X\cap H)$.
 Furthermore, $p_Y = p_{Y_1} p_{Y_2} = \gamma s_1^{m-1} p_{Y_1} + o(s_1^{m-1})$ for some $\gamma\in \R$. This proves the claim since $p(\lambda,\cdot) \in \Pcal_\CC(X)$.

  The local pieces of $ T_{X \cap H} $ are contained in $\Dcal(X\cap H)$ by Theorem~\ref{Proposition:Dlocalpieces}.
  So by duality (Theorem~\ref{Proposition:PDduality}) and using the fact 
  that the local pieces of $T_{X \cap H}$ span the top degree part of $\Dcal(X\cap H)$, 
  there must be a big cell $\Omega'$ in $H$ \st the corresponding local piece $T_{X\cap H}^{\Omega'}$  
   is not annihilated by $p_{m-1}(D)$.  
  Hence by Lemma~\ref{Lemma:LatticeConeVanishes}, there is
  a point $\mu \in \Omega' \cap (\Lambda' + \lambda) $  \st $ p_{m-1}(D)T_{X\cap H}^{\Omega'}( \mu ) \neq 0 $.
   Hence $p_\lambda(D) T_X$  is discontinuous in $\mu$, which is a contradiction. This finishes the proof.
\end{proof}

\begin{Remark}
 A different 
 approach to prove Theorem~\ref{Proposition:internalContinuousCharacterisation}
 would have been to use a modified version 
 of \cite[Corollary~19]{brion-vergne-1999} that characterises the smoothness of a piecewise polynomial function along a wall 
 in terms of the Laplace transform.

 A result similar to Lemma~\ref{Lemma:LatticeConeVanishes} for arbitrary piecewise-polynomial functions is known~\cite[Theorem~1]{wang-2000}.
 
\end{Remark}

 \section{Deletion-contraction and the proof of Theorem~\ref{Proposition:PeriodicInternalTutteEval}}
 \label{Section:DeletionContraction}

 In this section we will discuss deletion-contraction for finitely generated abelian groups and
 periodic $\Pcal$-spaces
 and then prove an analogue of 
  Proposition~\ref{Proposition:DeletionContractionP} on short exact sequences.
 This will allow us to prove   
 Theorem~\ref{Proposition:PeriodicInternalTutteEval}
 and Proposition~\ref{Proposition:InhomogeneousInternalBasis} that describe properties of the internal periodic $\Pcal$-space.
 
 \subsection{Deletion-contraction.}
  \label{Subsection:DeletionContraction}

Recall that we have defined deletion-contraction for $X\subseteq U$
and $\Pcal(X)\subseteq \sym(U)$ in 
Subsection~\ref{Subsection:ZonotopalSpaces}.
 Now we require deletion and contraction for $X\subseteq G$ and $\Pper(X)\subseteq
 \bigoplus_{ e_\phi\in \Vcal(X) } \sym(U)$.
We are working with finitely generated abelian groups in this section since they are closed under taking quotients.
This is in general not the case for lattices.
 
 Let $x\in X$. As usual, we call the list $X\setminus x$ the \emph{deletion} of $x$ and the image 
 of $X\setminus x$ under the projection $\pi_x : G \to G/x$ is called the \emph{contraction}
 of $x$. It is denoted by $X/x$.

The definition of the projection map $\pi_x : \Pper(X) \to \Pper(X/x)$
 requires a few more thoughts.
 Its definition has two ingredients:
 a projection of the polynomial part and a projection of the torus.
 
 Recall that $U=G\otimes \R$ and that $\Pcal(X)$ is contained in $\sym(G\otimes \R)$.
 The space  $\Pcal(X/x)$ is contained in $\sym((G/x) \otimes \R)$.
 Lemma~\ref{Lemma:tensorquotient} implies that
$\sym((G/x) \otimes \R)$ is canonically isomorphic to 
$\sym( (G \otimes \R)/(x\otimes 1))$. This implies that also in the case
where $X$ is contained in a finitely generated abelian group $G$, 
we can use the 
usual projection map $ \pi_x : \sym( U ) \to \sym( U / x ) $ to project
$\Pcal(X)\to \Pcal(X/x)$.

 Note that a map $e_{\bar\phi} : G/\langle x\rangle \to S^1$  is equivalent to a map $e_\phi : G \to S^1$ that satisfies $e_\phi(x) = 1$.
 This implies that $T(G/\langle x\rangle) \cong H_x$. 
 
 Let $x\in X$ be an element that is not torsion.
 Now we define the projection map  $\pi_x : \Pper(X) \to \Pper(X/x)$ as follows: 
let $e_\phi p_{X\setminus (X_\phi \cup X_t)} s_0^{t_X(\phi)} p_Y$ be a generator of $\Pper(X)$, where $p_Y\in \Pcal(X_\phi)$.
We define $\pi_x$ to be the map that sends this generator to
$0$ if $e_\phi\not\in \Vcal(X) \cap H_x$ and to 
$e_{ \bar \phi } \bar p_{X\setminus (X_\phi \cup X_t \cup \spa(x))} s_0^{t_{X/x}(\phi)} \bar p_Y  $ otherwise.  
 Here,  $\bar p$ denotes the image of $p$ under the projection $  \sym(U) \to \sym(U/x)$.  
Removing $\spa(x)$ in the prefactor is necessary to remove the elements that turn into torsion elements in $X/x$.
 Note that if $e_\phi\not\in \Vcal(X) \cap H_x  $, then $x|p$, hence $\bar p=0$.
 So it makes sense to send the corresponding generators to $0$.

  \begin{Example}
  \label{Example:DeletionContractionPper}
   Let $X=\begin{pmatrix} 1 & 0 & 0 \\ 0 & 2 & 1 \end{pmatrix}$. We contract  the second element and get $X/x = ((1,\bar 0),(0,\bar 1)) \subseteq \Z\oplus \Z_2$.   
  Note that $\Vcal(X) = \{1, (-1)^b \}$ and $\Vcal(X/x) = \{ 1, (-1)^{\bar b} \}$.
      Then $ \Pper(X) = \spa \{1,s_2,  (-1)^b s_2  \} $ and $\Pper(X/x) = \spa\{ 1, (-1)^b s_0 \}$. 
   The following sequence is exact:
   \begin{equation}
    0 \to \spa\{1\} \stackrel{\cdot s_2}{\longrightarrow} \spa \{1,s_2,  (-1)^b s_2  \}  \stackrel{\pi_x}{\longrightarrow}  \spa\{ 1, (-1)^b s_0 \}.
   \end{equation}

  \end{Example}

  \begin{Lemma}
  \label{Lemma:tensorquotient}
    Let $G$ be a finitely generated abelian group and let $H$ be a subgroup.

    Then $(G/H) \otimes \R \cong (G \otimes \R ) / (H\otimes \R)$.
    So in particular, $\Pcal(X/x) \subseteq \sym( (G/ x) \otimes \R) \cong \sym( U / (x\otimes 1) )$.
  \end{Lemma}
  \begin{proof}
   Note that $\R$ is a flat $\Z$-module, \ie the functor $\otimes_\Z \R$ is exact
       (this follows for example from Proposition~XVI.3.2 in \cite{lang-algebra-2002}).
   Hence, exactness of the sequence  
  $0 \to H \to G \to G/H \to 0$ implies that the following sequence is exact:
  \begin{align}
   0 \to H \otimes \R \to G  \otimes \R  \to G/H  \otimes \R  \to 0.
  \end{align}
  This implies the statement.
  \end{proof}

  \subsection{Exact sequences}
  Recall that for a graded vector space S, we write $S[1]$ to denote  the vector space with the degree
shifted up by one.
  
 \begin{Proposition}
 \label{Proposition:DeletionContractionCentralPeriodic}
 \XintroAbelianGroup
   Let $x\in X$ be an element that is not torsion.  
  Then the following is an exact sequence of graded vector spaces:
   \begin{align}
    0 \to \Pper( X\setminus x )[1]  \stackrel{\cdot p_x}{\longrightarrow} \Pper(X) \stackrel{\pi_x}{\longrightarrow} \Pper(X/x) \to 0.
   \end{align}
  \end{Proposition}
  \begin{proof}
   \emph{$\cdot p_x$ is well-defined:}  we will show that generators of
   $\Pper(X\setminus x)$ are mapped to generators of $\Pper(X)$.
   Let $e_\phi s_0^{\tors_{X\setminus x}(\phi)}  p_Y \in p_{X \setminus ( X_\phi \cup X_t \cup x)} s_0^{\tors_{X\setminus x}(\phi)} \Pcal(X_\phi \setminus x)$ be a generator. 
   Since $x$ is not torsion, $\tors_X(\phi) = \tors_{X\setminus x}(\phi) $, so the $s_0$ part is fine.
   If $x \in X_\phi$, then 
   $ p_x \Pcal(X_\phi \setminus x)\subseteq \Pcal(X_\phi)$ by
   Proposition~\ref{Proposition:DeletionContractionP}.
   If $x\not\in X_\phi$ then the prefactor is multiplied by $p_x$.
   
   \emph{$\pi_x$ is well-defined:} let $e_\phi p_{X\setminus ( X_\phi \cup X_t )} s_0^{t_X(\phi)} p_Y$ be a generator of $\Pper(X)$.
   If $e_\phi(x) =  1$, then by definition, it is mapped to
   $e_{ \bar \phi } \bar p_{X\setminus (X_\phi \cup X_t \cup \spa(x))} s_0^{t_{X/x}(\phi)} \bar p_Y$. 
   This is a generator of $\Pper(X/x)$ since $\bar p_Y$ is known to be in $\Pper(X_\phi/x)$ by Proposition~\ref{Proposition:DeletionContractionP}. 
   If $e_\phi(x) \neq 1$, then the generator  is mapped to $0$.

   \emph{$\pi_x \circ( \cdot x ) = 0$} is clear.
   
   \emph{Surjectivity of $\pi_x$:} let $h:= e_{\bar\phi} s_0^{t_{X/x}(\bar\phi)}  \bar p_{ X\setminus (X_{ \phi} \cup X_t \cup \spa(x))} \bar p_{  Y}$ be a generator of $\Pper(X/x)$.
   There is a vertex $e_\phi\in \Vcal(X) \cap H_x$ that corresponds to $e_{\bar \phi} \in\Vcal(X/x) $ 
   and $ e_{\phi} s_0^{t_{X}(\phi)}  p_{ X\setminus (X_{ \phi} \cup X_t )} p_{ Y} $ is a generator of $\Pper(X)$ that is contained in the preimage of $h$.

   \emph{Exactness in the middle:}
    It is sufficient to show that $\dim \Pper(X\setminus x) + \dim \Pper(X/x) = \dim \Pper(X)$.  This follows from
  Theorem~\ref{Proposition:CentralPeriodicPHilbertTutte} and the
   deletion-contraction formula for the arithmetic Tutte polynomial \eqref{equation:AriTutteDelCon}.
  \end{proof}

The following lemma is a special case of Lemma~\ref{Lemma:PspacesMolecules}.
It will be used in the proof of Lemma~\ref{Lemma:InternalDimensionInequality},
which will be used to prove Lemma~\ref{Lemma:PspacesMolecules}.
\begin{Lemma}
\label{Lemma:RankOne}
 Let $G$ be a finitely generated abelian group of rank zero, 
 or in other words, a \emph{finite} abelian group.
 Let $X$ be a non-empty finite list of elements of $G$. 
 Then
 \begin{align}
  \dim \Pper_-(X) = \dim \Pper(X)  = \aritutte_X (0,1) = \aritutte_X(1,1).
 \end{align}

\end{Lemma}
\begin{proof}
 The torus is  $T(G) =\hom(G, S^1) \cong G$. By definition, $\Vcal(X)=T(G)$. Since there are no hyperplanes, $\Pper(X) = \Pper_-(X)$.
 For each $e_\phi\in \Vcal(X)$, $\Pcal(X_\phi)=\R$, hence $\dim \Pper(X)=\dim \Pper_-(X)=\abs G$.

 To finish the proof, note that
   $\aritutte_X (0,1)= \sum_{S\subseteq X} \multari(S) (-1)^0  (0)^{\abs{S}} = \multari(\emptyset) = \abs{G} = \sum_{S\subseteq X} \multari(S) 0^0 (0)^{\abs{S}} = \aritutte_X(1,1)$.
\end{proof}

The following lemma is a weaker version of Proposition~\ref{Proposition:InhomogeneousInternalBasis}.
It will be used in the proof of Proposition~\ref{Proposition:DeletionContractionInternalPeriodic} below,
which will in turn be used to finish the proof of  Proposition~\ref{Proposition:InhomogeneousInternalBasis}.
\begin{Lemma}
\label{Lemma:fzInternalIndependent}
\XintroLattice
Then the set $\{ \tilde f_z : z \in \Zcal_-(X)  \}$ is a linearly independent subset of $\Pper_-(X)$. 
\end{Lemma}
\begin{proof}
 Linear independence follows from Proposition~\ref{Proposition:InhomogeneousBasisCentralPeriodic}.
 Containment in $\Pper_-(X)$
 is a consequence of Theorem~\ref{Proposition:internalContinuousCharacterisation} and 
 Theorem~\ref{Theorem:ModifiedBrionVergne}.
\end{proof}
  \begin{Lemma}
  \label{Lemma:InternalDimensionInequality}
    \XintroAbelianGroup
      Then  $\dim \Pcal_-(X) \le \aritutte_X(0,1)$.
  \end{Lemma}
  \begin{Proposition}
  \label{Proposition:DeletionContractionInternalPeriodic}
  \XintroAbelianGroup
  Let $x\in X$ be an element that is neither torsion nor a coloop. %
  Then the following is an exact sequence of graded vector spaces:
   \begin{align}
   \label{eq:exactSequencePperInt}
    0 \to \Pper_-( X\setminus x )[1]  \stackrel{\cdot p_x}{\longrightarrow} \Pper_-(X)
    \stackrel{\pi_x}{\longrightarrow} \Pper_-(X/x) \to 0.
   \end{align}
  \end{Proposition}
  \begin{proof}[Proof of Proposition~\ref{Proposition:DeletionContractionInternalPeriodic} and Lemma~\ref{Lemma:InternalDimensionInequality}]
  This proof is more complicated than the proof of  Proposition~\ref{Proposition:DeletionContractionCentralPeriodic}.
  As we do not know a canonical generating set for the space $\Pper_-(X)$, it is  more difficult to show that 
  $\cdot p_x$ is well-defined and that $\pi_x$ is surjective.
  Here is an outline of the proof:
  \begin{asparaenum}[(a)]
  \item  show that the following sequence is exact for $x\in X$ that is not torsion (but $x$ may be a coloop):
  \begin{align}
   \label{eq:leftexact}
    0 \to \Pper_-( X\setminus x ) [1] \stackrel{\cdot p_x}{\longrightarrow} \Pper_-(X) + p_x \cdot \Pper_-( X\setminus x) 
    \stackrel{\pi_x}{\longrightarrow} \Pper_-(X/x).
   \end{align}
  \item Deduce that $\dim(\Pper_-(X)) \le \aritutte_X(0,1)$, \ie prove Lemma~\ref{Lemma:InternalDimensionInequality}.
  \item Show the exactness of \eqref{eq:exactSequencePperInt} using Lemma~\ref{Lemma:fzInternalIndependent}.
  \end{asparaenum}
  \smallskip
  
\noindent Here are the details of the proof:
   \begin{asparaenum}[(a)]
    \item    
    \emph{$\pi_x$ is well-defined:}
    Obviously, $p_x\cdot\Pper_-(X\setminus x)$ is mapped to zero.
    It follows from 
    Proposition~\ref{Proposition:DeletionContractionCentralPeriodic}
    that $\Pper_-(X)$ is mapped to $\Pper(X/x)$. So we only have to check the differential equations.

     Consider $\bar H\in \Hcal(X/x)$. This corresponds to $H\in \Hcal(X)$ that contains $x$. 
        Let $\bar \lambda \in \bar H$. 
        Let $\bar \eta\in (U/x)^*$ be a normal vector for the hyperplane $\bar H\otimes 1\subseteq U/x$.
        There is a corresponding normal vector $\eta\in U^*=V$ for $H\otimes 1$ that satisfies $\eta(x)=0$.        
 Let $\lambda \in H$ be a representative of $\bar \lambda$.
 The choice of the representative does not matter because $e_\phi(x)\neq 1$ implies that $\pi_x$ maps the $e_\phi$ part to zero and 
 $e_\phi(x)=1$ implies $e_\phi(\lambda + kx)= e_\phi(\lambda)$ for $k\in \Z$.
         Note that $m_X(H) = m_{X/x}( \bar H )$. Hence
         $D_{\eta }^{m(H) - 1 } p(\lambda, \cdot) = 0 $ implies  $D_{ \bar\eta }^{m(H) - 1 } \bar p( \bar \lambda, \cdot) = 0 $
         for $p\in \Pper(X)$.\footnote{For an example, consider Example~\ref{Example:nontrivial2dtorarr} and in particular \eqref{eq:internalProjectionExample}.
         There is only one hyperplane in $\Z/2\Z\oplus \Z$. It corresponds  to  $H= \spa((1,0))$ in $\R^2$ 
         and  representatives for the two points that it contains are  $\lambda_1=(1,0)$ and $\lambda_2=(0,0)$.  
         The normal vector is $\eta = (0,1)$.  
}         

    \emph{Exactness in the middle:}
    Let  $p\in \Pper_-(X) +  p_x \cdot\Pper_-(X\setminus x)$ be an element \st $\pi_x(p) = 0$.
    The case $p\in p_x \cdot\Pper_-(X\setminus x)$ is trivial so suppose that $p\in \Pper_-(X)$.
   Then Proposition~\ref{Proposition:DeletionContractionCentralPeriodic} implies that
   $p = p_x \cdot h$ for some $h\in \Pper(X\setminus x)$.
   We have to show that $h$ is contained in  $\Pper_-(X \setminus x)$, \ie
   we have to check that $h$ satisfies the differential equations. 
 
   Let $H\in \Hcal(X \setminus x)$ and let $\lambda\in H$.
   If $ x \in H$, then $0 = D_\eta^{m_X(H)-1} p_x h(\lambda,\cdot) = p_x   D_\eta^{ m_{X\setminus x}(H)-1 } h(\lambda, \cdot)$.
   If $x\not\in H$, then    $m_{X\setminus x}(H)= m_X(H)-1$, so 
   $D_\eta^{m_X(\eta) - 1} p_x h(\lambda,\cdot) = 0$   implies  $D_\eta^{ m_{X\setminus x}(\eta)-1 }  h(\lambda, \cdot)=0$.

\smallskip
Now we have established the exactness of  \eqref{eq:leftexact}. This implies the following inequality:
  \begin{align}
  \label{equation:InequalitiesPlemma}
   \dim (\Pper_-( X )) \le \dim (\Pper_-( X ) + x \Pper_-(X\setminus x)) \le \dim \Pper(X\setminus x) + \dim \Pper(X/x).
  \end{align}
\item 
   We will now prove  by induction that $\dim(\Pper_-(X)) \le \aritutte_X(0,1)$.
    If $G$ is finite, then we are done by Lemma~\ref{Lemma:RankOne}. 

       Now suppose that $X$ contains only coloops and torsion elements. Let $x$ be a coloop.
   Using \cite[Lemma 5.7]{moci-adderio-2013} and the fact that $\aritutte_X(0,1) = 0$ if $X$ does not span a subgroup of finite index, we obtain
   that $\aritutte_X(0,1) = \aritutte_{X/x}(0,1)$. So in this case, since $\Pper_-(X\setminus x)=0$,  
   we obtain $ \dim\Pper_-( X ) \le \dim \Pper_-(X/x) \le \aritutte_{X/x}(0, 1) = \aritutte_X(0, 1)$ using \eqref{equation:InequalitiesPlemma} and induction.

    Now suppose that the $X$ contains an element $x$ that is neither torsion nor a coloop.
  Then by induction using \eqref{equation:InequalitiesPlemma} and \eqref{equation:AriTutteDelCon}, we obtain $\dim (\Pper_-( X )) \le \aritutte_X(0,1)$.

    \medskip
\item 
     Suppose that $X\subseteq \Lambda$  for some lattice $\Lambda$.
     By
    Lemma~\ref{Lemma:fzInternalIndependent} and Proposition~\ref{Proposition:ZonotopeArithmeticTutte}
    $\dim\Pper_-(X)\ge \aritutte_X(0,1)$. Hence $\dim\Pper_-(X)= \aritutte_X(0,1)$.
    This implies that all the inequalities in \eqref{equation:InequalitiesPlemma} must be equalities.
    Thus
    $x \Pcal_-(X\setminus x)) \subseteq  \Pper_-( X )$ and the projection map $\pi_x$ must be surjective. 
     Hence the sequence \eqref{eq:exactSequencePperInt} is exact.

     We call $Y\subseteq G'$ a \emph{minor} of $X\subseteq G$ if there are sublists $X_1, X_2\subseteq X$ \st $Y = (X\setminus X_1)/X_2$ and $G'=G/\sg{X_2}$.
     By induction, if $X$ is contained in a lattice $\Lambda$, for every minor $Y$ of $X$, we have $\aritutte_Y(0,1)= \dim \Pper_-(Y)$ and the sequence \eqref{eq:exactSequencePperInt} is exact.

Now note that every $X\subseteq G$ ($G$ finitely generated abelian group) is a minor of some $X'\subseteq \Lambda$ ($\Lambda$ lattice). This finishes the proof.
\qedhere
 \end{asparaenum}
  \end{proof}
\begin{Remark}
 If $x\in X$ is a coloop, then the  map $\pi_x$ in  \eqref{eq:leftexact} is not necessarily surjective. 
 For an example consider the case $X=((2,0),(0,2))$ (Example~\ref{Example:InternalSpaceColoops}).
The contraction is studied in Example~\ref{Example:internal}.
In this case $\dim \Pper_-(X) = 1 < 2 = \dim \Pper_-(X/x)$.
 \end{Remark}
  The following lemma will be used in the proof of Theorem~\ref{Proposition:PeriodicInternalTutteEval}.
  \begin{Lemma}[Molecules]
  \label{Lemma:PspacesMolecules}
   \XintroAbelianGroupN
   Suppose that $X$ contains only coloops and torsion elements. Such list are called \emph{molecules} in \cite{moci-adderio-2013}.

  If we choose a suitable isomorphism $G\cong \Z^d \oplus G_t$, then $X$ corresponds to the list 
   $(a_1 e_1,\ldots, a_d e_d, h_1,\ldots, h_k)$ with $h_i\in G_t$ and $a_i\in \Z_{\ge 1}$. As usual, $e_i\in \Z^d$ denotes the $i$th unit vector.
   Let $\xi_{a_\nu}^{j_k} \in T(\Z^d)$ denote the map that sends $e_\nu$ to $e^{2\pi i \frac{j_k}{a_\nu} }$ and all other $e_\mu$ to $0$.
Then
   \begin{align}
    \Vcal(X) &= \{ \xi_{a_1}^{j_1} \cdots \xi_{a_d}^{j_d} g  : 0\le j_i \le a_i-1, \, g\in T(G_t) \}, 
    \label{equation:Vmolecule}
    \displaybreak[2]
    \\
     \Pper(X)&=\bigoplus_{ e_\phi \in \Vcal(X) } e_\phi p_{ X\setminus (X_\phi \cup X_t) }  s_0^{ \tors(\phi) }  \R, \quad\text{ and } 
    \label{equation:centralmolecule}  
     \\
   \label{eq:internalPcoloops}
    \Pper_-(X) &= \spa  
    \{ (\xi_{a_1}^{j_1} - \xi_{a_1}^{j_1-1}) \cdots (\xi_{a_d}^{j_d} - \xi_{a_d}^{j_d-1} ) g s_0^{\tors(\phi) }  : 1 \le j_i \le a_i-1, \, g\in T(G_t) %
      \}.
 \end{align}
   Furthermore, $\hilb(\Pper_-(X),q) = q^{N-d} \aritutte(0,\frac 1q)$.%
  \end{Lemma}
\begin{Example}
 \label{Example:molecule}
  Let $X=(\bar 2)\subseteq \Z/4\Z$. Then the arithmetic Tutte polynomial is $\aritutte_X(\alpha, \beta) = 2 \beta + 2 $ 
  and $\Pper(X) = \Pper_-(X) = \spa\{1, g_4 s_0, g_4^2, g_4^3 s_0 \} $, 
  where $g_4^j : \Z / 4 \Z \to S_1$ is defined by $ g_4^j(k) = e^{ \frac{\pi i}{2} jk }$. 
\end{Example}
\begin{proof}
Note that \eqref{equation:Vmolecule} is trivial.
 As $X$ contains only coloops and torsion elements
  $\Pcal(X_\phi) = \R$ for all $e_\phi\in \Vcal(X)$.
  This implies  formula \eqref{equation:centralmolecule}.
  
  \smallskip
 Now let us consider $\Pper_-(X)$. For every $H\in \Hcal(X)$, $m(H)=1$.
 Hence the  differential equations that have to be satisfied do not involve a differential operator.
 We simply have to check $p(\lambda,\cdot)=0$ for all $\lambda$ that are contained in some $H$ (cf.~Example~\ref{Example:InternalSpaceColoops}).

Let $\Gamma(X)$ denote the set on the right-hand side of \eqref{eq:internalPcoloops}.
It is   clear that $\Gamma(X)$ is linearly independent.
Let $\lambda\in H\in \Hcal(X)$.
We can uniquely write $\lambda = \sum_{i=1}^d \nu_i a_i e_i + \sum_{j=1}^k \mu_j g_k$ for some coefficients 
 $\nu_i,\mu_j \in \R$ and $g_j\in G_t$. Since $\lambda$ lies in a hyperplane, at least one of the $\nu_i$ is zero. Then for this $i$,
 $(\xi_{a_i}^{j_i} - \xi_{a_i}^{j_i-1})(\lambda)=0$.  Hence each generator of $\Gamma(X)$ vanishes on $\lambda$. 
This shows that the set $\Gamma(X)$ is   contained in $\Pper_-(X)$.

By Lemma~\ref{Lemma:InternalDimensionInequality}, $\dim\Pper_-(X)\le \aritutte_X(0,1)$. Therefore, it is sufficient to show that
$\hilb(\spa(\Gamma(X)),q)= q^{N-d} \aritutte_X(0,1/q)$. 

 Since $X$ is a molecule, we can split $X$ into a disjoint union of the free elements $X_f\subseteq \Z^d$ and 
the torsion elements $X_t \subseteq G_t$
(cf.~Example 4.9 in \cite{moci-adderio-2013}). 
The arithmetic matroid defined by $X$ can then be seen as a direct sum $X_f\oplus X_t$ of the matroids defined by $X_f$ and $X_t$
and
$\aritutte_X(\alpha,\beta)= \aritutte_{X_f}(\alpha,\beta) \cdot \aritutte_{X_t}(\alpha,\beta)$.
The two matroids have multiplicity functions $\multari_f$ and $\multari_t$ that are defined by the lists $X_f\subseteq \Z^d$ and $X_t\subseteq G_t$, respectively.
Note that $\multari_f(A) = \prod_{ a_ie_i \in A   } a_i$. Hence $\aritutte_{X_f}(0, q ) = \sum_{I\subseteq 2^{[d]}} (-1)^{d- \abs{I}} \prod_{i\in I}  a_i = \prod_{i=1}^d (a_i - 1)$.

Note that $\abs{X_t} = N - d$.
 It is easy to see that%
 \begin{equation}
  q^{N-d}\hilb(\spa(\Gamma(X)), \frac 1q) = \left( \prod_{i=1} (a_i - 1) \right) \sum_{ i = 0 }^{N-d} \mu_i q^{i} 
   = \aritutte_{X_f}(0, q)  \sum_{ i = 0 }^{N-d} \mu_i q^{i},
 \end{equation}
 where $\mu_i = \abs{ \{g \in T(G_t) : \abs{(X_t)_g} = i \} } $. As usual, $ (X_t)_g := (x \in X_t : g(x)= 1)$.

Note that  $ \aritutte_X(0,q) = \aritutte_{X_f}(0,q) %
  \sum_{A\subseteq X_t} \multari_t(A) (q-1)^{ \abs{A}  }$.
 So all that remains to be shown is that
 $\sum_{ i = 0 }^{N-d} \mu_i q^{ N-d - i} = \sum_{A\subseteq X_t} \multari_t(A) (q-1)^{ \abs{A}  }$.
The right-hand side of this equation can be expanded as
\begin{align}
  \aritutte_{X_t}(0,q) &= \sum_{A\subseteq X_t} \multari_t(A) (q-1)^{\abs{A}}  
  = \sum_{A\subseteq X_t} \sum_{i=0}^{\abs{A}} \multari_t(A) q^i (-1)^{\abs{A} - i } \binom{\abs A}{ i }
   \\
  &= \sum_{i=0}^{N-d} q^i \underbrace{\sum_{\abs A \ge i} \multari_t(A) (-1)^{\abs A - i} \binom{\abs A}{ i }}_{\nu_i}.
\end{align}
 We need  to show that $\nu_i = \mu_i$  for all $i$.

By definition, for $A\subseteq X_t$,
$\multari_t(A) = \abs{ G_t / \sg{A} }$. Since $G_t/\sg{A}$ is  finite  and the dual of a finite abelian group is (non-canonically) isomorphic to itself, 
 $\abs{ G / \sg{A}} =  \abs{T( G/ \sg{A})}$. Furthermore, $T( G/ \sg{A}) =
 \{ g \in T(G_t)  : g( x) = 1 \text{ for all } x \in A \} =  \{ g \in T(G_t) : A \subseteq (X_t)_g \} $. Hence 
 \begin{equation}
  \multari_t(A) = \abs{\{ g \in T(G_t) : A \subseteq (X_t)_g \}}.
 \end{equation}
Let $n(A) :=  \abs{\{ g \in T(G_t) : A = (X_t)_g \}}$. 
Using the inclusion-exclusion principle we obtain $n(A)= \sum_{A\subseteq  C \subseteq X_t} (-1)^{\abs C - \abs A} \multari_t(C)$.
Hence 
\begin{align}
 \mu_i &= \sum_{\abs A = i } n(A) = \sum_{\abs A = i } \sum_{A\subseteq  C \subseteq X_t} (-1)^{\abs C - \abs A} \multari_t(C) \\
 &= \sum_{\abs{C} \ge i} \multari_t (C)  (-1)^{\abs C - i } \binom{\abs C}{i}  = \nu_i.
 \notag \qedhere
\end{align}
\end{proof}

\begin{proof}[Proof of Theorem~\ref{Proposition:PeriodicInternalTutteEval}]
 This follows by induction using Lemma~\ref{Lemma:PspacesMolecules} as a base case and
 Proposition~\ref{Proposition:DeletionContractionInternalPeriodic} and  \eqref{equation:AriTutteDelCon} for the induction step.
\end{proof}

\begin{proof}[Proof of Proposition~\ref{Proposition:InhomogeneousInternalBasis}]
Combine Lemma~\ref{Lemma:fzInternalIndependent}, Proposition~\ref{Proposition:ZonotopeArithmeticTutte}, and Theorem~\ref{Proposition:PeriodicInternalTutteEval}.
\end{proof}

%% file: SLPAM_examples.tex
%
%
%
%

\section{Examples}
\label{Section:Examples}

\subsection{Main examples}
In this subsection we will continue to study the Zwart--Powell element and we will also consider the list $X=(1,2,4)$.

 \begin{Example}[Zwart--Powell, continued]
 \label{Example:ZPelementB}
This is a continuation of Example~\ref{Example:ZPelementA}.

The toric arrangement in $(\R/\Z)^2$ 
defined by $X$ is shown in Figure~\ref{Figure:toricarr}. 
On the torus $T(\Z^2)$ it has two vertices, $1$ and $e_{\phi_1}(a,b)=(-1)^{a+b}$ . They correspond to the points $(0,0)$ and
$\phi_1:=(1/2, 1/2)$ in $\R^2/\Z^2$.

The continuous zonotopal spaces are
$\Pcal(X)=\R[s_1,s_2]_{ \le 2 }$,
$\Pcal_-(X)=\R[s_1,s_2]_{ \le 1 }$,
and $\Dcal(X)=\R[t_1, t_2]_{ \le 2 }$.

The discrete Dahmen--Micchelli space is $\DM(X) =
\spa\{ 1, t_1, t_2, t_1^2, t_1t_2, t_2^2, e_{\phi_1} \}$.
 The periodic $\Pcal$-spaces are  
\begin{align*} %
 \Pper(X)   &= \spa \{ 1, s_1, s_2,\; s_1^2, s_1s_2,  s_2^2, \; e_{\phi_1} s_1s_2    \},  
 \\
 \Pper_-(X) &= \spa \{ 1, s_1, s_2,  ( 1 -   e_{ {\phi_1} } )  s_1s_2  \}. 
\end{align*}
$\Bper(X)  =  \{ 1, s_2 , s_2 (s_1 +s_2 ), s_1 , s_1 (s_1 +s_2 ), s_1 s_2 ,   e_{\phi_1} s_1 s_2     \}$ is the homogeneous basis for $\Pper(X)$.
The differential equations for  $\Pper_-(X)$ are ($k\in \Z$):
\begin{equation*}
  D_{s_1}^2 p((0,k   ), \cdot) = D_{s_2}^2 p((k,0 ), \cdot) = D_{s_1+s_2}^2 p((k,-k   ), \cdot) = D_{s_1-s_2}^2 p((k,k   ), \cdot) = 0.
\end{equation*}
 The Tutte polynomial is $\tutte_X( \alpha, \beta)= \alpha^2 + \beta^2 + 2 \alpha + 2 \beta$ and
 the arithmetic Tutte polynomial is $\aritutte_X(\alpha,\beta) = \alpha^2 + \beta^2 + 2\alpha + 2\beta + 1$.
 Note that
   $q^2 \tutte_X(1, q^{-1})  = 1 + 2q + 3 q^2 = \hilb(\Pcal(X),q)$, 
   $q^2 \tutte_X(0,q^{-1})   = 1 + 2q  = \hilb(\Pcal_-(X), q)$,   
   $q^2 \aritutte_X(1, q^{-1}) = 1 + 2q + 4 q^2 = \hilb(\Pper(X),q)$,
   and %
$q^2 \aritutte_X(0,q^{-1}) = 1 + 2q + q^2 = \hilb(\Pper_-(X), q)$. 
The periodic Todd operator is
\begin{equation}
\begin{split}
\toddper(X,0) &= %
  \frac{s_1}{1-e^{-s_1}}
 \frac{s_2}{1-e^{-s_2}}
 \frac{s_1+s_2}{1-e^{-s_1-s_2}}
 \frac{-s_1 + s_2}{1-e^{s_1-s_2}}
 \\
 &\quad + e_{\phi_1}  \frac{s_1}{1+e^{-s_1}}
 \frac{s_2}{1+e^{-s_2}}
 \frac{s_1+s_2}{1-e^{-s_1-s_2}}
 \frac{-s_1 + s_2}{1-e^{s_1-s_2}}
\end{split}
\end{equation}
\begin{figure}[tbp]
  \begin{center}
    \subfigure[The toric arrangement corresponding to the Zwart--Powell Example in $\R^2/\Z^2$ \label{Figure:toricarr}]{
     \qquad\includegraphics[width=3cm]{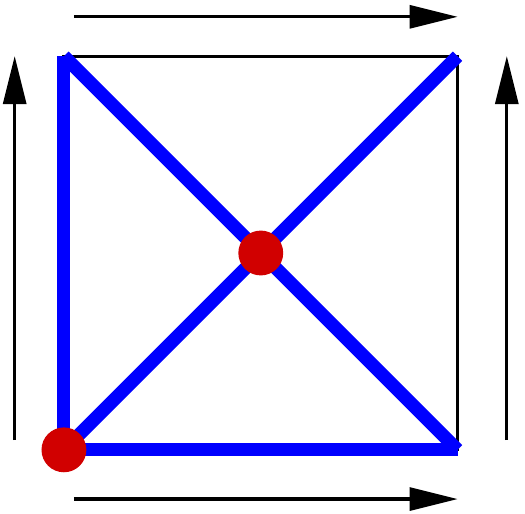} \qquad\qquad
     }
     \quad
    \subfigure[The box spline defined by the list $X = (1,2,4)$ \label{Figure:124boxspline}]{
      \;\includegraphics[width=6cm]{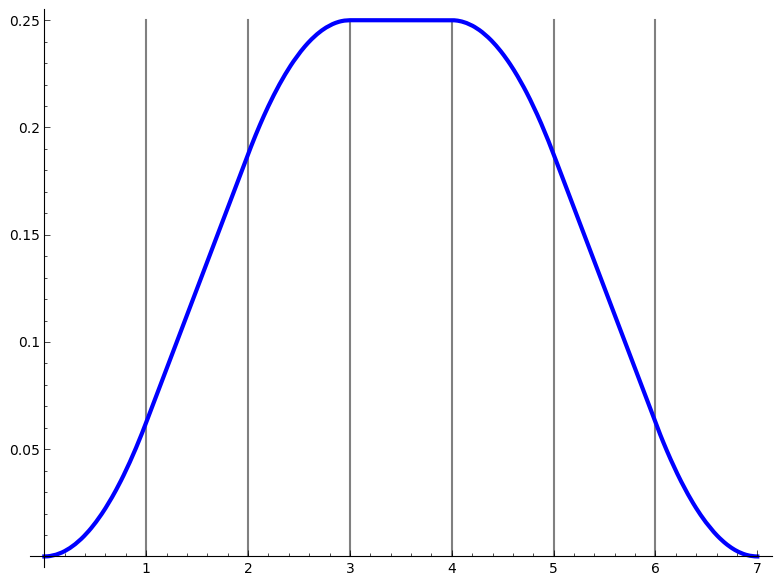}\;
     }
  \caption{A toric arrangement in $\R^2/\Z^2$ and a box spline }
  
  \end{center}
\end{figure} 
The projections of the periodic Todd operators are:
\begin{align*}
  \tilde f_{(0,0)} &= 1 + \frac 12 s_1 + \frac 32 s_2 +
  \frac{3}{4} s_1 s_2 + s_2^2  + \frac 14 e_\phi \\
\tilde f_{(0,1)} &=  1 + \frac 12 s_1 + \frac 12 s_2 + \frac 14 s_1 s_2   -  e_{\phi} \frac{s_1s_2}{4} 
    &
     \tilde f_{(1,1)} &=  1 - \frac 12 s_1 + \frac 12 s_2 - \frac 14 s_1 s_2  +  e_{\phi_1} \frac{s_1s_2}{4} 
    \\
     \tilde f_{(0,2)} &=  1 + \frac 12 s_1 - \frac 12  s_2  - \frac 14 s_1 s_2 +  e_{\phi_1} \frac{s_1s_2}{4} 
&
     \tilde f_{(1,2)} &=  1 - \frac 12 s_1 - \frac 12  s_2 + \frac 14 s_1 s_2 -  e_{\phi_1} \frac{s_1s_2}{4} 
  \end{align*}
\end{Example}

\begin{Example}[Zwart--Powell and the isomorphism {$L : \Pper(X) \to \CC[\Lambda]/ \dJC (X)$}]
 \label{Example:ZPpairing}
 In this example we use the algorithm described in Remark~\ref{Remark:PairingEval}
 to calculate the map $L: \Pper_\CC(X) \to \CC[\Lambda]/ \dJC(X)$. 
Recall that  $\CC[\Lambda] \cong \CC[ a_1^{\pm 1}, a_2^{\pm 1}]$
 and $\sym_\CC(U)\cong \CC[s_1,\ldots, s_d]$.
 \begin{asparaenum}[(a)]
  \item 
  The toric arrangement has two vertices: $\Vcal(X) = \{ (1,1), (-1,-1)  \} \subseteq (\CC^*)^2$.
 The primary decomposition of the discrete cocircuit ideal is
%
%
%
%
 \begin{align*}
 \dJC(X) &= 
\biggl( ( 1 - a_1 )( 1 - a_2 )( 1 - a_1 a_2 ),
  ( 1 - a_1 )( 1 - a_2 )( a_1 -  a_2 ),
  \\&\qquad
    ( 1 - a_1 )( 1 - a_1 a_2 )( a_1 -  {a_2} ),   ( 1 - a_2 )( 1 - a_1  a_2 )( a_1 -  a_2 ) \biggr)
  \\
  &= 
 \underbrace{\big(   (a_1-1)^3, (a_1-1)^2(a_2-1) , (a_1-1)(a_2-1)^2, (a_2-1)^3
 \big)}_{e_\phi=(1,1)}
 \cap  \underbrace{\big( a_1 + 1, a_2 + 1  \big)}_{e_\phi=(-1,-1)}. 
 \end{align*}
\item 
 To begin with, we  consider  the vertex $e_\phi=(-1,-1)$. We  choose the representative $ \theta (u_1,u_2) = \frac 12 ( u_1 + u_2 ) $.
Then $\Pcal_\CC(X_\phi)= \CC \cong \sym_\CC(U) / \cJC(X_\phi) \cong \CC[\Lambda] / \dJC(X)_\phi$.
So $i_{\log}^\theta \circ \tau_\theta \circ j_\phi$ maps $ 1\in \Pcal(X_\phi)$ to $ \bar 1\in  \CC[\Lambda] / \dJC(X)_\phi $.
 Now we consider the vertex $e_\phi=(1,1)$. We  choose the representative $\theta(u_1,u_2)=0$.
 Since $(a_1 - 1)^3$ and $(a_2-1)^3$ are contained in $\dJC(X)$,  we only have to develop the logarithm up to degree $2$.
 Hence 
 $i_{\log}^\theta ( \tau_\theta (j_\phi(s_1)) ) = \log(a_1) = a_1 - 1  - \frac{ (a_1 - 1)^2 }{2} 
   = -\frac{a_1^2}{2} + 2a_1  -  { \frac 32 } $.
  Similarly, $ i_{\log}^{\theta} ( \tau_\theta (j_\phi(s_2) )) =  - \frac{a_2^2}{2} + 2a_2  - \frac 32$.
 Hence, $i_{\log}^{\theta} \circ  \tau_\theta \circ j_\phi$ maps $\Pcal_\CC(X)$ to $\CC[\Lambda]/ \dJC(X)_{(1,1)}$ in the following way:  
 \begin{align*}
   1 &\mapsto 1 & 
   s_1^2 & \mapsto (a_1 - 1)^2
   \\
     s_1 & \mapsto     -\frac{a_1^2}{2} + 2a_1  -\frac 32 
  &
   s_1s_2 & \mapsto  (a_1 - 1)(a_2 - 1)
   \\
   s_2 &\mapsto - \frac{a_2^2}{2} + 2a_2  - \frac 32 &
   s_2^2 &\mapsto  (a_2 - 1)^2
  \end{align*}
\item
Now we have to find the embeddings $\kappa^\phi : \CC[\Lambda]/\dJC(X)_\phi \hookrightarrow  \CC[\Lambda]/\dJC(X)_\phi$.
Note that $  (a_2^2 - 4a_2+7) (a_2+1) - (a_2-1)^3 = 8 $.
Hence  $\kappa^{(1,1)}(1) = 1 + \frac 18( a_2 - 1)^3 =  \frac 18 a_2^3 - \frac 38 a_2^2 + \frac 38 a_2+ \frac 78  $
and
 $\kappa^{(-1,-1)}(1) = 1 - \frac 18(a_2^2-4a_2+7)(a_2+1) = -\frac{1}{8} a_2^{3} + \frac{3}{8} a_2^{2} - \frac{3}{8} a_2 + \frac{1}{8}$.

Hence the map  $L$ maps $\Pper_\CC(X)$ to $\CC[\Lambda]/ \dJC(X) $ in the following way: 
 \begin{align*}
   1 &\mapsto \frac 18 a_2^3 - \frac 38 a_2^2 + \frac 38 a_2+ \frac 78   & 
   s_1^2 & \mapsto  \frac 12 a_2^3 + a_1^2 - \frac 32 a_2^2 - 2 a_1 + \frac 32 a_2 + \frac 12
   \\
     s_1 & \mapsto      - \frac 12 a_2^3 - \frac 12 a_1^2 + \frac 32 a_2^2 + 2 a_1 - \frac 32 a_2 - 1
  &
   s_1s_2 & \mapsto  \frac 12 a_2^3 + a_1 a_2 - \frac 32 a_2^2 - a_1 + \frac 12 a_2 + \frac 12
   \\
   s_2 &\mapsto    - \frac 12 a_2^3 + a_2^2 + \frac 12 a_2 - 1
   &
   s_2^2 &\mapsto  \frac 12 a_2^3 - \frac 12 a_2^2 - \frac 12 a_2 + \frac 12 \\
   &&  \hspace*{-1.2cm}(-1)^{u_1+ u_2} s_1s_2 & \mapsto -\frac{1}{8} a_2^{3} + \frac{3}{8} a_2^{2} - \frac{3}{8} a_2 + \frac{1}{8} 
  \end{align*}  
One can easily check that the coefficients of the terms on the right-hand side always sum to $0$ except in the case of $L(1)$. This must hold
because of Theorem~\ref{Theorem:pairingisomorphism} and the fact that $ \discpairP{s_1}{1} = \discpairP{s_2}{1} = \ldots = 0$.
\end{asparaenum}
 \end{Example}

 \begin{Example}[The list $X=(1,2,4)$]
 \label{Example:OneTwoFour}
   Let $X = (1,2,4)$. Let $\xi_4$ denote the map that sends $k$ to $e^{\frac{\pi i }{2}k}$, \ie $\xi_4$ is a fourth root of unity. Then
   $\Pcal(X)= \spa \{ 1,s,s^2\}$ and $\Pper(X) = \spa\{ 1,s, s^2, \:\xi_4 s^2, \xi_4^3 s^2, \xi_4^2 s, \xi_4^2 s^2  \}$.
   The elements of the internal space must satisfy $D_s^2 f=0$ at the origin.
   Hence $ \Pper_-(X) = \{ 1,s, \xi_2 s,\:   (1 - \xi_4) s^2, ( \xi_4  - \xi_4^3 ) s^2,  ( \xi_4^2  - \xi_4^3 ) s^2 \}  $.
   
   Furthermore, $\Dcal(X) = \spa \{ 1,t, t^2\}$ and $\DM(X) = \spa \{ 1, t, t^2, \xi_4, \xi_4^2, \xi_4^2 t, \xi_4^3   \}$.
 The box spline is shown in Figure~\ref{Figure:124boxspline}.
Formulas for the splines and the vector partition function are: 
 \begin{align*}
      B_X(u) &=\begin{cases}
           \frac{1}{16} u^2 & 0 \le u \le 1
           \\
           \frac 18 u  - \frac{1}{16} & 1 \le u \le 2
           \\
           - \frac1{16} u^2 + \frac 38 u - \frac5{16} & 2 \le u \le 3
           \\
           \frac 14 & 3 \le u \le 4
           \\
           -\frac1{16} u^2 + \frac 12 u - \frac 34 & 4 \le u \le 5
           \\
           -\frac 18 u + \frac{13}{16} & 5 \le u \le 6
           \\
           \frac{1}{16} u^2 - \frac{7}{8}u  + \frac{49}{16} &     6 \le u \le 7        
          \end{cases}
          & \:
            \vpf_X(u) &= \begin{cases}
                 \frac{1}{16} u^2 + \frac 12 u + 1 & u \equiv 0  \mod 4  \\
                 \frac{1}{16} u^2 + \frac 38 u + \frac{9}{16} & u \equiv 1   \mod 4  \\
                 \frac{1}{16} u^2 + \frac 12 u  + \frac{12}{16} & u \equiv 2  \mod 4  \\
                 \frac{1}{16} u^2 + \frac 38 u + \frac{5}{16} & u \equiv 3  \mod 4  \\
                \end{cases}
   \end{align*}
   \begin{align*}
 T_X(u) &= \frac{1}{16} u^2 
 & \;
 \vpf_X(u) &=  \frac{1}{16} u^2 + \frac{7+ \xi_4^2}{16} u + \frac{21 + 7\xi_4^2}{32} + \frac{1}{16} \xi_4  (1 -   i) + 
 \frac{1}{16} \xi_4^3  ( 1 +   i)
 \end{align*}
 The projection of $\toddper(X,0)$ is
   \begin{align*}
     \tilde p_0 
%
      &=
     1 + \frac{7}{2} s + \frac{21}{4} s^2 + \xi_4  (\frac 12 - \frac 12 i) s^2 + \xi_4^3  (\frac 12 + \frac 12 i) s^2 
     + 
     \xi_4^2 (\frac 12s + \frac 74 s^2 ).
\displaybreak[2]
\\
     \toddperbox(X)  &=
     1 + \frac{7}{2} s + \frac{21}{4} s^2 + \xi_4  (\frac 12 - \frac 12 i) s^2 \frac{(1+ i \tau_1)(1+\tau_2)}{ (1- \tau_1)(1-\tau_2)} 
 \\ & \qquad
  + \xi_4^3  (\frac 12 + \frac 12 i) s^2 \frac{(1 - i \tau_1)(1+\tau_2)}{ (1- \tau_1)(1-\tau_2)}   
  +   \xi_2 (\frac 12s + \frac 74 s^2 )  \frac{(1 +  \tau_1)}{ (1- \tau_1) } %
   \end{align*}
is the operator defined in Remark~\ref{Remark:NoBoxGeneralisation}.

 The arithmetic Tutte polynomial is $\aritutte_X(\alpha,\beta) = (\alpha-1) + 7 + 4(\beta-1) + (\beta-1)^2 = \alpha + \beta^2 + 2\beta + 3$.
 Hence $ q^2\aritutte_X(0,q^{-1}) = 1 + 2q + 3q^2 = \hilb(\Pper_-(X),q)$ 
 and $q^2\aritutte_X(1,q^{-1}) = 1 + 2q + 4q^2 = \hilb(\Pper(X),q)$. 
\end{Example}
\begin{Example}[The list $X=(1,2,4)$ and the isomorphism $L$]
\label{Example:OneTwoFourIso}
The primary decomposition of the discrete cocircuit ideal is
 \begin{align}
  \dJC(X) &=
  \ideal \{ (1-a)(1-a^2)(1-a^4)   \} 
 \\ & = \ideal\{ (a-1)^3\} \cap  \ideal\{(a+1)^2 \} \cap \ideal \{ a  - i \} \cap \ideal \{  a+i \}.
  \end{align}
The toric arrangement has four vertices: $\Vcal(X)= \{ 1, i, -1, -i \} \subseteq \CC^*$.
They can be represented by $\theta= 0,\frac 14, \frac 12,\frac 34$.
One obtains that
$i_{\log}^0 ( \tau_0 (j_1 (s))) = -\frac{a^2}{2} + 2a - \frac 32$, 
$i_{\log}^0 ( \tau_0 ( j_1 (s^2))) = a^2 - 2a + 1$, and
$i_{\log}^{\frac 12} ( \tau_{\frac 12} ( j_{-1} ( s))) =  -a-1 $.
For $\theta\in \{\frac 14, \frac 34\}$, the spaces are trivial and $i_{\log}^\theta \circ \tau_\theta \circ j_\phi$
  just maps $1$ to $\bar 1$.
  
 Now if we lift these elements we obtain that the map $L$ maps  $\Pper_\CC(X)$ to $\CC[\Lambda] / \dJC(X)$ in the following way:  
\begin{align*}
    1 &\mapsto \frac{9}{32} a^{6} - \frac{1}{4} a^{5} - \frac{13}{32} a^{4} + \frac{1}{4} a^{3} - \frac{1}{32} a^{2} + \frac{1}{2} a + \frac{21}{32}
    \\
    s & \mapsto -\frac{5}{16} a^{6} + \frac{1}{8} a^{5} + \frac{7}{16} a^{4} + \frac{5}{16} a^{2} - \frac{1}{8} a - \frac{7}{16}
    \\
   s^2 & \mapsto
    \frac{1}{8} a^{6} - \frac{1}{8} a^{4} - \frac{1}{8} a^{2} + \frac{1}{8} 
    \displaybreak[2] \\
     \xi_4^2 s & \mapsto -\frac{5}{32} a^{6} + \frac{1}{4} a^{5} + \frac{1}{32} a^{4} - \frac{1}{4} a^{3} + \frac{13}{32} a^{2} - \frac{1}{2} a + \frac{7}{32}
    \\
     \xi_4^2 s^2 &\mapsto \frac1{16} a^6 - \frac 18 a^5 + \frac{1}{16} a^4 - \frac{1}{16} a^2 + \frac{1}{8} a - \frac{1}{16}
%
%
    \displaybreak[2] \\
     \xi_4 s^2 &\mapsto   
     \left(\frac{1}{16}  i  - \frac{1}{16}\right) a^{6} - \frac{1}{8}  i  a^{5} + \left(-\frac{1}{16}  i  + \frac{3}{16}\right) a^{4} + \frac{1}{4}  i  a^{3} + \left(-\frac{1}{16}  i  - \frac{3}{16}\right) a^{2} - \frac{1}{8}  i  a + \frac{1}{16}  i  + \frac{1}{16}
    \\
      \xi_4^3 s^2 &\mapsto \left(-\frac{1}{16}  i  - \frac{1}{16}\right) a^{6} + \frac{1}{8}  i  a^{5} + \left(\frac{1}{16}  i  + \frac{3}{16}\right) a^{4} - \frac{1}{4}  i  a^{3} +
     \left(\frac{1}{16}  i  - \frac{3}{16}\right) a^{2} + \frac{1}{8}  i  a - \frac{1}{16}  i  + \frac{1}{16}
\end{align*}
\end{Example}

\subsection{Examples involving torsion and deletion-contraction}

\begin{figure}[tbp]
  \begin{center}
   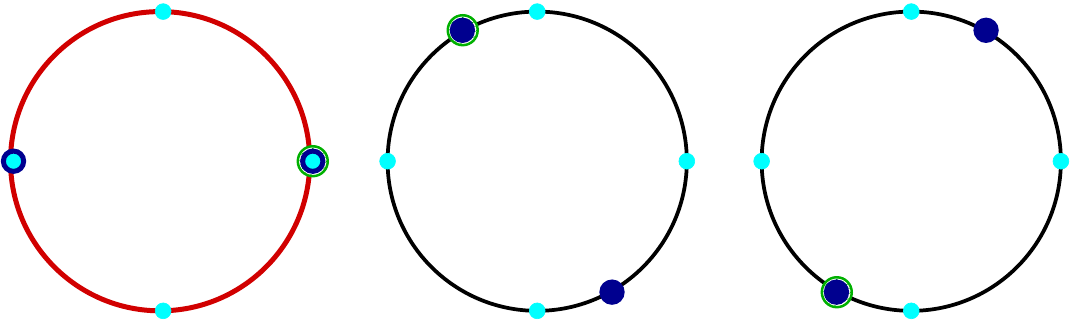
  \end{center}
  \caption{A toric arrangement in $T(\Z \oplus \Z/3\Z)\cong S^1 \times \Z/3\Z$.
  }
 \label{Figure:ToricArrOneD}
\end{figure}

 \begin{Example}[A toric arrangement on a disconnected torus]
\label{Example:toriarrwithtorsion}
\begin{align}
 \text{Let } X = \begin{pmatrix}
            4 & 2 &  1	 & 0 \\
            \bar 0 & \bar 1 & \bar 2 & \bar 1
          \end{pmatrix} = (x_1, x_2, x_3, x_4) \subseteq \Z\oplus \Z/3\Z.
\end{align}

 Note that $T( \Z\oplus \Z/3\Z ) \cong S^1 \times \{ g_3^0, g_3^1, g_3^2   \}$, where
 $(\alpha, g_3^k)$  maps  $(a,\bar b)$ to $\alpha^a \cdot e^{2\pi i\frac{k}{3} }$ for $\alpha\in S^1\subseteq \CC$, $k\in \{0,1,2\}$, $a\in \Z$ and $\bar b\in \Z/3\Z$.

 The corresponding toric arrangement is shown in Figure~\ref{Figure:ToricArrOneD}.
 $x_1$ defines the twelve (small) cyan vertices. $x_2$ defines the six (medium sized) blue vertices and $x_3$ defines the 
 three (large) green vertices.  Note that $\rank(x_4)=0$, hence $x_4$ does not define a vertex but a one-dimensional
 hypersurface, the leftmost  (red) copy of the $S^1$.
\end{Example}

\begin{Example} 
\label{Example:nontrivial2dtorarr}
 Let $X=\begin{pmatrix}
         2 & 4 & 0 & -1 \\ 0 & 1 & 2 & 1
        \end{pmatrix}$.
 Note that $\abs{\Vcal(X)}= 14$ (see Figure~\ref{Figure:nontrivial2dtorarr}) and 
 $\dim\Pper(X) = 23$.  Some of the differential equations that have to be satisfied by the elements of $\Pper_-(X)$
 are $D_{s_1}^2 p(0,\cdot) = D_{s_2}^2 p(0,\cdot) = D_{s_1}D_{s_2} p(0,\cdot) = D_{s_2}^2 p((1,0),\cdot) = D_{s_1}^2 p((0,1),\cdot)=0$.
 We leave it to the reader to calculate $\Pper(X)$ and $\Pper_-(X)$.

 Let $x=(2,0)$ be the first column.
 Then $X/x = ((\bar 0,1), (\bar 0, 2), (\bar 1, 1)) \subseteq \Z/2\Z \oplus \Z$ and
 $\Vcal(X/x) = \{ 1,  (-1)^b, (-1)^{\bar a}, (-1)^{\bar a + b}   \} $.
 The differential equations for the internal space are $D_{s_2}^2 p((\bar 0,0), \cdot) = D_{s_2}^2 p((\bar 1,0),\cdot)= 0$.
 Hence 
 \begin{align*}
   \Pper(X/x) &= \{ 1, \,s_2,\, s_2^2, (-1)^b s_2, \, (-1)^b s_2^2,\,   (-1)^{\bar a}s_2^2,\,  (-1)^{\bar a + b} s_2,\, (-1)^{\bar a + b} s_2^2 \}  \text{ and}  \\
   \Pper_-(X/x) &= \{1, \, s_2,\,   (-1)^{ b }s_2,\,    (-1)^{ \bar a + b } s_2,\,
        s_2^2 - (-1)^b s_2^2,\,  (-1)^{\bar a}s_2^2 -  (-1)^{ \bar a + b } s_2^2 \}.  
 \end{align*}
%
%
%
 In general, it is non-trivial to find preimages of elements of $\Pper_-(X/x)$ in $\Pper_-(X)$. 
 For example,  can you find an element of 
  $ \pi_x^{-1} (    s_2^{2} 
 + (-1)^{\bar a}  
  2   s_2^{2}  
  - (-1)^b  
      s_2^{2} 
 - (-1)^{\bar a+b}
    2   s_2^{2}) $? 
This may help you to do so:
%
%
 \begin{align}
  \tilde f_{(0,1)} &=
  \frac{9}{4}    s_1^{2} + \frac{9}{4}    s_1 s_2 + \frac{1}{4}    s_2^{2} + \frac{5}{2}   s_1 + s_2 + 1
  + (-1)^a (
  -\frac{5}{4}   s_1^{2} + \frac{3}{4}   s_1 s_2 + \frac{1}{2}   s_2^{2} - \frac{1}{2}   s_1 + \frac{1}{2}   s_2)
  \notag
  \\
 & \quad 
 - (-1)^b  
 (-s_1^{2} + \frac{3}{4}   s_1 s_2 + \frac{1}{4}   s_2^{2})
 - (-1)^{a+b}
 (5   s_1^{2} + \frac{13}{4}   s_1 s_2 + \frac{1}{2}   s_2^{2} + 2   s_1 + \frac{1}{2}   s_2)
+ \ldots 
\notag
\\
\label{eq:internalProjectionExample}
  \pi_x( \tilde f_{(0,1)} ) &=
  \frac{1}{4}   s_2^{2} + s_2 + 1
  + (-1)^{\bar a} (
  \frac{1}{2}   s_2^{2} +  \frac{1}{2}   s_2)
  - (-1)^b  
   \frac{1}{4}   s_2^{2} 
  - (-1)^{\bar a+b}
 (  \frac{1}{2}   s_2^{2} +  \frac{1}{2}   s_2)
 \end{align}

\begin{figure}[tbp]
  \begin{center}
   \begin{minipage}{8cm}
    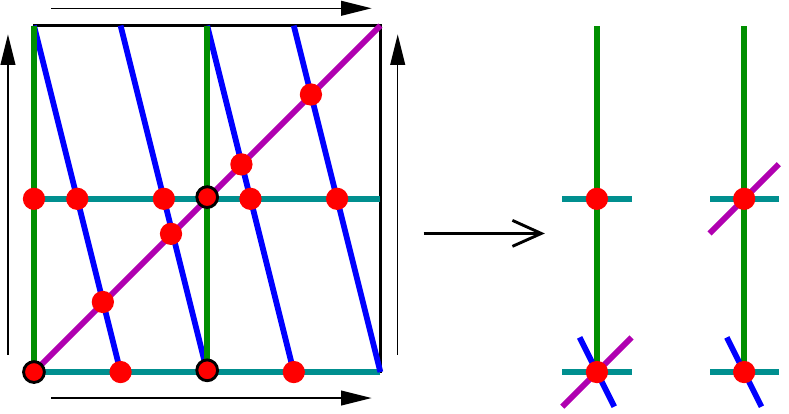
   \end{minipage}
   \begin{minipage}{4.2cm}    
   \includegraphics[width=4cm]{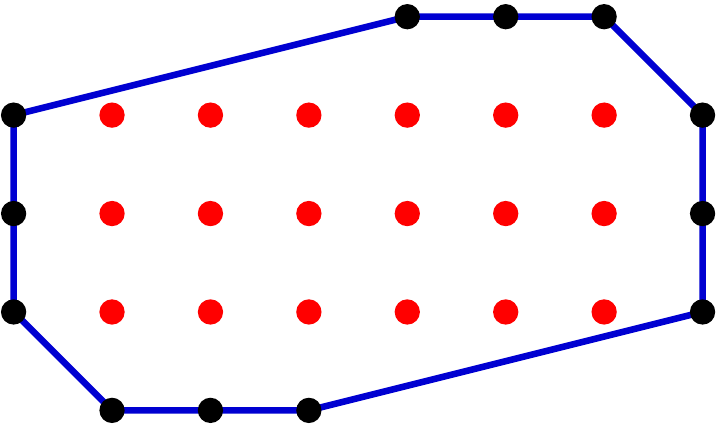}
   \end{minipage}

  \end{center}
  \caption{On the left: the toric arrangement in Example~\ref{Example:nontrivial2dtorarr} drawn in $\R^2/\Z^2$. The three vertices  
  that are circled have a non-trivial $\Pcal$-space attached to it.
  In the middle: the projection to $T(\Z\oplus \Z/2\Z)$.
  On the right: the zonotope.
  }
 \label{Figure:nontrivial2dtorarr}
\end{figure} 
       
\end{Example}

%% file: ToricArrOneD.pdf_t
\begin{picture}(0,0)%
\includegraphics{ToricArrOneD.pdf}%
\end{picture}%
\setlength{\unitlength}{1575sp}%
\begingroup\makeatletter\ifx\SetFigFont\undefined%
\gdef\SetFigFont#1#2#3#4#5{%
  \reset@font\fontsize{#1}{#2pt}%
  \fontfamily{#3}\fontseries{#4}\fontshape{#5}%
  \selectfont}%
\fi\endgroup%
\begin{picture}(12863,3806)(6591,-5564)
\put(11161,-5371){\makebox(0,0)[lb]{\smash{{\SetFigFont{9}{10.8}{\familydefault}{\mddefault}{\updefault}{\color[rgb]{0,0,0}$g_3^1$}%
}}}}
\put(6661,-5371){\makebox(0,0)[lb]{\smash{{\SetFigFont{9}{10.8}{\familydefault}{\mddefault}{\updefault}{\color[rgb]{0,0,0}$g_3^0$}%
}}}}
\put(15661,-5371){\makebox(0,0)[lb]{\smash{{\SetFigFont{9}{10.8}{\familydefault}{\mddefault}{\updefault}{\color[rgb]{0,0,0}$g_3^2$}%
}}}}
\put(12961,-2311){\makebox(0,0)[lb]{\smash{{\SetFigFont{8}{9.6}{\familydefault}{\mddefault}{\updefault}{\color[rgb]{0,0,0}$i$}%
}}}}
\put(17461,-2311){\makebox(0,0)[lb]{\smash{{\SetFigFont{8}{9.6}{\familydefault}{\mddefault}{\updefault}{\color[rgb]{0,0,0}$i$}%
}}}}
\put(8461,-2311){\makebox(0,0)[lb]{\smash{{\SetFigFont{8}{9.6}{\familydefault}{\mddefault}{\updefault}{\color[rgb]{0,0,0}$i$}%
}}}}
\put(9901,-3751){\makebox(0,0)[lb]{\smash{{\SetFigFont{8}{9.6}{\familydefault}{\mddefault}{\updefault}{\color[rgb]{0,0,0}$1$}%
}}}}
\put(14401,-3751){\makebox(0,0)[lb]{\smash{{\SetFigFont{8}{9.6}{\familydefault}{\mddefault}{\updefault}{\color[rgb]{0,0,0}$1$}%
}}}}
\put(18901,-3751){\makebox(0,0)[lb]{\smash{{\SetFigFont{8}{9.6}{\familydefault}{\mddefault}{\updefault}{\color[rgb]{0,0,0}$1$}%
}}}}
\end{picture}%

%% file: nontrivial2dtorarr.pdf_t
\begin{picture}(0,0)%
\includegraphics{nontrivial2dtorarr.pdf}%
\end{picture}%
\setlength{\unitlength}{3646sp}%
\begingroup\makeatletter\ifx\SetFigFont\undefined%
\gdef\SetFigFont#1#2#3#4#5{%
  \reset@font\fontsize{#1}{#2pt}%
  \fontfamily{#3}\fontseries{#4}\fontshape{#5}%
  \selectfont}%
\fi\endgroup%
\begin{picture}(4080,2145)(7924,-2974)
\put(10261,-1951){\makebox(0,0)[lb]{\smash{{\SetFigFont{11}{13.2}{\familydefault}{\mddefault}{\updefault}{\color[rgb]{0,0,0}$\pi_x$}%
}}}}
\end{picture}%

%% file: SLPAM_sageappendix.tex
\begin{appendix}
 \section{Commands for sage and Singular}
 \label{appendix:sagesingular}
 In this appendix we explain how Examples~\ref{Example:ZPpairing} and~\ref{Example:OneTwoFourIso} can be calculated using computer algebra programs.
 We use the algorithm described in Remark~\ref{Remark:PairingEval}.
 Most of the calculations can be done in Sage \cite{sage-62} which uses Singular  \cite{singular-316} for some of the calculations.
 
 Here is the code for the Zwart-Powell element (Example~\ref{Example:ZPpairing}):
\begin{verbatim} 
sage: K.<j> = QQ[I]
sage: R.<a,b> = K[]   # the polynomial ring in two variables over the field Q[i]
sage: J = ideal((1-a)*(1-b)*(1-a*b), (1-a)*(1-b)*(a-b), (1-a)*(a-b)*(1-a*b), 
(1-b)*(a-b)*(1-a*b)) # the discrete cocircuit ideal
sage: J.variety() # the points defined by the ideal
[{a: -1, b: -1}, {a: 1, b: 1}] 
sage: [J1,J2] = J.primary_decomposition() # the primary decomposition
sage: J1
Ideal (b^3 - 3*b^2 + 3*b - 1, a*b^2 - 2*a*b - b^2 + a + 2*b - 1, 
a^2*b - a^2 - 2*a*b + 2*a + b - 1, a^3 - 3*a^2 + 3*a - 1) of Multivariate Polynomial Ring
in a, b over Number Field in I with defining polynomial x^2 + 1
sage: J2
Ideal (b + 1, a + 1) of Multivariate Polynomial Ring 
in a, b over Number Field in I with defining polynomial x^2 + 1
sage: f1 = (a-1) - (a-1)**2/2  # the image of s1 under tau and the logarithmic isomorphism
sage: f2 = (b-1) - (b-1)**2/2  # the image of s2 under tau and the logarithmic isomorphism
sage: J1.reduce(f1*f1) # f1*f1 reduced modulo the ideal J1
a^2 - 2*a + 1
sage: J1.reduce(f1*f2)
a*b - a - b + 1
sage: J1.reduce(f2*f2)
b^2 - 2*b + 1
sage: g1 = 1+1/8*(a-1)**3 # the lifting of 1 in C[a,b]/J1 to C[a,b]/J
sage: J.reduce(g1)
1/8*b^3 - 3/8*b^2 + 3/8*b + 7/8
sage: J.reduce(g1*f1)
-1/2*b^3 - 1/2*a^2 + 3/2*b^2 + 2*a - 3/2*b - 1
sage: J.reduce(g1*f2)
-1/2*b^3 + b^2 + 1/2*b - 1
sage: J.reduce(g1*f1**2)
1/2*b^3 + a^2 - 3/2*b^2 - 2*a + 3/2*b + 1/2
sage: J.reduce(g1*f1*f2)
1/2*b^3 + a*b - 3/2*b^2 - a + 1/2*b + 1/2
sage: J.reduce(g1*f2**2)
1/2*b^3 - 1/2*b^2 - 1/2*b + 1/2
\end{verbatim}
 The equation  $  (a_2^2 - 4a_2+7) (a_2+1) - (a_2-1)^3 = 8 $ can be found using the {\tt liftstd} function of Singular  \cite{singular-316}:
\begin{verbatim}
> ring r = 0,(a,b),dp;
> ideal J1 = (1-a)**3, (1-a)**2*(1-b), (1-a)*(1-b)**2, (1-b)**3;
> ideal J2 = a + 1, b + 1;
> matrix T;
> def sm = liftstd(J1 + J2, T);
> sm;
sm[1]=8
> T;
T[1,1]=0
T[2,1]=0
T[3,1]=0
T[4,1]=1
T[5,1]=0
T[6,1]=b2-4b+7
> matrix(J1+J2)
_[1,1]=-a3+3a2-3a+1
_[1,2]=-a2b+a2+2ab-2a-b+1
_[1,3]=-ab2+2ab+b2-a-2b+1
_[1,4]=-b3+3b2-3b+1
_[1,5]=a+1
_[1,6]=b+1
> matrix(J1+J2)*T; // This gives us the equation above
_[1,1]=8
\end{verbatim}
Here is the code for the list   $X=(1,2,4)$ (Example~\ref{Example:OneTwoFourIso}):
\begin{verbatim}
sage: K.<j> = QQ[I]
sage: R = sage.rings.polynomial.multi_polynomial_libsingular.
MPolynomialRing_libsingular(K, 1, ('a',), TermOrder('degrevlex',1))
# we have to tell sage that we want to use Singular
# otherwise, primary decomposition is not available for polynomial rings in one variable
sage: R.inject_variables()
Defining a
sage: R
Multivariate Polynomial Ring in a over Number Field in I with defining polynomial x^2 + 1
sage: J= ideal( (1-a)*(1-a**2)*(1-a**4)) # the discrete cocircuit ideal
sage:  J.primary_decomposition()
[Ideal (a^3 - 3*a^2 + 3*a - 1) of Multivariate Polynomial Ring 
 in a, b over Number Field in I with defining polynomial x^2 + 1,
 Ideal (a^2 + 2*a + 1) of Multivariate Polynomial Ring 
 in a, b over Number Field in I with defining polynomial x^2 + 1,
 Ideal (a + (I)) of Multivariate Polynomial Ring 
 in a, b over Number Field in I with defining polynomial x^2 + 1,
 Ideal (a + (-I)) of Multivariate Polynomial Ring
 in a, b over Number Field in I with defining polynomial x^2 + 1]
sage: R.<a> = K[] # change the implementation of the ring, otherwise CRT_list does not work
sage: J  = ideal( (1-a)*(1-a**2)*(1-a**4)) # the discrete cocircuit ideal
sage: J1 = ideal( (a-1)**3 )  # the ideal corresponding to the vertex 1
sage: J2 = ideal( (a+1)**2 ) # the ideal corresponding to the vertex -1
sage: g1 = CRT_list( [ 1, 0, 0, 0], [ (a-1)**3, (a+1)**2, (a+j), (a-j) ] )
sage: g2 = CRT_list( [ 0, 1, 0, 0], [ (a-1)**3, (a+1)**2, (a+j), (a-j) ] )
sage: g3 = CRT_list( [ 0, 0, 1, 0], [ (a-1)**3, (a+1)**2, (a+j), (a-j) ] )
sage: g4 = CRT_list( [ 0, 0, 0, 1], [ (a-1)**3, (a+1)**2, (a+j), (a-j) ] )
sage: [g1,g2,g3,g4]
[9/32*a^6 - 1/4*a^5 - 13/32*a^4 + 1/4*a^3 - 1/32*a^2 + 1/2*a + 21/32,
 -5/32*a^6 + 1/4*a^5 + 1/32*a^4 - 1/4*a^3 + 13/32*a^2 - 1/2*a + 7/32,
 (-1/16*I - 1/16)*a^6 + 1/8*I*a^5 + (1/16*I + 3/16)*a^4 
 - 1/4*I*a^3 + (1/16*I - 3/16)*a^2 + 1/8*I*a - 1/16*I + 1/16,
 (1/16*I - 1/16)*a^6 - 1/8*I*a^5 + (-1/16*I + 3/16)*a^4 
 + 1/4*I*a^3 + (-1/16*I - 3/16)*a^2 - 1/8*I*a + 1/16*I + 1/16]
sage: f1 = -a**2/2 + 2*a - 3/2 # the image of s under the iota map for vertex 1
sage: f2 = -a -1
sage: f1**2
1/4*a^4 - 2*a^3 + 11/2*a^2 - 6*a + 9/4
sage: J1.reduce(f1**2)
a^2 - 2*a + 1
sage: J1.reduce(f1**3)
0
sage: J2.reduce(f2**2)
0
sage: [g1, J.reduce( g1 * f1), J.reduce( g1 * f1**2) ]    # generators corresponding 
# to the space at vertex 1
[9/32*a^6 - 1/4*a^5 - 13/32*a^4 + 1/4*a^3 - 1/32*a^2 + 1/2*a + 21/32,
 -5/16*a^6 + 1/8*a^5 + 7/16*a^4 + 5/16*a^2 - 1/8*a - 7/16,
 1/8*a^6 - 1/8*a^4 - 1/8*a^2 + 1/8]
sage: [g2, J.reduce(g2 * f2)] # generators corresponding to the space at vertex -1
[-5/32*a^6 + 1/4*a^5 + 1/32*a^4 - 1/4*a^3 + 13/32*a^2 - 1/2*a + 7/32,
 1/16*a^6 - 1/8*a^5 + 1/16*a^4 - 1/16*a^2 + 1/8*a - 1/16]
sage: J.reduce(g2**2) 
-5/32*a^6 + 1/4*a^5 + 1/32*a^4 - 1/4*a^3 + 13/32*a^2 - 1/2*a + 7/32 
# note that this is equal to g2
sage: g3 # generator corresponding to the space at vertex -i
(-1/16*I - 1/16)*a^6 + 1/8*I*a^5 + (1/16*I + 3/16)*a^4 - 1/4*I*a^3 
+ (1/16*I - 3/16)*a^2 + 1/8*I*a - 1/16*I + 1/16
sage: g4 # generator corresponding to the space at vertex i
(1/16*I - 1/16)*a^6 - 1/8*I*a^5 + (-1/16*I + 3/16)*a^4 + 1/4*I*a^3
 + (-1/16*I - 3/16)*a^2 - 1/8*I*a + 1/16*I + 1/16
sage: J.reduce(f2*g2)
1/16*a^6 - 1/8*a^5 + 1/16*a^4 - 1/16*a^2 + 1/8*a - 1/16 
sage: 1/16*6 + 1/8*5 + 1/16*4 - 1/16*2 - 1/8 
1
# < L( e_{-1} s s), e_{-1} t>_\nabla = s(D) t = 1
sage: [g1.substitute({a:1}), (J.reduce(g1*f1)).substitute({a:1}), 
 (J.reduce(g1*f1**2)).substitute({a:1}), g2.substitute({a:1}),  
 (J.reduce(g2*f2**2)).substitute({a:1}), g3.substitute({a:1}), g4.substitute({a:1})]
[1, 0, 0, 0, 0, 0, 0]
 # another check: < L( p ), 1 >_\nabla = 1 iff p = 1
\end{verbatim}

\end{appendix}

%% file: SLPAM_main.bbl
\providecommand{\bysame}{\leavevmode\hbox to3em{\hrulefill}\thinspace}
\providecommand{\MR}{\relax\ifhmode\unskip\space\fi MR }
\providecommand{\MRhref}[2]{%
  \href{http://www.ams.org/mathscinet-getitem?mr=#1}{#2}
}
\providecommand{\href}[2]{#2}
\begin{thebibliography}{10}

\bibitem{akopyan-saakyan-1988}
A.~A. Akopyan and A.~A. Saakyan, \emph{A system of differential equations that
  is related to the polynomial class of translates of a box spline}, Mat.
  Zametki \textbf{44} (1988), no.~6, 705--724, 861. \MR{983544 (90k:41016)}

\bibitem{ardila-postnikov-2009}
Federico Ardila and Alexander Postnikov, \emph{Combinatorics and geometry of
  power ideals}, Trans. Amer. Math. Soc. \textbf{362} (2010), no.~8,
  4357--4384. \MR{2608410}

\bibitem{ardila-postnikov-errata-2012}
\bysame, \emph{Two counterexamples for power ideals of hyperplane
  arrangements}, 2012, \url{arXiv:1211.1368}, to appear in Trans. Amer. Math.
  Soc. as a correction to \cite{ardila-postnikov-2009}.

\bibitem{beck-robins-ComputingTheContinuousDiscretely}
Matthias Beck and Sinai Robins, \emph{Computing the continuous discretely.
  integer-point enumeration in polyhedra.}, Undergraduate Texts in Mathematics,
  Springer, New York, 2007. \MR{2271992 (2007h:11119)}

\bibitem{benartzi-ron-1988}
Asher Ben-Artzi and Amos Ron, \emph{Translates of exponential box splines and
  their related spaces}, Trans. Amer. Math. Soc. \textbf{309} (1988), no.~2,
  683--710. \MR{961608 (89m:41008)}

\bibitem{berget-2010}
Andrew Berget, \emph{Products of linear forms and {T}utte polynomials},
  European J. Combin. \textbf{31} (2010), no.~7, 1924--1935. \MR{2673030
  (2011j:05062)}

\bibitem{boysal-vergne-2009}
Arzu Boysal and Mich{\`e}le Vergne, \emph{Paradan's wall crossing formula for
  partition functions and {K}hovanski-{P}ukhlikov differential operator}, Ann.
  Inst. Fourier (Grenoble) \textbf{59} (2009), no.~5, 1715--1752. \MR{2573189
  (2010j:52043)}

\bibitem{branden-moci-2014}
Petter Br{\"a}nd{\'e}n and Luca Moci, \emph{The multivariate arithmetic {T}utte
  polynomial}, Trans. Amer. Math. Soc. \textbf{366} (2014), no.~10, 5523--5540.
  \MR{3240933}

\bibitem{brion-vergne-1997}
Michel Brion and Mich{\`e}le Vergne, \emph{Residue formulae, vector partition
  functions and lattice points in rational polytopes}, J. Amer. Math. Soc.
  \textbf{10} (1997), no.~4, 797--833. \MR{1446364 (98e:52008)}

\bibitem{brion-vergne-1999}
\bysame, \emph{Arrangement of hyperplanes. {I}. {R}ational functions and
  {J}effrey-{K}irwan residue}, Ann. Sci. \'Ecole Norm. Sup. (4) \textbf{32}
  (1999), no.~5, 715--741. \MR{1710758 (2001e:32039)}

\bibitem{cavazzani-moci-2013}
Francesco Cavazzani and Luca Moci, \emph{Geometric realizations and duality for
  {D}ahmen-{M}icchelli modules and {D}e {C}oncini-{P}rocesi-{V}ergne modules},
  2012, \url{arXiv:1303.0902}.

\bibitem{cox-little-oshea-2005}
David~A. Cox, John Little, and Donal O'Shea, \emph{Using algebraic geometry},
  second ed., Graduate Texts in Mathematics, vol. 185, Springer, New York,
  2005. \MR{2122859 (2005i:13037)}

\bibitem{cox-little-schenck-2011}
David~A. Cox, John~B. Little, and Henry~K. Schenck, \emph{Toric varieties},
  Graduate Studies in Mathematics, vol. 124, American Mathematical Society,
  Providence, RI, 2011. \MR{2810322 (2012g:14094)}

\bibitem{moci-adderio-ehrhart-2012}
Michele D'Adderio and Luca Moci, \emph{Ehrhart polynomial and arithmetic
  {T}utte polynomial}, European J. Combin. \textbf{33} (2012), no.~7, 1479 --
  1483.

\bibitem{moci-adderio-2013}
Michele D'Adderio and Luca Moci, \emph{Arithmetic matroids, the {T}utte
  polynomial and toric arrangements}, Adv. Math. \textbf{232} (2013), no.~1,
  335--367.

\bibitem{dahmen-micchelli-1985}
Wolfgang Dahmen and Charles~A. Micchelli, \emph{On the local linear
  independence of translates of a box spline}, Studia Math. \textbf{82} (1985),
  no.~3, 243--263. \MR{825481 (87k:41008)}

\bibitem{dahmen-micchelli-1985b}
\bysame, \emph{On the solution of certain systems of partial difference
  equations and linear dependence of translates of box splines}, Trans. Amer.
  Math. Soc. \textbf{292} (1985), no.~1, 305--320. \MR{805964 (86k:41014)}

\bibitem{boor-dyn-ron-1991}
Carl de~Boor, Nira Dyn, and Amos Ron, \emph{On two polynomial spaces associated
  with a box spline}, Pacific J. Math. \textbf{147} (1991), no.~2, 249--267.
  \MR{1084708 (92d:41018)}

\bibitem{boor-hoellig-1982}
Carl de~Boor and Klaus H{\"o}llig, \emph{{$B$}-splines from parallelepipeds},
  J. Analyse Math. \textbf{42} (1982/83), 99--115. \MR{729403 (86d:41008)}

\bibitem{BoxSplineBook}
Carl de~Boor, Klaus H{\"o}llig, and Sherman~D. Riemenschneider, \emph{Box
  splines}, Applied Mathematical Sciences, vol.~98, Springer-Verlag, New York,
  1993. \MR{1243635 (94k:65004)}

\bibitem{concini-procesi-book}
Corrado De~Concini and Claudio Procesi, \emph{Topics in hyperplane
  arrangements, polytopes and box-splines}, Universitext, Springer, New York,
  2011. \MR{2722776}

\bibitem{deconcini-procesi-vergne-2010b}
Corrado De~Concini, Claudio Procesi, and Mich\`ele Vergne, \emph{Vector
  partition functions and index of transversally elliptic operators},
  Transform. Groups \textbf{15} (2010), no.~4, 775--811. \MR{2753257
  (2012a:58034)}

\bibitem{deconcini-procesi-vergne-2011}
\bysame, \emph{Infinitesimal index: cohomology computations}, Transformation
  Groups \textbf{16} (2011), no.~3, 717--735 (English).

\bibitem{deconcini-procesi-vergne-2013}
\bysame, \emph{Box splines and the equivariant index theorem}, J. Inst. Math.
  Jussieu \textbf{12} (2013), 503--544.

\bibitem{deconcini-procesi-vergne-infinitesimal-2013}
\bysame, \emph{The infinitesimal index}, J. Inst. Math. Jussieu \textbf{12}
  (2013), no.~2, 297--334. \MR{3028788}

\bibitem{deloera-2005}
Jes{\'u}s~A. De~Loera, \emph{The many aspects of counting lattice points in
  polytopes}, Math. Semesterber. \textbf{52} (2005), no.~2, 175--195.
  \MR{2159956 (2006c:52015)}

\bibitem{singular-316}
Wolfram Decker, Gert-Martin Greuel, Gerhard Pfister, and Hans Sch\"onemann,
  \emph{{\sc Singular} {3-1-6} --- {A} computer algebra system for polynomial
  computations}, \url{http://www.singular.uni-kl.de}, 2012.

\bibitem{dyn-ron-1990}
Nira Dyn and Amos Ron, \emph{Local approximation by certain spaces of
  exponential polynomials, approximation order of exponential box splines, and
  related interpolation problems}, Trans. Amer. Math. Soc. \textbf{319} (1990),
  no.~1, 381--403. \MR{956032 (90i:41020)}

\bibitem{eisenbud-1995}
David Eisenbud, \emph{Commutative algebra with a view toward algebraic
  geometry}, Graduate Texts in Mathematics, vol. 150, Springer-Verlag, New
  York, 1995. \MR{1322960 (97a:13001)}

\bibitem{hirzebruch-1956}
Friedrich Hirzebruch, \emph{Neue topologische {M}ethoden in der algebraischen
  {G}eometrie}, Ergebnisse der Mathematik und ihrer Grenzgebiete (N.F.), Heft
  9, Springer-Verlag, Berlin, 1956. \MR{0082174 (18,509b)}

\bibitem{holtz-ron-2011}
Olga Holtz and Amos Ron, \emph{Zonotopal algebra}, Adv. Math. \textbf{227}
  (2011), no.~2, 847--894. \MR{2793025}

\bibitem{holtz-ron-xu-2012}
Olga Holtz, Amos Ron, and Zhiqiang Xu, \emph{Hierarchical zonotopal spaces},
  Trans. Amer. Math. Soc. \textbf{364} (2012), no.~2, 745--766. \MR{2846351}

\bibitem{jia-1990}
Rong-Qing Jia, \emph{{Subspaces invariant under translations and dual bases for
  box splines.}}, Chin. Ann. Math., Ser. A \textbf{11} (1990), no.~6, 733--743
  (Chinese).

\bibitem{pukhlikov-khovanski-1992}
Askold Khovanski{\u\i} and Aleksandr Pukhlikov, \emph{The {R}iemann-{R}och
  theorem for integrals and sums of quasipolynomials on virtual polytopes},
  Algebra i Analiz \textbf{4} (1992), no.~4, 188--216. \MR{1190788 (94c:14044)}

\bibitem{lang-algebra-2002}
Serge Lang, \emph{Algebra}, third ed., Graduate Texts in Mathematics, vol. 211,
  Springer-Verlag, New York, 2002. \MR{1878556 (2003e:00003)}

\bibitem{lenz-hzpi-2012}
Matthias Lenz, \emph{Hierarchical zonotopal power ideals}, European J. Combin.
  \textbf{33} (2012), no.~6, 1120--1141.

\bibitem{lenz-fpsac-2014}
\bysame, \emph{Splines, lattice points, and (arithmetic) matroids}, Proceedings
  of 26th International Conference on Formal Power Series and Algebraic
  Combinatorics (FPSAC 2014), DMTCS Proceedings, Assoc. Discrete Math. Theor.
  Comput. Sci., Nancy, France, 2014, pp.~49--60.

\bibitem{lenz-interpolation-online-2013}
\bysame, \emph{Interpolation, box splines, and lattice points in zonotopes},
  International Mathematics Research Notices (first published online July 10,
  2013), 16 pages.

\bibitem{lenz-todd-online-2014}
\bysame, \emph{Lattice points in polytopes, box splines, and {T}odd operators},
  International Mathematics Research Notices (first published online June 12,
  2014), 22 pages.

\bibitem{moci-tutte-2012}
Luca Moci, \emph{A {T}utte polynomial for toric arrangements}, Trans. Amer.
  Math. Soc. \textbf{364} (2012), no.~2, 1067--1088. \MR{2846363}

\bibitem{orlik-terao-1994}
Peter Orlik and Hiroaki Terao, \emph{Commutative algebras for arrangements},
  Nagoya Math. J. \textbf{134} (1994), 65--73. \MR{1280653 (95j:52025)}

\bibitem{MatroidTheory-Oxley}
James~G. Oxley, \emph{Matroid theory}, Oxford Science Publications, The
  Clarendon Press Oxford University Press, New York, 1992. \MR{1207587
  (94d:05033)}

\bibitem{stanley-2007}
Richard~P. Stanley, \emph{An introduction to hyperplane arrangements},
  Geometric combinatorics, IAS/Park City Math. Ser., vol.~13, Amer. Math. Soc.,
  Providence, RI, 2007, pp.~389--496. \MR{2383131}

\bibitem{sage-62}
W.\thinspace{}A. Stein et~al., \emph{{S}age {M}athematics {S}oftware ({V}ersion
  6.2)}, The Sage Development Team, 2014, {\tt http://www.sagemath.org}.

\bibitem{szenes-vergne-2003}
Andr{\'a}s Szenes and Mich{\`e}le Vergne, \emph{Residue formulae for vector
  partitions and {E}uler-{M}ac{L}aurin sums}, Adv. in Appl. Math. \textbf{30}
  (2003), no.~1-2, 295--342, Formal power series and algebraic combinatorics
  (Scottsdale, AZ, 2001). \MR{1979797 (2005j:52018)}

\bibitem{vergne-2003}
Mich{\`e}le Vergne, \emph{Residue formulae for {V}erlinde sums, and for number
  of integral points in convex rational polytopes}, European women in
  mathematics ({M}alta, 2001), World Sci. Publ., River Edge, NJ, 2003,
  pp.~225--285. \MR{2012207 (2004k:11158)}

\bibitem{wang-2000}
Ren-Hong Wang, \emph{Multivariate spline and algebraic geometry}, J. Comput.
  Appl. Math. \textbf{121} (2000), no.~1-2, 153--163, Numerical analysis in the
  20th century, Vol. I, Approximation theory. \MR{1780047}

\end{thebibliography}
